\documentclass[letterpaper]{amsart}
\usepackage{amssymb}
\usepackage{amsmath}
\usepackage{amsfonts}

\newtheorem{theorem}{Theorem}[section]
\theoremstyle{plain}

\newtheorem{definition}{Definition}[section]
\newtheorem{example}{Example}

\newtheorem{lemma}{Lemma}[section]

\newtheorem{proposition}{Proposition}[section]
\newtheorem{remark}{Remark}[section]

\numberwithin{equation}{section}
\input{tcilatex}

\begin{document}
\title[Unsteady Navier-Stokes type equations]{Deterministic Homogenization
of Unsteady Navier-Stokes type equations }
\author{Lazarus Signing}
\address{University of Ngaoundere, Department of Mathematics and Computer
Science, P.O.Box \ 454 Ngaoundere (Cameroon)}
\email{lsigning@uy1.uninet.cm}
\urladdr{}
\thanks{}
\urladdr{}
\date{}
\dedicatory{}
\thanks{}

\begin{abstract}
In this paper we study the deterministic homogenization problems for
unsteady Navier-Stokes type equations, on one hand in an open set $\Omega $
of $\mathbb{R}^{N}$, on the other hand in porous media $\Omega ^{\varepsilon
}$. In the second case, the equations are classical unsteady Navier-Stokes
one, and the porous media are periodic. \ 
\end{abstract}

\maketitle

\section{Introduction}

We study the homogenization of unsteady Navier-Stokes type equations in two
distinct settings. In the first setting, the equations are considered in a
fixed bounded open set in the $N$-dimensional numerical space and moreover
the usual Laplace operator involved in the classical Navier-Stokes equations
is replaced by an elliptic linear differential operator of order two, in
divergence form, with spatially varying coefficients. In the second setting,
the equations are the classical unsteady Navier-Stokes one and are
considered in periodic porous media. Precisely, we investigate the following
problems:

\textbf{Problem I.} Let $\Omega $ be a smooth bounded open set of $\mathbb{R}%
_{x}^{N}$ (the $N$-dimensional numerical space of variable $x=\left(
x_{1},...,x_{N}\right) $), and let $T$ and $\varepsilon $ be real numbers
with $T>0$ and $0<\varepsilon <1$. We consider the partial differential
operator 
\begin{equation*}
P^{\varepsilon }=-\sum_{i,j=1}^{N}\frac{\partial }{\partial x_{i}}\left(
a_{ij}^{\varepsilon }\frac{\partial }{\partial x_{j}}\right)
\end{equation*}%
in $\Omega \times ]0,T[$, where $a_{ij}^{\varepsilon }\left( x\right)
=a_{ij}\left( \frac{x}{\varepsilon }\right) $\quad ($x\in \Omega $), $%
a_{ij}\in L^{\infty }\left( \mathbb{R}_{y}^{N};\mathbb{R}\right) $ ($1\leq
i,j\leq N$) with 
\begin{equation}
a_{ij}=a_{ji}\text{,}  \label{eq1.1}
\end{equation}%
and the assumption that there is a constant $\alpha >0$ such that 
\begin{equation}
\sum_{i,j=1}^{N}a_{ij}\left( y\right) \zeta _{j}\zeta _{i}\geq \alpha
\left\vert \zeta \right\vert ^{2}\text{\ for all }\zeta =\left( \zeta
_{j}\right) \in \mathbb{R}^{N}\text{ and}  \label{eq1.2}
\end{equation}%
for almost all $y\in \mathbb{R}^{N}$, $\mathbb{R}_{y}^{N}$ being the $N$%
-demensional numerical space $\mathbb{R}^{N}$ of variables $y=\left(
y_{1},...,y_{N}\right) $ and $\left\vert {\small \cdot }\right\vert $
denoting the Euclidean norm in $\mathbb{R}^{N}$. The operator $%
P^{\varepsilon }$ acts on scalar functions, say $\varphi \in L^{2}\left(
0,T;H^{1}\left( \Omega \right) \right) $. However, we may as well view $%
P^{\varepsilon }$ as acting on vector functions $\mathbf{u}=\left(
u^{i}\right) \in L^{2}\left( 0,T;H^{1}\left( \Omega \right) ^{N}\right) $ in
a \textit{diagonal way}, i.e.,%
\begin{equation*}
\left( P^{\varepsilon }\mathbf{u}\right) ^{i}=P^{\varepsilon }u^{i}\text{%
\qquad }\left( i=1,...,N\right) \text{.}
\end{equation*}%
For any Roman character such as $i$, $j$ (with $1\leq i,j\leq N$), $u^{i}$
(resp. $u^{j}$) denotes the $i$-th (resp. $j$-th) component of a vector
function $\mathbf{u}$ in $L_{loc}^{1}\left( \Omega \times ]0,T[\right) ^{N}$
or in $L_{loc}^{1}\left( \mathbb{R}_{y}^{N}\times \mathbb{R}_{\tau }\right)
^{N}$ where $\mathbb{R}_{\tau }$ is the numerical space $\mathbb{R}$ of
variables $\tau $. Further, for any real $0<\varepsilon <1$, we define $%
u^{\varepsilon }$ as 
\begin{equation*}
u^{\varepsilon }\left( x,t\right) =u\left( \frac{x}{\varepsilon },\frac{t}{%
\varepsilon }\right) \text{\qquad }\left( \left( x,t\right) \in \Omega
\times ]0,T[\right)
\end{equation*}%
for $u\in L_{loc}^{1}\left( \mathbb{R}_{y}^{N}\times \mathbb{R}_{\tau
}\right) $, as is customary in homogenization theory. More generally, for $%
u\in L_{loc}^{1}\left( Q\times \mathbb{R}_{y}^{N}\times \mathbb{R}_{\tau
}\right) $ with $Q=\Omega \times ]0,T[$, it is customary to put%
\begin{equation*}
u^{\varepsilon }\left( x,t\right) =u\left( x,t,\frac{x}{\varepsilon },\frac{t%
}{\varepsilon }\right) \text{\qquad }\left( \left( x,t\right) \in \Omega
\times ]0,T[\right)
\end{equation*}%
whenever the right-hand side makes sense (see, e.g., \cite{bib15}).

After these preliminaries, let $\mathbf{f}=\left( f^{i}\right) \in
L^{2}\left( 0,T;H^{-1}\left( \Omega ;\mathbb{R}\right) ^{N}\right) $. For
any fixed $0<\varepsilon <1$, we consider the initial boundary value problem%
\begin{equation}
\frac{\partial \mathbf{u}_{\varepsilon }}{\partial t}+P^{\varepsilon }%
\mathbf{u}_{\varepsilon }+\sum_{j=1}^{N}u_{\varepsilon }^{j}\frac{\partial 
\mathbf{u}_{\varepsilon }}{\partial x_{j}}+\mathbf{grad}p_{\varepsilon }=%
\mathbf{f}\text{ in }\Omega \times ]0,T[\text{,}  \label{eq1.3}
\end{equation}%
\begin{equation}
div\mathbf{u}_{\varepsilon }=0\text{ in }\Omega \times ]0,T[\text{,}
\label{eq1.4}
\end{equation}%
\begin{equation}
\mathbf{u}_{\varepsilon }=0\text{ on }\partial \Omega \times ]0,T[\text{,}
\label{eq1.5}
\end{equation}%
\begin{equation}
\mathbf{u}_{\varepsilon }\left( 0\right) =0\text{ in }\Omega  \label{eq1.6}
\end{equation}%
where $\frac{\partial \mathbf{u}_{\varepsilon }}{\partial x_{j}}=\left( 
\frac{\partial u^{1}}{\partial x_{j}},...,\frac{\partial u^{N}}{\partial
x_{j}}\right) $. As in \cite{bib31}, for $N=2$ (\ref{eq1.3})-(\ref{eq1.6})
uniquely define $\left( \mathbf{u}_{\varepsilon },p_{\varepsilon }\right) $
with $\mathbf{u}_{\varepsilon }\in \mathcal{W}\left( 0,T\right) $ and $%
p_{\varepsilon }\in L^{2}\left( 0,T;L^{2}\left( \Omega ;\mathbb{R}\right) 
\mathfrak{/}\mathbb{R}\right) $, where 
\begin{equation*}
\mathcal{W}\left( 0,T\right) =\left\{ \mathbf{u}\in L^{2}\left( 0,T;V\right)
:\mathbf{u}^{\prime }\in L^{2}\left( 0,T;V^{\prime }\right) \right\}
\end{equation*}%
$V$ being the space of functions $\mathbf{u}$ in $H_{0}^{1}\left( \Omega ;%
\mathbb{R}\right) ^{N}$ with $div\mathbf{u}=0$ ($V^{\prime }$ is the
topological dual of $V$), and where 
\begin{equation*}
L^{2}\left( \Omega ;\mathbb{R}\right) \mathfrak{/}\mathbb{R}=\left\{ v\in
L^{2}\left( \Omega ;\mathbb{R}\right) :\int_{\Omega }vdx=0\right\} \text{.}
\end{equation*}%
Let us recall that $\mathcal{W}\left( 0,T\right) $ is provided with the norm%
\begin{equation*}
\left\Vert \mathbf{u}\right\Vert _{\mathcal{W}\left( 0,T\right) }=\left(
\left\Vert \mathbf{u}\right\Vert _{L^{2}\left( 0,T;V\right) }^{2}+\left\Vert 
\mathbf{u}^{\prime }\right\Vert _{L^{2}\left( 0,T;V^{\prime }\right)
}^{2}\right) ^{\frac{1}{2}}\text{\qquad }\left( \mathbf{u}\in \mathcal{W}%
\left( 0,T\right) \right) \text{,}
\end{equation*}%
which makes it a Hilbert space with the following properties (see \cite%
{bib31}): $\mathcal{W}\left( 0,T\right) $ is continuously embedded in $%
\mathcal{C}\left( \left[ 0,T\right] ;L^{2}\left( \Omega \right) ^{N}\right) $
and is compactly embedded in $L^{2}\left( 0,T;L^{2}\left( \Omega \right)
^{N}\right) $.

Our aim here is to investigate the asymptotic behavior, as $\varepsilon
\rightarrow 0$, of $\left( \mathbf{u}_{\varepsilon },p_{\varepsilon }\right) 
$ under an abstract assumption on $a_{ij}$ $\left( 1\leq i,j\leq N\right) $
covering a wide range of concrete behaviours beyond the classical
periodicity hypothesis. The latter states that $a_{ij}\left( y+k\right)
=a_{ij}\left( y\right) $ for almost all $y\in \mathbb{R}^{N}$ and for all $%
k\in \mathbb{Z}^{N}$($\mathbb{Z}$ denotes the integers). The study of this
problem turns out to be of benefit to the modelling of heterogeneous fluid
flows, in particular multi-phase flows, fluids with spatially varying
viscosities, and others. The linear version of this problem (i.e., the
homogenization of (\ref{eq1.3})-(\ref{eq1.6}) without the term $%
\sum_{j=1}^{N}u_{\varepsilon }^{j}\frac{\partial \mathbf{u}_{\varepsilon }}{%
\partial x_{j}}$) has been studied by the author \cite{bib27} under the
periodicity hypothesis on the coefficients $a_{ij}$ via the two-scale
convergence techniques. We mention also the paper by Choe and Kim \cite{bib5}
dealing with that linear version by the well known asymptotic expansion
combined with Tartar's energy method. Further, the steady version was first
investigated in \cite{bib22} by the \textit{sigma-convergence} method. This
paper deals with a more complicated situation where the equations are
non-stationary and non-linear, and the estimates of the pressure and the
acceleration become a laborious issue as it is shown in the proof of
Proposition \ref{pr2.1}. As far as i know, this topic has not yet been
seriously investigated.

The main result of this part of the work can be stated as follows: Let $%
\left( \mathbf{u}_{\varepsilon },p_{\varepsilon }\right) \in \mathcal{W}%
\left( 0,T\right) \times L^{2}\left( 0,T;L^{2}\left( \Omega ;\mathbb{R}%
\right) \mathfrak{/}\mathbb{R}\right) $ be the unique solution of (\ref%
{eq1.3})-(\ref{eq1.6}). As $\varepsilon $ goes to zero, $\left( \mathbf{u}%
_{\varepsilon },p_{\varepsilon }\right) $ converges in some topology to some

$\left( \mathbf{u}_{0},p_{0}\right) \in \mathcal{W}\left( 0,T\right) \times
L^{2}\left( 0,T;L^{2}\left( \Omega ;\mathbb{R}\right) \mathfrak{/}\mathbb{R}%
\right) $, where $\left( \mathbf{u}_{0},p_{0}\right) $ is the unique
solution of the initial boundary value problem (\ref{eq2.53})-(\ref{eq2.56}%
). The macroscopic homogenized equations (\ref{eq2.53})-(\ref{eq2.56}) is of
the incompressible Navier-Stokes type. This result is proved in the periodic
setting by Theorem \ref{th2.3} and Lemma \ref{lem2.5}, and in general
deterministic setting by Theorem \ref{th3.1} and Lemma \ref{lem3.2}.

Our approach is the \textit{sigma-convergence} method derived from two-scale
convergence ideas \cite{bib1}, \cite{bib14} by means of so-called
homogenization algebras \cite{bib17}, \cite{bib18}.

\textbf{Problem II.} Let us put 
\begin{equation*}
Y=\left( -\frac{1}{2},\frac{1}{2}\right) ^{N}
\end{equation*}%
with $N\geq 2$, $Y$ being viewed as a subset of $\mathbb{R}_{y}^{N}$ (the
space $\mathbb{R}^{N}$ of variables $y=\left( y_{1},...,y_{N}\right) $). Let 
$Y_{s}$ be a connected open set in $\mathbb{R}^{N}$ with $\overline{Y}%
_{s}\subset Y$ ($\overline{Y}_{s}$ the closure of $Y_{s}$ in $\mathbb{R}^{N}$%
) and with smooth boundary $\Gamma =\partial Y_{s}$. We put 
\begin{equation*}
Y_{f}=Y\backslash \overline{Y}_{s}\text{ }
\end{equation*}%
and 
\begin{equation*}
\Theta =\underset{k\in \mathbb{Z}^{N}}{{\LARGE \cup }}\left( k+\overline{Y}%
_{s}\right) \text{.}
\end{equation*}%
In view of the compactness of $\overline{Y}_{s}$, it is an easy exercise to
verify that $\Theta $ is closed in $\mathbb{R}^{N}$. Now, let $\Omega $ be a
connected smooth bounded open set in $\mathbb{R}_{x}^{N}$. Let $%
0<\varepsilon <1$. We define 
\begin{equation*}
\Omega _{\varepsilon }=\Omega \backslash \varepsilon \Theta \text{.}
\end{equation*}%
This is a Lipschitz bounded open set in $\mathbb{R}_{x}^{N}$.

We are now in a position to state the problem under consideration in the
present setting. Let $T>0$. Given $\mathbf{f=}\left( f^{i}\right) \in
L^{2}\left( 0,T;L^{2}\left( \Omega ;\mathbb{R}\right) ^{N}\right) $, we
consider the initial boundary value problem in $\Omega _{\varepsilon }\times
]0,T[$ for the Navier-Stokes equations: 
\begin{equation}
\frac{\partial \mathbf{u}_{\varepsilon }}{\partial t}-\nu \Delta \mathbf{u}%
_{\varepsilon }+\sum_{j=1}^{N}u_{\varepsilon }^{j}\frac{\partial \mathbf{u}%
_{\varepsilon }}{\partial x_{j}}+\mathbf{grad}p_{\varepsilon }=\mathbf{f}%
\text{ in }\Omega _{\varepsilon }\times ]0,T[\text{,}  \label{eq1.7}
\end{equation}%
\begin{equation}
div\mathbf{u}_{\varepsilon }=0\text{ in }\Omega _{\varepsilon }\times ]0,T[%
\text{,}  \label{eq1.8}
\end{equation}%
\begin{equation}
\mathbf{u}_{\varepsilon }=0\text{ on }\partial \Omega _{\varepsilon }\times
]0,T[\text{,}  \label{eq1.9}
\end{equation}%
\begin{equation}
\mathbf{u}_{\varepsilon }\left( 0\right) =0\text{ in }\Omega _{\varepsilon }
\label{eq1.10}
\end{equation}%
where $\nu >0$ is the kinematic viscosity coefficient. For $N=2$ the problem
(\ref{eq1.7})-(\ref{eq1.10}) admits a unique solution $\left( \mathbf{u}%
_{\varepsilon },p_{\varepsilon }\right) $ with $\mathbf{u}_{\varepsilon }\in
L^{2}\left( 0,T;H_{0}^{1}\left( \Omega _{\varepsilon }\right) ^{N}\right) $
and $p_{\varepsilon }\in L^{2}\left( 0,T;L^{2}\left( \Omega _{\varepsilon
}\right) /\mathbb{R}\right) $ (see, e.g., \cite{bib10}, \cite{bib31}).

The aim is to study the limiting behavior of $\left( \mathbf{u}_{\varepsilon
},p_{\varepsilon }\right) $ as $\varepsilon \rightarrow 0$. In other words,
our purpose here is to discuss the homogenization of the initial boundary
value problem, (\ref{eq1.7})-(\ref{eq1.10}), the non-stationary
Navier-Stokes equations governing an incompressible fluid flow in the domain 
$\Omega _{\varepsilon }$.

Many authors have studied the homogenization, in porous media, of fluid
flows governed by the Stokes as well as the Navier-Stokes equations in
various physical contexts. We refer for example to \cite{bib2}, \cite{bib3}, 
\cite{bib6}, \cite{bib8}, \cite{bib12} and \cite{bib13}. Those authors
derive mostly the Darcy's law without the proof of a global convergence
result as it is stated and proved in Theorem \ref{th4.1}. Our topic here is
concerned with the sigma-convergence of the non-stationary incompressible
Navier-Stokes equations in porous media, with it rigorous proofs of
convergence results.

By means of the sigma-convergence, we derive the homogenized problem for (%
\ref{eq1.7})-(\ref{eq1.10}) which is given by (\ref{eq4.23})-(\ref{eq4.24}).
Equation (\ref{eq4.24}) is the Darcy's law with a time parameter. A similar
result as been established for the stationary case in \cite[Section 4]{bib23}%
, but in view of the difficulties encountered in the proof of estimates for
the pressure and the acceleration in the non-stationary case, Theorem \ref%
{th4.1} seems not to have been published before in the literature.

Unless otherwise specified, vector spaces throughout are considered over the
complex field, $\mathbb{C}$, and scalar functions are assumed to take
complex values. Let us recall some basic notation. If $X$ and $F$ denote a
locally compact space and a Banach space, respectively, then we write $%
\mathcal{C}\left( X;F\right) $ for continuous mappings of $X$ into $F$, and $%
\mathcal{B}\left( X;F\right) $ for those mappings in $\mathcal{C}\left(
X;F\right) $ that are bounded. We denote by $\mathcal{K}\left( X;F\right) $
the mappings in $\mathcal{C}\left( X;F\right) $ having compact supports.\ We
shall assume $\mathcal{B}\left( X;F\right) $ to be equipped with the
supremum norm $\left\Vert u\right\Vert _{\infty }=\sup_{x\in X}\left\Vert
u\left( x\right) \right\Vert $ ($\left\Vert {\small \cdot }\right\Vert $
denotes the norm in $F$). For shortness we will write $\mathcal{C}\left(
X\right) =\mathcal{C}\left( X;\mathbb{C}\right) $, $\mathcal{B}\left(
X\right) =\mathcal{B}\left( X;\mathbb{C}\right) $ and $\mathcal{K}\left(
X\right) =\mathcal{K}\left( X;\mathbb{C}\right) $. Likewise in the case when 
$F=\mathbb{C}$, the usual spaces $L^{p}\left( X;F\right) $ and $%
L_{loc}^{p}\left( X;F\right) $ ($X$ provided with a positive Radon measure)
will be denoted by $L^{p}\left( X\right) $ and $L_{loc}^{p}\left( X\right) $%
, respectively. Finally, the numerical space $\mathbb{R}^{N}$ and its open
sets are each provided with Lebesgue measure denoted by $dx=dx_{1}...dx_{N}$.

The rest of the paper is organized as follows. Section 2 is devoted to the
homogenization of (\ref{eq1.3})-(\ref{eq1.6}) under the periodicity
assumption on the coefficients $a_{ij}$. In Section 3 we reconsider the
homogenization of (\ref{eq1.3})-(\ref{eq1.6}) in a more general setting. The
periodicity hypothesis on the coefficients $a_{ij}$ is here replaced by an
abstract assumption covering a variety of concrete behaviours, the
periodicity being a particular case. Finally, in Section 4 we discuss the
homogenization of problem (\ref{eq1.7})-(\ref{eq1.10}).

\section{Periodic homogenization of unsteady Navier-Stokes type equations}

\subsection{Preliminaries}

Let $\Omega $ be a smooth bounded open set in $\mathbb{R}^{N}$. For fixed $%
0<\varepsilon <1$, we introduce the bilinear form $a^{\varepsilon }$ on $%
H_{0}^{1}\left( \Omega ;\mathbb{R}\right) ^{N}\times H_{0}^{1}\left( \Omega ;%
\mathbb{R}\right) ^{N}$ defined by%
\begin{equation*}
a^{\varepsilon }\left( \mathbf{u},\mathbf{v}\right)
=\sum_{k=1}^{N}\sum_{i,j=1}^{N}\int_{\Omega }a_{ij}^{\varepsilon }\frac{%
\partial u^{k}}{\partial x_{j}}\frac{\partial v^{k}}{\partial x_{i}}dx
\end{equation*}%
for $\mathbf{u}=\left( u^{k}\right) $ and $\mathbf{v}=\left( v^{k}\right) $
in $H_{0}^{1}\left( \Omega ;\mathbb{R}\right) ^{N}$. By virtue of (\ref%
{eq1.1}), the form $a^{\varepsilon }$ is symmetric. Further, in view of (\ref%
{eq1.2}), 
\begin{equation}
a^{\varepsilon }\left( \mathbf{v},\mathbf{v}\right) \geq \alpha \left\Vert 
\mathbf{v}\right\Vert _{H_{0}^{1}\left( \Omega \right) ^{N}}^{2}
\label{eq2.1}
\end{equation}%
for every $\mathbf{v}=\left( v^{k}\right) \in H_{0}^{1}\left( \Omega ;%
\mathbb{R}\right) ^{N}$ and $0<\varepsilon <1$, where 
\begin{equation*}
\left\Vert \mathbf{v}\right\Vert _{H_{0}^{1}\left( \Omega \right)
^{N}}=\left( \sum_{k=1}^{N}\int_{\Omega }\left\vert \nabla v^{k}\right\vert
dx\right) ^{\frac{1}{2}}
\end{equation*}%
with $\nabla v^{k}=\left( \frac{\partial v^{k}}{\partial x_{1}},...,\frac{%
\partial v^{k}}{\partial x_{N}}\right) $. On the other hand, it is clear
that a constant $c_{0}>0$ exists such that%
\begin{equation}
\left\vert a^{\varepsilon }\left( \mathbf{u},\mathbf{v}\right) \right\vert
\leq c_{0}\left\Vert \mathbf{u}\right\Vert _{H_{0}^{1}\left( \Omega \right)
^{N}}\left\Vert \mathbf{v}\right\Vert _{H_{0}^{1}\left( \Omega \right) ^{N}}
\label{eq2.2}
\end{equation}%
for every $\mathbf{u}$, $\mathbf{v}\in H_{0}^{1}\left( \Omega ;\mathbb{R}%
\right) ^{N}$ and $0<\varepsilon <1$. We introduce also the trilinear form $%
b $ on $H_{0}^{1}\left( \Omega ;\mathbb{R}\right) ^{N}\times H_{0}^{1}\left(
\Omega ;\mathbb{R}\right) ^{N}\times H_{0}^{1}\left( \Omega ;\mathbb{R}%
\right) ^{N}$ defined by%
\begin{equation*}
b\left( \mathbf{u},\mathbf{v},\mathbf{w}\right)
=\sum_{k=1}^{N}\sum_{j=1}^{N}\int_{\Omega }u^{j}\frac{\partial v^{k}}{%
\partial x_{j}}w^{k}dx
\end{equation*}%
for $\mathbf{u}=\left( u^{k}\right) $, $\mathbf{v}=\left( v^{k}\right) $ et $%
\mathbf{w}=\left( w^{k}\right) \in H_{0}^{1}\left( \Omega ;\mathbb{R}\right)
^{N}$. The form $b$ has the following properties \cite[pp.162-163]{bib31}:%
\begin{equation}
\left\vert b\left( \mathbf{u},\mathbf{v},\mathbf{w}\right) \right\vert \leq
c\left( N\right) \left\Vert \mathbf{u}\right\Vert _{H_{0}^{1}\left( \Omega
\right) ^{N}}\left\Vert \mathbf{v}\right\Vert _{H_{0}^{1}\left( \Omega
\right) ^{N}}\left\Vert \mathbf{w}\right\Vert _{H_{0}^{1}\left( \Omega
\right) ^{N}}  \label{eq2.3}
\end{equation}%
for all $\mathbf{u}$, $\mathbf{v}$ and $\mathbf{w}\in H_{0}^{1}\left( \Omega
;\mathbb{R}\right) ^{N}$, where the positive constant $c\left( N\right) $
depends on $N$ and $\Omega $;%
\begin{equation}
b\left( \mathbf{u},\mathbf{v},\mathbf{v}\right) =0\qquad \left( \mathbf{u}%
\in V\text{, }\mathbf{v}\in H_{0}^{1}\left( \Omega ;\mathbb{R}\right)
^{N}\right)  \label{eq2.4}
\end{equation}%
and%
\begin{equation}
b\left( \mathbf{u},\mathbf{v},\mathbf{w}\right) =-b\left( \mathbf{u},\mathbf{%
w},\mathbf{v}\right) \qquad \left( \mathbf{u}\in V\text{, }\mathbf{v}\text{
et }\mathbf{w}\in H_{0}^{1}\left( \Omega ;\mathbb{R}\right) ^{N}\right) 
\text{.}  \label{eq2.5}
\end{equation}%
For $\mathbf{u}$ and $\mathbf{v}\in H_{0}^{1}\left( \Omega ;\mathbb{R}%
\right) ^{N}$, let us consider the linear form $B\left( \mathbf{u},\mathbf{v}%
\right) $ on $H_{0}^{1}\left( \Omega ;\mathbb{R}\right) ^{N}$ defined by%
\begin{equation*}
\left\langle B\left( \mathbf{u},\mathbf{v}\right) ,\mathbf{w}\right\rangle
=b\left( \mathbf{u},\mathbf{v},\mathbf{w}\right) \qquad \left( \mathbf{w}\in
H_{0}^{1}\left( \Omega ;\mathbb{R}\right) ^{N}\right) \text{.}
\end{equation*}%
Let us set for $\mathbf{u}\in H_{0}^{1}\left( \Omega ;\mathbb{R}\right) ^{N}$%
\begin{equation*}
\widetilde{B}\left( \mathbf{u}\right) =B\left( \mathbf{u},\mathbf{u}\right) 
\text{.}
\end{equation*}%
In view of (\ref{eq2.3}), we have $\widetilde{B}\left( \mathbf{u}\right) \in
H^{-1}\left( \Omega ;\mathbb{R}\right) ^{N}$ and 
\begin{equation*}
\left\Vert \widetilde{B}\left( \mathbf{u}\right) \right\Vert _{H^{-1}\left(
\Omega \right) ^{N}}\leq c\left( N\right) \left\Vert \mathbf{u}\right\Vert
_{H_{0}^{1}\left( \Omega \right) ^{N}}^{2}\text{.}
\end{equation*}

Before we can establish some estimates on the velocity $\mathbf{u}%
_{\varepsilon }$, the acceleration $\frac{\partial \mathbf{u}_{\varepsilon }%
}{\partial t}$ and the pressure $p_{\varepsilon }$, let us recall the
following results.

\begin{lemma}
\label{lem2.1} Let $\Omega $ be a bounded open set in $\mathbb{R}^{2}$. We
have the following inequalities:%
\begin{equation}
\left\Vert v\right\Vert _{L^{4}\left( \Omega \right) }\leq 2^{\frac{1}{4}%
}\left\Vert v\right\Vert _{L^{2}\left( \Omega \right) }^{\frac{1}{2}%
}\left\Vert \mathbf{grad}v\right\Vert _{L^{2}\left( \Omega \right) }^{\frac{1%
}{2}}\qquad \left( v\in H_{0}^{1}\left( \Omega ;\mathbb{R}\right) \right) 
\text{,}  \label{eq2.6}
\end{equation}%
\begin{equation}
\left\vert b\left( \mathbf{u},\mathbf{v},\mathbf{w}\right) \right\vert \leq
2^{\frac{1}{2}}\left\vert \mathbf{u}\right\vert ^{\frac{1}{2}}\left\Vert 
\mathbf{u}\right\Vert ^{\frac{1}{2}}\left\Vert \mathbf{v}\right\Vert
\left\vert \mathbf{w}\right\vert ^{\frac{1}{2}}\left\Vert \mathbf{w}%
\right\Vert ^{\frac{1}{2}}\qquad \left( \mathbf{u}\text{, }\mathbf{v}\text{, 
}\mathbf{w}\in H_{0}^{1}\left( \Omega ;\mathbb{R}\right) ^{2}\right) \text{,}
\label{eq2.7}
\end{equation}%
where $\left\vert {\small \cdot }\right\vert $ and $\left\Vert {\small \cdot 
}\right\Vert $ are respectively the norms in $L^{2}\left( \Omega \right)
^{N} $ and $H_{0}^{1}\left( \Omega \right) ^{N}$. Moreover, if $\mathbf{u}%
\in L^{2}\left( 0,T;V\right) \cap L^{\infty }\left( 0,T;H\right) $, $H$
being the closure of $\mathcal{V}=\left\{ \mathbf{u}\in \mathcal{D}\left(
\Omega ;\mathbb{R}\right) ^{2}:div\mathbf{u}=0\right\} $ in $L^{2}\left(
\Omega ;\mathbb{R}\right) ^{2}$, then $\widetilde{B}\left( \mathbf{u}\right)
\in L^{2}\left( 0,T;V^{\prime }\right) $ and%
\begin{equation}
\left\Vert \widetilde{B}\left( \mathbf{u}\right) \right\Vert _{L^{2}\left(
0,T;V^{\prime }\right) }\leq 2^{\frac{1}{2}}\left\Vert \mathbf{u}\right\Vert
_{L^{\infty }\left( 0,T;H\right) }\left\Vert \mathbf{u}\right\Vert
_{L^{2}\left( 0,T;V\right) }\text{.}  \label{eq2.8}
\end{equation}
\end{lemma}

The proof of the above lemma can be found in \cite[pp.291-293]{bib31}.

The following regularity result is fundamental for the estimates of the
solution $\left( \mathbf{u}_{\varepsilon },p_{\varepsilon }\right) $ of (\ref%
{eq1.3})-(\ref{eq1.6}).

\begin{lemma}
\label{lem2.2} Suppose in (\ref{eq1.3})-(\ref{eq1.6}) that $N=2$ and 
\begin{equation}
\mathbf{f}\text{, }\mathbf{f}^{\prime }\in L^{2}\left( 0,T;V^{\prime
}\right) \text{ and }\mathbf{f}\left( 0\right) \in L^{2}\left( \Omega ;%
\mathbb{R}\right) ^{N}\text{.}  \label{eq2.9}
\end{equation}%
Then the solution $\mathbf{u}_{\varepsilon }$ verifies:%
\begin{equation}
\mathbf{u}_{\varepsilon }^{\prime }\in L^{2}\left( 0,T;V\right) \cap
L^{\infty }\left( 0,T;H\right) \text{.}  \label{eq2.10}
\end{equation}
\end{lemma}

The proof of the above lemma follows by the same line of argument as in the
proof of \cite[p.299, Theorem 3.5]{bib31}. So we omit it. We are now able to
prove the result on the estimates.

\begin{proposition}
\label{pr2.1} Under the hypotheses of Lemma \ref{lem2.2}, there exists a
constant $c>0$ (independent of $\varepsilon $) such that the pair $\left( 
\mathbf{u}_{\varepsilon },p_{\varepsilon }\right) $ solution of (\ref{eq1.3}%
)-(\ref{eq1.6}) in $\mathcal{W}\left( 0,T\right) \times L^{2}\left(
0,T;L^{2}\left( \Omega ;\mathbb{R}\right) \mathfrak{/}\mathbb{R}\right) $
satisfies:%
\begin{equation}
\left\Vert \mathbf{u}_{\varepsilon }\right\Vert _{\mathcal{W}\left(
0,T\right) }\leq c  \label{eq2.11}
\end{equation}%
\begin{equation}
\left\Vert \frac{\partial \mathbf{u}_{\varepsilon }}{\partial t}\right\Vert
_{L^{2}\left( 0,T;H^{-1}\left( \Omega \right) ^{N}\right) }\leq c
\label{eq2.12}
\end{equation}%
and 
\begin{equation}
\left\Vert p_{\varepsilon }\right\Vert _{L^{2}\left( 0,T;L^{2}\left( \Omega
\right) \right) }\leq c\text{.}  \label{eq2.13}
\end{equation}
\end{proposition}

\begin{proof}
Let $\left( \mathbf{u}_{\varepsilon },p_{\varepsilon }\right) $ be the
solution of (\ref{eq1.3})-(\ref{eq1.6}). We have%
\begin{equation}
\left( \mathbf{u}_{\varepsilon }^{\prime }\left( t\right) ,\mathbf{v}\right)
+a^{\varepsilon }\left( \mathbf{u}_{\varepsilon }\left( t\right) ,\mathbf{v}%
\right) +b\left( \mathbf{u}_{\varepsilon }\left( t\right) ,\mathbf{u}%
_{\varepsilon }\left( t\right) ,\mathbf{v}\right) =\left( \mathbf{f}\left(
t\right) ,\mathbf{v}\right) \text{\quad }\left( \mathbf{v}\in V\right)
\label{eq2.14}
\end{equation}%
for almost all $t\in \left[ 0,T\right] $, where $\left( ,\right) $ denotes
the duality pairing between $V^{\prime }$ and $V$ as well as between $%
H^{-1}\left( \Omega ;\mathbb{R}\right) ^{N}$ and $H_{0}^{1}\left( \Omega ;%
\mathbb{R}\right) ^{N}$. By taking in particular $\mathbf{v}=\mathbf{u}%
_{\varepsilon }\left( t\right) $ in (\ref{eq2.14}), we have for almost all $%
t\in \left[ 0,T\right] $ 
\begin{equation*}
\frac{d}{dt}\left\vert \mathbf{u}_{\varepsilon }\left( t\right) \right\vert
^{2}+2\alpha \left\Vert \mathbf{u}_{\varepsilon }\left( t\right) \right\Vert
^{2}\leq \frac{1}{\alpha }\left\Vert \mathbf{f}\left( t\right) \right\Vert
_{V^{\prime }}^{2}+\alpha \left\Vert \mathbf{u}_{\varepsilon }\left(
t\right) \right\Vert ^{2}
\end{equation*}%
since $b\left( \mathbf{u}_{\varepsilon }\left( t\right) ,\mathbf{u}%
_{\varepsilon }\left( t\right) ,\mathbf{u}_{\varepsilon }\left( t\right)
\right) =0$ in view of (\ref{eq2.4}) ($\left\vert {\small \cdot }\right\vert 
$ and $\left\Vert {\small \cdot }\right\Vert $ are respectively the norms in 
$L^{2}\left( \Omega \right) ^{N}$ and $H_{0}^{1}\left( \Omega \right) ^{N}$%
). Hence, for every $s\in \left[ 0,T\right] $ 
\begin{equation*}
\left\vert \mathbf{u}_{\varepsilon }\left( s\right) \right\vert ^{2}+\alpha
\int_{0}^{s}\left\Vert \mathbf{u}_{\varepsilon }\left( t\right) \right\Vert
^{2}dt\leq \frac{1}{\alpha }\int_{0}^{T}\left\Vert \mathbf{f}\left( t\right)
\right\Vert _{V^{\prime }}^{2}dt
\end{equation*}%
and 
\begin{equation}
\left\vert \mathbf{u}_{\varepsilon }\left( s\right) \right\vert ^{2}\leq 
\frac{1}{\alpha }\int_{0}^{T}\left\Vert \mathbf{f}\left( t\right)
\right\Vert _{V^{\prime }}^{2}dt  \label{eq2.15}
\end{equation}%
since $\mathbf{u}_{\varepsilon }\left( 0\right) =0$. We have also%
\begin{equation}
\alpha \int_{0}^{T}\left\Vert \mathbf{u}_{\varepsilon }\left( t\right)
\right\Vert ^{2}dt\leq \frac{1}{\alpha }\int_{0}^{T}\left\Vert \mathbf{f}%
\left( t\right) \right\Vert _{V^{\prime }}^{2}dt\text{.}  \label{eq2.16}
\end{equation}%
On the other hand, the abstract parabolic problem for (\ref{eq1.3})-(\ref%
{eq1.6}) gives%
\begin{equation*}
\mathbf{u}_{\varepsilon }^{\prime }=\mathbf{f}-A_{\varepsilon }\mathbf{u}%
_{\varepsilon }-\widetilde{B}\left( \mathbf{u}_{\varepsilon }\right) \text{,}
\end{equation*}%
where $A_{\varepsilon }$ is the linear operator of $V$ into $V^{\prime }$
defined by 
\begin{equation*}
(A_{\varepsilon }\mathbf{u},\mathbf{v})=a^{\varepsilon }\left( \mathbf{u},%
\mathbf{v}\right) \qquad \left( \mathbf{u},\text{ }\mathbf{v}\in V\right) 
\text{.}
\end{equation*}%
Hence, in view of (\ref{eq2.2})%
\begin{equation}
\left\Vert \mathbf{u}_{\varepsilon }^{\prime }\right\Vert _{L^{2}\left(
0,T;V^{\prime }\right) }\leq \left\Vert \mathbf{f}\right\Vert _{L^{2}\left(
0,T;V^{\prime }\right) }+c_{0}\left\Vert \mathbf{u}_{\varepsilon
}\right\Vert _{L^{2}\left( 0,T;V\right) }+\left\Vert \widetilde{B}\left( 
\mathbf{u}_{\varepsilon }\right) \right\Vert _{L^{2}\left( 0,T;V^{\prime
}\right) }\text{,}  \label{eq2.17}
\end{equation}%
since $\widetilde{B}\left( \mathbf{u}_{\varepsilon }\right) \in L^{2}\left(
0,T;V^{\prime }\right) $ (see (\ref{eq2.8})). Thus, by (\ref{eq2.15})-(\ref%
{eq2.17}) and (\ref{eq2.8}) one quickly arrives at (\ref{eq2.11}). Let us
show (\ref{eq2.12}). By virtue of Lemma \ref{lem2.2} we have $\mathbf{u}%
_{\varepsilon }^{\prime }\in L^{2}\left( 0,T;V\right) \cap L^{\infty }\left(
0,T;H\right) $. On the other hand, we are allowed to differentiate (\ref%
{eq2.14}) in the distribution sense on $]0,T[$. We get 
\begin{equation*}
\frac{d}{dt}\left( \mathbf{u}_{\varepsilon }^{\prime },\mathbf{v}\right)
+a^{\varepsilon }\left( \mathbf{u}_{\varepsilon }^{\prime },\mathbf{v}%
\right) +b\left( \mathbf{u}_{\varepsilon }^{\prime },\mathbf{u}_{\varepsilon
},\mathbf{v}\right) +b\left( \mathbf{u}_{\varepsilon },\mathbf{u}%
_{\varepsilon }^{\prime },\mathbf{v}\right) =\left( \mathbf{f}^{\prime },%
\mathbf{v}\right) \text{\quad }\left( \mathbf{v}\in V\right) \text{,}
\end{equation*}%
i.e.,%
\begin{equation*}
\frac{d}{dt}\left( \mathbf{u}_{\varepsilon }^{\prime },\mathbf{v}\right)
=\left( \mathbf{f}^{\prime }-A_{\varepsilon }\mathbf{u}_{\varepsilon
}^{\prime }-B\left( \mathbf{u}_{\varepsilon }^{\prime },\mathbf{u}%
_{\varepsilon }\right) -B\left( \mathbf{u}_{\varepsilon },\mathbf{u}%
_{\varepsilon }^{\prime }\right) ,\mathbf{v}\right) \text{.}
\end{equation*}%
But the function $\mathbf{f}^{\prime }-A_{\varepsilon }\mathbf{u}%
_{\varepsilon }^{\prime }-B\left( \mathbf{u}_{\varepsilon }^{\prime },%
\mathbf{u}_{\varepsilon }\right) -B\left( \mathbf{u}_{\varepsilon },\mathbf{u%
}_{\varepsilon }^{\prime }\right) $ belongs to $L^{2}\left( 0,T;V^{\prime
}\right) $: indeed, $\mathbf{f}^{\prime }\in L^{2}\left( 0,T;V^{\prime
}\right) $ by hypothesis, $A_{\varepsilon }\mathbf{u}_{\varepsilon }^{\prime
}\in L^{2}\left( 0,T;V^{\prime }\right) $, and further we have%
\begin{equation*}
\int_{0}^{T}\left\Vert B\left( \mathbf{u}_{\varepsilon }^{\prime }\left(
t\right) ,\mathbf{u}_{\varepsilon }\left( t\right) \right) \right\Vert
_{V^{\prime }}^{2}dt\leq 2\left\Vert \mathbf{u}_{\varepsilon }^{\prime
}\right\Vert _{L^{\infty }\left( 0,T;H\right) }\left\Vert \mathbf{u}%
_{\varepsilon }\right\Vert _{L^{\infty }\left( 0,T;H\right)
}\int_{0}^{T}\left\Vert \mathbf{u}_{\varepsilon }^{\prime }\left( t\right)
\right\Vert \left\Vert \mathbf{u}_{\varepsilon }\left( t\right) \right\Vert
dt
\end{equation*}%
and%
\begin{equation*}
\int_{0}^{T}\left\Vert B\left( \mathbf{u}_{\varepsilon }\left( t\right) ,%
\mathbf{u}_{\varepsilon }^{\prime }\left( t\right) \right) \right\Vert
_{V^{\prime }}^{2}dt\leq 2\left\Vert \mathbf{u}_{\varepsilon }^{\prime
}\right\Vert _{L^{\infty }\left( 0,T;H\right) }\left\Vert \mathbf{u}%
_{\varepsilon }\right\Vert _{L^{\infty }\left( 0,T;H\right)
}\int_{0}^{T}\left\Vert \mathbf{u}_{\varepsilon }^{\prime }\left( t\right)
\right\Vert \left\Vert \mathbf{u}_{\varepsilon }\left( t\right) \right\Vert
dt
\end{equation*}%
by virtue of part (\ref{eq2.7}) of Lemma \ref{lem2.1}. Thus by \cite[p.250,
Lemma 1.1]{bib31}, $\mathbf{u}_{\varepsilon }^{\prime \prime }\in
L^{2}\left( 0,T;V^{\prime }\right) $ and 
\begin{equation}
\left( \mathbf{u}_{\varepsilon }^{\prime \prime },\mathbf{v}\right)
+a^{\varepsilon }\left( \mathbf{u}_{\varepsilon }^{\prime },\mathbf{v}%
\right) +b\left( \mathbf{u}_{\varepsilon }^{\prime },\mathbf{u}_{\varepsilon
},\mathbf{v}\right) +b\left( \mathbf{u}_{\varepsilon },\mathbf{u}%
_{\varepsilon }^{\prime },\mathbf{v}\right) =\left\langle \mathbf{f}^{\prime
},\mathbf{v}\right\rangle  \label{eq2.18}
\end{equation}%
for all $\mathbf{v}\in V$. Furthermore, $\mathbf{u}_{\varepsilon }^{\prime
}\in L^{2}\left( 0,T;V\right) $ and $\mathbf{u}_{\varepsilon }^{\prime
\prime }\in L^{2}\left( 0,T;V^{\prime }\right) $ imply that $\mathbf{u}%
_{\varepsilon }^{\prime }\in \mathcal{W}\left( 0,T\right) $. But $\mathcal{W}%
\left( 0,T\right) $ is continuously embedded in $\mathcal{C}\left( \left[ 0,T%
\right] ;H\right) $ (see \cite{bib31}), thus $\mathbf{u}_{\varepsilon
}^{\prime }\in \mathcal{C}\left( \left[ 0,T\right] ;H\right) $. Moreover,
replacing $\mathbf{v}$ by $\mathbf{u}_{\varepsilon }^{\prime }\left(
t\right) $ in (\ref{eq2.14}), we obtain%
\begin{equation*}
\left\vert \mathbf{u}_{\varepsilon }^{\prime }\left( t\right) \right\vert
^{2}+a^{\varepsilon }\left( \mathbf{u}_{\varepsilon }\left( t\right) ,%
\mathbf{u}_{\varepsilon }^{\prime }\left( t\right) \right) +b\left( \mathbf{u%
}_{\varepsilon }\left( t\right) ,\mathbf{u}_{\varepsilon }\left( t\right) ,%
\mathbf{u}_{\varepsilon }^{\prime }\left( t\right) \right) =\left( \mathbf{f}%
\left( t\right) ,\mathbf{u}_{\varepsilon }^{\prime }\left( t\right) \right) 
\text{.}
\end{equation*}%
As $\mathbf{u}_{\varepsilon }^{\prime }\in \mathcal{C}\left( 0,T;H\right) $,
one has in particular for $t=0$,%
\begin{equation*}
\left\vert \mathbf{u}_{\varepsilon }^{\prime }\left( 0\right) \right\vert
^{2}=\left( \mathbf{f}\left( 0\right) ,\mathbf{u}_{\varepsilon }^{\prime
}\left( 0\right) \right) -a^{\varepsilon }\left( \mathbf{u}_{\varepsilon
}\left( 0\right) ,\mathbf{u}_{\varepsilon }^{\prime }\left( 0\right) \right)
-b\left( \mathbf{u}_{\varepsilon }\left( 0\right) ,\mathbf{u}_{\varepsilon
}\left( 0\right) ,\mathbf{u}_{\varepsilon }^{\prime }\left( 0\right) \right)
\end{equation*}%
where $\left( ,\right) $ denotes also the scalar product in $H$. But $%
\mathbf{u}_{\varepsilon }\left( 0\right) =0$, thus by the preceding equality
we obtain%
\begin{equation}
\left\vert \mathbf{u}_{\varepsilon }^{\prime }\left( 0\right) \right\vert
\leq \left\vert \mathbf{f}\left( 0\right) \right\vert \text{.}
\label{eq2.19}
\end{equation}%
The inequality (\ref{eq2.19}) shows that $\mathbf{u}_{\varepsilon }^{\prime
}\left( 0\right) $ lies in a bounded subset of $H$.\ On\ the other hand, by
taking in particular $\mathbf{v}=\mathbf{u}_{\varepsilon }^{\prime }\left(
t\right) $ in (\ref{eq2.18}), we get 
\begin{equation}
\frac{d}{dt}\left\vert \mathbf{u}_{\varepsilon }^{\prime }\left( t\right)
\right\vert ^{2}+2\alpha \left\Vert \mathbf{u}_{\varepsilon }^{\prime
}\left( t\right) \right\Vert ^{2}+2b\left( \mathbf{u}_{\varepsilon }^{\prime
}\left( t\right) ,\mathbf{u}_{\varepsilon }\left( t\right) ,\mathbf{u}%
_{\varepsilon }^{\prime }\left( t\right) \right) \leq 2\left( \mathbf{f}%
^{\prime }\left( t\right) ,\mathbf{u}_{\varepsilon }^{\prime }\left(
t\right) \right) \text{,}  \label{eq2.20}
\end{equation}%
since $b\left( \mathbf{u}_{\varepsilon }\left( t\right) ,\mathbf{u}%
_{\varepsilon }^{\prime }\left( t\right) ,\mathbf{u}_{\varepsilon }^{\prime
}\left( t\right) \right) =0$. But, by virtue of Lemme \ref{lem2.1}%
\begin{equation*}
2\left\vert b\left( \mathbf{u}_{\varepsilon }^{\prime }\left( t\right) ,%
\mathbf{u}_{\varepsilon }\left( t\right) ,\mathbf{u}_{\varepsilon }^{\prime
}\left( t\right) \right) \right\vert \leq 2^{\frac{3}{2}}\left\vert \mathbf{u%
}_{\varepsilon }^{\prime }\left( t\right) \right\vert \left\Vert \mathbf{u}%
_{\varepsilon }^{\prime }\left( t\right) \right\Vert \left\Vert \mathbf{u}%
_{\varepsilon }\left( t\right) \right\Vert
\end{equation*}%
\begin{equation*}
\leq \alpha \left\Vert \mathbf{u}_{\varepsilon }^{\prime }\left( t\right)
\right\Vert ^{2}+\frac{2}{\alpha }\left\Vert \mathbf{u}_{\varepsilon }\left(
t\right) \right\Vert ^{2}\left\vert \mathbf{u}_{\varepsilon }^{\prime
}\left( t\right) \right\vert ^{2}\text{.}
\end{equation*}%
Hence, we deduce from (\ref{eq2.20}) that 
\begin{equation}
\frac{d}{dt}\left\vert \mathbf{u}_{\varepsilon }^{\prime }\left( t\right)
\right\vert ^{2}+\frac{\alpha }{2}\left\Vert \mathbf{u}_{\varepsilon
}^{\prime }\left( t\right) \right\Vert ^{2}\leq \frac{2}{\alpha }\left\Vert 
\mathbf{f}^{\prime }\left( t\right) \right\Vert _{V^{\prime }}^{2}+\frac{2}{%
\alpha }\left\Vert \mathbf{u}_{\varepsilon }\left( t\right) \right\Vert
^{2}\left\vert \mathbf{u}_{\varepsilon }^{\prime }\left( t\right)
\right\vert ^{2}\text{.}  \label{eq2.21}
\end{equation}%
By (\ref{eq2.21}) we have 
\begin{equation}
\frac{d}{dt}\left\vert \mathbf{u}_{\varepsilon }^{\prime }\left( t\right)
\right\vert ^{2}-\frac{2}{\alpha }\left\Vert \mathbf{u}_{\varepsilon }\left(
t\right) \right\Vert ^{2}\left\vert \mathbf{u}_{\varepsilon }^{\prime
}\left( t\right) \right\vert ^{2}\leq \frac{2}{\alpha }\left\Vert \mathbf{f}%
^{\prime }\left( t\right) \right\Vert _{V^{\prime }}^{2}\text{.}
\label{eq2.22}
\end{equation}%
As 
\begin{equation*}
\exp \left( -\int_{0}^{t}\frac{2}{\alpha }\left\Vert \mathbf{u}_{\varepsilon
}\left( s\right) \right\Vert ^{2}ds\right) \leq 1\text{,}
\end{equation*}%
multiplying (\ref{eq2.22}) by $\exp \left( -\int_{0}^{t}\frac{2}{\alpha }%
\left\Vert \mathbf{u}_{\varepsilon }\left( s\right) \right\Vert
^{2}ds\right) $ yields%
\begin{equation*}
\left( \frac{d}{dt}\left\vert \mathbf{u}_{\varepsilon }^{\prime }\left(
t\right) \right\vert ^{2}-\frac{2}{\alpha }\left\Vert \mathbf{u}%
_{\varepsilon }\left( t\right) \right\Vert ^{2}\left\vert \mathbf{u}%
_{\varepsilon }^{\prime }\left( t\right) \right\vert ^{2}\right) \exp \left(
-\int_{0}^{t}\frac{2}{\alpha }\left\Vert \mathbf{u}_{\varepsilon }\left(
s\right) \right\Vert ^{2}ds\right) \leq \frac{2}{\alpha }\left\Vert \mathbf{f%
}^{\prime }\left( t\right) \right\Vert _{V^{\prime }}^{2}\text{, }
\end{equation*}%
i.e., 
\begin{equation}
\frac{d}{dt}\left\{ \left\vert \mathbf{u}_{\varepsilon }^{\prime }\left(
t\right) \right\vert ^{2}\exp \left( -\int_{0}^{t}\frac{2}{\alpha }%
\left\Vert \mathbf{u}_{\varepsilon }\left( s\right) \right\Vert
^{2}ds\right) \right\} \leq \frac{2}{\alpha }\left\Vert \mathbf{f}^{\prime
}\left( t\right) \right\Vert _{V^{\prime }}^{2}\text{.}  \label{eq2.23}
\end{equation}%
Thus, integrating (\ref{eq2.23}) on $\left( 0,t\right) $, $t\in \left(
0,T\right) $, we have%
\begin{equation}
\left\vert \mathbf{u}_{\varepsilon }^{\prime }\left( t\right) \right\vert
^{2}\leq \left\{ \left\vert \mathbf{u}_{\varepsilon }^{\prime }\left(
0\right) \right\vert ^{2}+\frac{2}{\alpha }\int_{0}^{T}\left\Vert \mathbf{f}%
^{\prime }\left( t\right) \right\Vert _{V^{\prime }}^{2}dt\right\} \exp
\left( \int_{0}^{T}\frac{2}{\alpha }\left\Vert \mathbf{u}_{\varepsilon
}\left( s\right) \right\Vert ^{2}ds\right) \text{ }  \label{eq2.24}
\end{equation}%
for all $t\in \left( 0,T\right) $. It follows from (\ref{eq2.16}), (\ref%
{eq2.19}) and (\ref{eq2.24}) that the sequence $\left( \mathbf{u}%
_{\varepsilon }^{\prime }\right) _{\varepsilon >0}$ is bounded in $%
L^{2}\left( 0,T;L^{2}\left( \Omega \right) ^{N}\right) $. Hence, the
sequence $\left( \frac{\partial \mathbf{u}_{\varepsilon }}{\partial t}%
\right) _{\varepsilon >0}$ is bounded in $L^{2}\left( 0,T;H^{-1}\left(
\Omega \right) ^{N}\right) $ and (\ref{eq2.12}) is proved. Further, by part (%
\ref{eq2.7}) of Lemma \ref{lem2.1} we have 
\begin{equation*}
\left\vert b\left( \mathbf{u},\mathbf{v},\mathbf{w}\right) \right\vert \leq
2^{\frac{1}{2}}\left\vert \mathbf{u}\right\vert ^{\frac{1}{2}}\left\Vert 
\mathbf{u}\right\Vert ^{\frac{1}{2}}\left\Vert \mathbf{v}\right\Vert
\left\vert \mathbf{w}\right\vert ^{\frac{1}{2}}\left\Vert \mathbf{w}%
\right\Vert ^{\frac{1}{2}}\text{\quad }\left( \mathbf{u},\mathbf{v},\mathbf{w%
}\in H_{0}^{1}\left( \Omega ;\mathbb{R}\right) ^{N}\right) \text{.}
\end{equation*}%
Thus, if $\mathbf{u}\in V$ then $b\left( \mathbf{u},\mathbf{v},\mathbf{w}%
\right) =-b\left( \mathbf{u},\mathbf{w},\mathbf{v}\right) $ and 
\begin{equation*}
\left\vert b\left( \mathbf{u},\mathbf{v},\mathbf{w}\right) \right\vert \leq
2^{\frac{1}{2}}\left\vert \mathbf{u}\right\vert ^{\frac{1}{2}}\left\Vert 
\mathbf{u}\right\Vert ^{\frac{1}{2}}\left\Vert \mathbf{w}\right\Vert
\left\vert \mathbf{v}\right\vert ^{\frac{1}{2}}\left\Vert \mathbf{v}%
\right\Vert ^{\frac{1}{2}}\text{\quad }\left( \mathbf{v}\text{, }\mathbf{w}%
\in H_{0}^{1}\left( \Omega ;\mathbb{R}\right) ^{N}\right) \text{.}
\end{equation*}%
In particular, 
\begin{equation*}
\left\vert b\left( \mathbf{u},\mathbf{u},\mathbf{v}\right) \right\vert \leq
2^{\frac{1}{2}}\left\vert \mathbf{u}\right\vert \left\Vert \mathbf{u}%
\right\Vert \left\Vert \mathbf{v}\right\Vert \text{ for }\mathbf{u}\in V%
\text{ and }\mathbf{v}\in H_{0}^{1}\left( \Omega ;\mathbb{R}\right) ^{N}%
\text{,}
\end{equation*}%
thus%
\begin{equation*}
\left\Vert \widetilde{B}\left( \mathbf{u}\right) \right\Vert _{H^{-1}\left(
\Omega \right) ^{N}}\leq 2^{\frac{1}{2}}\left\vert \mathbf{u}\right\vert
\left\Vert \mathbf{u}\right\Vert \text{\qquad }\left( \mathbf{u}\in V\right) 
\text{.}
\end{equation*}%
It follows from the preceding inequality that $\widetilde{B}\left( \mathbf{u}%
_{\varepsilon }\right) \in L^{2}\left( 0,T;H^{-1}\left( \Omega \right)
^{N}\right) $ and 
\begin{equation}
\int_{0}^{T}\left\Vert \widetilde{B}\left( \mathbf{u}_{\varepsilon }\right)
\right\Vert _{H^{-1}\left( \Omega \right) ^{N}}^{2}dt\leq 2\left\Vert 
\mathbf{u}_{\varepsilon }\right\Vert _{L^{\infty }\left( 0,T;H\right)
}^{2}\left\Vert \mathbf{u}_{\varepsilon }\right\Vert _{L^{2}\left(
0,T;V\right) }^{2}\text{.}  \label{eq2.25}
\end{equation}%
On the other hand, $p_{\varepsilon }\left( t\right) \in L^{2}\left( \Omega ;%
\mathbb{R}\right) \mathfrak{/}\mathbb{R}$ for almost all $t\in \left(
0,T\right) $. Consequently, by virtue of \cite[p. 30]{bib30} there exists
some $\mathbf{v}_{\varepsilon }\left( t\right) \in $ $H_{0}^{1}\left( \Omega
;\mathbb{R}\right) ^{N}$ such that%
\begin{equation}
div\mathbf{v}_{\varepsilon }\left( t\right) =p_{\varepsilon }\left( t\right)
\label{eq2.26}
\end{equation}%
\begin{equation}
\left\Vert \mathbf{v}_{\varepsilon }\left( t\right) \right\Vert
_{H_{0}^{1}\left( \Omega \right) ^{N}}\leq c_{1}\left\Vert p_{\varepsilon
}\left( t\right) \right\Vert _{L^{2}\left( \Omega \right) }\text{,}
\label{eq2.27}
\end{equation}%
where the constant $c_{1}$ depends solely on $\Omega $. Multiplying (\ref%
{eq1.3}) by $\mathbf{v}_{\varepsilon }\left( t\right) $, we have for almost
all $t\in \left( 0,T\right) $ 
\begin{equation}
\left( \mathbf{u}_{\varepsilon }^{\prime },\mathbf{v}_{\varepsilon }\left(
t\right) \right) +a^{\varepsilon }\left( \mathbf{u}_{\varepsilon }\left(
t\right) ,\mathbf{v}_{\varepsilon }\left( t\right) \right) +b\left( \mathbf{u%
}_{\varepsilon }\left( t\right) ,\mathbf{u}_{\varepsilon }\left( t\right) ,%
\mathbf{v}_{\varepsilon }\left( t\right) \right) -\int_{\Omega
}p_{\varepsilon }\left( t\right) div\mathbf{v}_{\varepsilon }\left( t\right)
dx=\left( \mathbf{f}\left( t\right) ,\mathbf{v}_{\varepsilon }\left(
t\right) \right) \text{.}  \label{eq2.28}
\end{equation}%
Integrating (\ref{eq2.28}) on $\left( 0,T\right) $, and using (\ref{eq2.25}%
)-(\ref{eq2.27}) yield 
\begin{equation}
\left\Vert p_{\varepsilon }\right\Vert _{L^{2}\left( Q\right) }^{2}\leq
c_{1}c\left\Vert \mathbf{u}_{\varepsilon }^{\prime }\right\Vert
_{L^{2}\left( 0,T;L^{2}\left( \Omega \right) ^{N}\right) }\left\Vert
p_{\varepsilon }\right\Vert _{L^{2}\left( Q\right) }+\sqrt{2}c_{1}\left\Vert 
\mathbf{u}_{\varepsilon }\right\Vert _{L^{\infty }\left( 0,T;H\right)
}\left\Vert \mathbf{u}_{\varepsilon }\right\Vert _{L^{2}\left( 0,T;V\right)
}\left\Vert p_{\varepsilon }\right\Vert _{L^{2}\left( Q\right) }
\label{eq2.29}
\end{equation}%
\begin{equation*}
+c_{1}\left\Vert \mathbf{f}\right\Vert _{L^{2}\left( 0,T;H^{-1}\left( \Omega
\right) \right) }\left\Vert p_{\varepsilon }\right\Vert _{L^{2}\left(
Q\right) }+c_{0}c_{1}\left\Vert \mathbf{u}_{\varepsilon }\right\Vert
_{L^{2}\left( 0,T;V\right) }\left\Vert p_{\varepsilon }\right\Vert
_{L^{2}\left( Q\right) }\text{,}
\end{equation*}%
where $c$ is the constant in the Poincar\'{e} inequality, $c_{0}$ and $c_{1}$
are the constants in (\ref{eq2.2}) and (\ref{eq2.27})\ respectively. It
follows from (\ref{eq2.29}) that%
\begin{equation}
\left\Vert p_{\varepsilon }\right\Vert _{L^{2}\left( Q\right) }\leq
c_{1}c\left\Vert \mathbf{u}_{\varepsilon }^{\prime }\right\Vert
_{L^{2}\left( 0,T;L^{2}\left( \Omega \right) ^{N}\right) }+\sqrt{2}%
c_{1}\left\Vert \mathbf{u}_{\varepsilon }\right\Vert _{L^{\infty }\left(
0,T;H\right) }\left\Vert \mathbf{u}_{\varepsilon }\right\Vert _{L^{2}\left(
0,T;V\right) }  \label{eq2.30}
\end{equation}%
\begin{equation*}
+c_{1}\left\Vert \mathbf{f}\right\Vert _{L^{2}\left( 0,T;H^{-1}\left( \Omega
\right) \right) }+c_{0}c_{1}\left\Vert \mathbf{u}_{\varepsilon }\right\Vert
_{L^{2}\left( 0,T;V\right) }\text{.}
\end{equation*}%
Using (\ref{eq2.11}), (\ref{eq2.12}) and (\ref{eq2.16}) already proved, one
quickly arrives a (\ref{eq2.13}) by (\ref{eq2.30}). The proof of the
proposition is complete.
\end{proof}

\subsection{A convergence result for (\protect\ref{eq1.3})-(\protect\ref%
{eq1.6})}

We set $Y=\left( -\frac{1}{2},\frac{1}{2}\right) ^{N}$, $Y$ considered as a
subset of $\mathbb{R}_{y}^{N}$ (the space $\mathbb{R}^{N}$ of variables $%
y=\left( y_{1},...,y_{N}\right) $). We set also $Z=\left( -\frac{1}{2},\frac{%
1}{2}\right) $, $Z$ considered as a subset of $\mathbb{R}_{\tau }$ (the
space $\mathbb{R}$ of variables $\tau $). Our purpose is to study the
homogenization of (\ref{eq1.3})-(\ref{eq1.6}) under the periodicity
hypothesis on $a_{ij}$.

\subsubsection{\textbf{Preliminaries}}

Let us first recall that a function $u\in L_{loc}^{1}\left( \mathbb{R}%
_{y}^{N}\times \mathbb{R}_{\tau }\right) $ is said to be $Y\times Z$%
-periodic if for each $\left( k,l\right) \in \mathbb{Z}^{N}\times \mathbb{Z}$
($\mathbb{Z}$ denotes the integers), we have $u\left( y+k,\tau +l\right)
=u\left( y,\tau \right) $ almost everywhere (a.e.) in $\left( y,\tau \right)
\in \mathbb{R}^{N}\times \mathbb{R}$. If in addition $u$ is continuous, then
the preceding equality holds for every $\left( y,\tau \right) \in \mathbb{R}%
^{N}\times \mathbb{R}$. The space of all $Y\times Z$-periodic continuous
complex functions on $\mathbb{R}_{y}^{N}\times \mathbb{R}_{\tau }$ is
denoted by $\mathcal{C}_{per}\left( Y\times Z\right) $; that of all $Y\times
Z$-periodic functions in $L_{loc}^{p}\left( \mathbb{R}_{y}^{N}\times \mathbb{%
R}_{\tau }\right) $ $\left( 1\leq p<\infty \right) $ is denoted by $%
L_{per}^{p}\left( Y\times Z\right) $. $\mathcal{C}_{per}\left( Y\times
Z\right) $ is a Banach space under the supremum norm on $\mathbb{R}%
^{N}\times \mathbb{R}$, whereas $L_{per}^{p}\left( Y\times Z\right) $ is a
Banach space under the norm 
\begin{equation*}
\left\Vert u\right\Vert _{L^{p}\left( Y\times Z\right) }=\left(
\int_{Z}\int_{Y}\left\vert u\left( y,\tau \right) \right\vert ^{p}dyd\tau
\right) ^{\frac{1}{p}}\text{\qquad }\left( u\in L_{per}^{p}\left( Y\times
Z\right) \right) \text{.}
\end{equation*}

We will need the space $H_{\#}^{1}\left( Y\right) $ of $Y$-periodic
functions $u\in H_{loc}^{1}\left( \mathbb{R}_{y}^{N}\right)
=W_{loc}^{1,2}\left( \mathbb{R}_{y}^{N}\right) $ such that $\int_{Y}u\left(
y\right) dy=0$. Provided with the gradient norm, 
\begin{equation*}
\left\Vert u\right\Vert _{H_{\#}^{1}\left( Y\right) }=\left(
\int_{Y}\left\vert \nabla _{y}u\right\vert ^{2}dy\right) ^{\frac{1}{2}}\text{%
\qquad }\left( u\in H_{\#}^{1}\left( Y\right) \right) \text{,}
\end{equation*}%
where $\nabla _{y}u=\left( \frac{\partial u}{\partial y_{1}},...,\frac{%
\partial u}{\partial y_{N}}\right) $, $H_{\#}^{1}\left( Y\right) $ is a
Hilbert space. We will also need the space $L_{per}^{2}\left(
Z;H_{\#}^{1}\left( Y\right) \right) $ with the norm%
\begin{equation*}
\left\Vert u\right\Vert _{L_{per}^{2}\left( Z;H_{\#}^{1}\left( Y\right)
\right) }=\left( \int_{Z}\int_{Y}\left\vert \nabla _{y}u\left( y,\tau
\right) \right\vert ^{2}dyd\tau \right) ^{\frac{1}{2}}\text{ }\left( u\in
L_{per}^{2}\left( Z;H_{\#}^{1}\left( Y\right) \right) \right)
\end{equation*}%
which is a Hilbert space.

Before we can recall the concept of $\Sigma $-convergence, let us introduce
one further notation. The letter $E$ throughout will denote a family of real
numbers $0<\varepsilon <1$ admitting $0$ as an accumulation point. For
example, $E$ may be the whole interval $\left( 0,1\right) $; $E$ may also be
an ordinary sequence $\left( \varepsilon _{n}\right) _{n\in \mathbb{N}}$
with $0<\varepsilon _{n}<1$ and $\varepsilon _{n}\rightarrow 0$ as $%
n\rightarrow \infty $. In the latter case $E$ will be referred to as a 
\textit{fundamental sequence}.

Let $\Omega $ be a bounded open set in $\mathbb{R}_{x}^{N}$ and $Q=\Omega
\times ]0,T[$ with $T\in \mathbb{R}_{+}^{\ast }$, and let $1\leq p<\infty $.

\begin{definition}
\label{def2.1} A sequence $\left( u_{\varepsilon }\right) _{\varepsilon \in
E}\subset L^{p}\left( Q\right) $ is said to:

(i) weakly $\Sigma $-converge in $L^{p}\left( Q\right) $ to some $u_{0}\in
L^{p}\left( Q;L_{per}^{p}\left( Y\times Z\right) \right) $ if as

\noindent $E\ni \varepsilon \rightarrow 0$, 
\begin{equation}
\int_{Q}u_{\varepsilon }\left( x,t\right) \psi ^{\varepsilon }\left(
x,t\right) dxdt\rightarrow \int \int \int_{Q\times Y\times Z}u_{0}\left(
x,t,y,\tau \right) \psi \left( x,t,y,\tau \right) dxdtdyd\tau  \label{eq2.31}
\end{equation}%
\begin{equation*}
\begin{array}{c}
\text{for all }\psi \in L^{p^{\prime }}\left( Q;\mathcal{C}_{per}\left(
Y\times Z\right) \right) \text{ }\left( \frac{1}{p^{\prime }}=1-\frac{1}{p}%
\right) \text{, where }\psi ^{\varepsilon }\left( x,t\right) = \\ 
\psi \left( x,t,\frac{x}{\varepsilon },\frac{t}{\varepsilon }\right) \text{ }%
\left( \left( x,t\right) \in Q\right) \text{;}%
\end{array}%
\end{equation*}

(ii) strongly $\Sigma $-converge in $L^{p}\left( Q\right) $ to some $%
u_{0}\in L^{p}\left( Q;L_{per}^{p}\left( Y\times Z\right) \right) $ if the
following property is verified: 
\begin{equation}
\left\{ 
\begin{array}{c}
\text{Given }\eta >0\text{ and }v\in L^{p}\left( Q;\mathcal{C}_{per}\left(
Y\times Z\right) \right) \text{ with} \\ 
\left\Vert u_{0}-v\right\Vert _{L^{p}\left( Q\times Y\times Z\right) }\leq 
\frac{\eta }{2}\text{, there is some }\alpha >0\text{ such} \\ 
\text{that }\left\Vert u_{\varepsilon }-v^{\varepsilon }\right\Vert
_{L^{p}\left( Q\right) }\leq \eta \text{ provided }E\ni \varepsilon \leq
\alpha \text{.}%
\end{array}%
\right.  \label{eq2.32}
\end{equation}
\end{definition}

We will briefly express weak and strong two-scale convergence by writing $%
u_{\varepsilon }\rightarrow u_{0}$ in $L^{p}\left( Q\right) $-weak $\Sigma $
and $u_{\varepsilon }\rightarrow u_{0}$ in $L^{p}\left( Q\right) $-strong $%
\Sigma $, respectively.

\begin{remark}
\label{rem2.1} It is of interest to know that if $u_{\varepsilon
}\rightarrow u_{0}$ in $L^{p}\left( Q\right) $-weak $\Sigma $, then (\ref%
{eq2.31}) holds for $\psi \in \mathcal{C}\left( \overline{Q};L_{per}^{\infty
}\left( Y\times Z\right) \right) $. See \cite[Proposition 10]{bib16} for the
proof.
\end{remark}

Instead of repeating here the main results underlying two-scale convergence
or $\Sigma $-convergence theory for periodic structures, we find it more
convenient to draw the reader's attention to a few references, see, e.g., 
\cite{bib1}, \cite{bib11}, \cite{bib16} and \cite{bib32}.

However, we recall below two fundamental results. First of all, let 
\begin{equation*}
\mathcal{Y}\left( 0,T\right) =\left\{ v\in L^{2}\left( 0,T;H_{0}^{1}\left(
\Omega ;\mathbb{R}\right) \right) :v^{\prime }\in L^{2}\left(
0,T;H^{-1}\left( \Omega ;\mathbb{R}\right) \right) \right\} \text{.}
\end{equation*}%
$\mathcal{Y}\left( 0,T\right) $ is provided with the norm 
\begin{equation*}
\left\Vert v\right\Vert _{\mathcal{Y}\left( 0,T\right) }=\left( \left\Vert
v\right\Vert _{L^{2}\left( 0,T;H_{0}^{1}\left( \Omega \right) \right)
}^{2}+\left\Vert v^{\prime }\right\Vert _{L^{2}\left( 0,T;H^{-1}\left(
\Omega \right) \right) }^{2}\right) ^{\frac{1}{2}}\qquad \left( v\in 
\mathcal{Y}\left( 0,T\right) \right)
\end{equation*}%
which makes it a Hilbert space.

\begin{theorem}
\label{th2.1} Assume that $1<p<\infty $ and further $E$ is a fundamental
sequence. Let a sequence $\left( u_{\varepsilon }\right) _{\varepsilon \in
E} $ be bounded in $L^{p}\left( Q\right) $. Then, a subsequence $E^{\prime }$
can be extracted from $E$ such that $\left( u_{\varepsilon }\right)
_{\varepsilon \in E^{\prime }}$ weakly $\Sigma $-converges in $L^{p}\left(
Q\right) $.
\end{theorem}

\begin{theorem}
\label{th2.2} Let $E$ be a fundamental sequence. Suppose a sequence $\left(
u_{\varepsilon }\right) _{\varepsilon \in E}$ is bounded in $\mathcal{Y}%
\left( 0,T\right) $. Then, a subsequence $E^{\prime }$ can be extracted from 
$E$ such that, as $E^{\prime }\ni \varepsilon \rightarrow 0$, 
\begin{equation*}
u_{\varepsilon }\rightarrow u_{0}\text{ in }\mathcal{Y}\left( 0,T\right) 
\text{-weak,\qquad \qquad \qquad \qquad \qquad \qquad }
\end{equation*}%
\begin{equation*}
u_{\varepsilon }\rightarrow u_{0}\text{ in }L^{2}\left( Q\right) \text{-weak 
}\Sigma \text{,\qquad \qquad \qquad \qquad \qquad }\quad
\end{equation*}%
\begin{equation*}
\frac{\partial u_{\varepsilon }}{\partial x_{j}}\rightarrow \frac{\partial
u_{0}}{\partial x_{j}}+\frac{\partial u_{1}}{\partial y_{j}}\text{ in }%
L^{2}\left( Q\right) \text{-weak }\Sigma \text{ }\left( 1\leq j\leq N\right) 
\text{,}
\end{equation*}%
where $u_{0}\in \mathcal{Y}\left( 0,T\right) $, $u_{1}\in L^{2}\left(
Q;L_{per}^{2}\left( Z;H_{\#}^{1}\left( Y\right) \right) \right) $.
\end{theorem}

The proof of Theorem \ref{th2.1} can be found in, e.g., \cite{bib1}, \cite%
{bib11}, whereas Theorem \ref{th2.2} has its proof in, e.g., \cite{bib16}
and \cite{bib24}.

\subsubsection{A global homogenization theorem}

Before we can establish a so-called global homogenization theorem for (\ref%
{eq1.3})-(\ref{eq1.6}), we require a few basic notation and results. To
begin, let 
\begin{equation*}
\mathcal{V}_{Y}=\left\{ \mathbf{\psi }\in \mathcal{C}_{per}^{\infty }\left(
Y;\mathbb{R}\right) ^{N}:\int_{Y}\mathbf{\psi }\left( y\right) dy=0,\text{ }%
div_{y}\mathbf{\psi =}0\right\} \text{, }
\end{equation*}%
\begin{equation*}
V_{Y}=\left\{ \mathbf{w}\in H_{\#}^{1}\left( Y;\mathbb{R}\right) ^{N}:div_{y}%
\mathbf{w=}0\right\} \text{, }
\end{equation*}%
where: $\mathcal{C}_{per}^{\infty }\left( Y;\mathbb{R}\right) =\mathcal{C}%
^{\infty }\left( \mathbb{R}^{N};\mathbb{R}\right) \cap \mathcal{C}%
_{per}\left( Y\right) $, $div_{y}$ denotes the divergence operator in $%
\mathbb{R}_{y}^{N}$. We provide $V_{Y}$ with the $H_{\#}^{1}\left( Y\right)
^{N}$-norm, which makes it a Hilbert space. There is no difficulty in
verifying that $\mathcal{V}_{Y}$ is dense in $V_{Y}$ (proceed as in \cite[%
Proposition 3.2]{bib21}). With this in mind, set 
\begin{equation*}
\mathbb{F}_{0}^{1}=L^{2}\left( 0,T;V\right) \times L^{2}\left(
Q;L_{per}^{2}\left( Z;V_{Y}\right) \right) \text{.}
\end{equation*}%
This is a Hilbert space with norm 
\begin{equation*}
\left\Vert \mathbf{v}\right\Vert _{\mathbb{F}_{0}^{1}}=\left( \left\Vert 
\mathbf{v}_{0}\right\Vert _{L^{2}\left( 0,T;V\right) }^{2}+\left\Vert 
\mathbf{v}_{1}\right\Vert _{L^{2}\left( Q;L_{per}^{2}\left( Z;V_{Y}\right)
\right) }^{2}\right) ^{\frac{1}{2}}\text{, }\mathbf{v=}\left( \mathbf{v}_{0},%
\mathbf{v}_{1}\right) \in \mathbb{F}_{0}^{1}\text{.}
\end{equation*}%
On the other hand, put 
\begin{equation*}
\mathbf{\tciFourier }_{0}^{\infty }=\mathcal{D}\left( 0,T;\mathcal{V}\right)
\times \left[ \mathcal{D}\left( Q;\mathbb{R}\right) \otimes \left[ \mathcal{C%
}_{per}^{\infty }\left( Z;\mathbb{R}\right) \otimes \mathcal{V}_{Y}\right] %
\right] \text{,}
\end{equation*}%
where $\mathcal{C}_{per}^{\infty }\left( Z;\mathbb{R}\right) =\mathcal{C}%
^{\infty }\left( \mathbb{R};\mathbb{R}\right) \cap \mathcal{C}_{per}\left(
Z\right) $, $\mathcal{C}_{per}^{\infty }\left( Z;\mathbb{R}\right) \otimes 
\mathcal{V}_{Y}$ stands for the space of vector functions $\mathbf{w}$ on $%
\mathbb{R}_{y}^{N}\times \mathbb{R}_{\tau }$ of the form%
\begin{equation*}
\mathbf{w}\left( y,\tau \right) =\sum_{finite}\chi _{i}\left( \tau \right) 
\mathbf{v}_{i}\left( y\right) \text{ \ }\left( \tau \in \mathbb{R},\text{ }%
y\in \mathbb{R}^{N}\right)
\end{equation*}%
with $\chi _{i}\in \mathcal{C}_{per}^{\infty }\left( Z;\mathbb{R}\right) $, $%
\mathbf{v}_{i}\in \mathcal{V}_{Y}$, and where $\mathcal{D}\left( Q;\mathbb{R}%
\right) \otimes \left( \mathcal{C}_{per}^{\infty }\left( Z;\mathbb{R}\right)
\otimes \mathcal{V}_{Y}\right) $ is the space of vector functions on $%
Q\times \mathbb{R}_{y}^{N}\times \mathbb{R}$ of the form 
\begin{equation*}
\mathbf{\psi }\left( x,t,y,\tau \right) =\sum_{finite}\varphi _{i}\left(
x,t\right) \mathbf{w}_{i}\left( y,\tau \right) \text{ }\left( \left(
x,t\right) \in Q,\text{ }\left( y,\tau \right) \in \mathbb{R}^{N}\times 
\mathbb{R}\right)
\end{equation*}%
with $\varphi _{i}\in \mathcal{D}\left( Q;\mathbb{R}\right) $, $\mathbf{w}%
_{i}\in \mathcal{C}_{per}^{\infty }\left( Z;\mathbb{R}\right) \otimes 
\mathcal{V}_{Y}$. Since $\mathcal{V}$ is dense in $V$ (see \cite[p.18]{bib31}%
), it is clear that $\mathbf{\tciFourier }_{0}^{\infty }$ is dense in $%
\mathbb{F}_{0}^{1}$.

Now, let 
\begin{equation*}
\mathbb{U}=V\times L^{2}\left( \Omega ;L_{per}^{2}\left( Z;V_{Y}\right)
\right) \text{.}
\end{equation*}%
Provided with the norm%
\begin{equation*}
\left\Vert \mathbf{v}\right\Vert _{\mathbb{U}}=\left( \left\Vert \mathbf{v}%
_{0}\right\Vert ^{2}+\left\Vert \mathbf{v}_{1}\right\Vert _{L^{2}\left(
\Omega ;L_{per}^{2}\left( Z;V_{Y}\right) \right) }^{2}\right) ^{\frac{1}{2}%
}\qquad \left( \mathbf{v}=\left( \mathbf{v}_{0},\mathbf{v}_{1}\right) \in 
\mathbb{U}\right) \text{,}
\end{equation*}%
$\mathbb{U}$ is a Hilbert space. Let us set%
\begin{equation*}
\widehat{a}_{\Omega }\left( \mathbf{u},\mathbf{v}\right)
=\sum_{i,j,k=1}^{N}\int \int \int_{\Omega \times Y\times Z}a_{ij}\left( 
\frac{\partial u_{0}^{k}}{\partial x_{j}}+\frac{\partial u_{1}^{k}}{\partial
y_{j}}\right) \left( \frac{\partial v_{0}^{k}}{\partial x_{i}}+\frac{%
\partial v_{1}^{k}}{\partial y_{i}}\right) dxdyd\tau
\end{equation*}%
for $\mathbf{u=}\left( \mathbf{u}_{0},\mathbf{u}_{1}\right) $ and $\mathbf{v=%
}\left( \mathbf{v}_{0},\mathbf{v}_{1}\right) $ in $\mathbb{U}$. This defines
a symmetric continuous bilinear form $\widehat{a}_{\Omega }$ on $\mathbb{U}%
\times \mathbb{U}$. Furthermore, $\widehat{a}_{\Omega }$ is $\mathbb{U}$%
-elliptic. Specifically, 
\begin{equation}
\widehat{a}_{\Omega }\left( \mathbf{u},\mathbf{u}\right) \geq \alpha
\left\Vert \mathbf{u}\right\Vert _{\mathbb{U}}^{2}\text{ }\left( \mathbf{u}%
\in \mathbb{U}\right)  \label{eq2.33}
\end{equation}%
as is easily checked by using (\ref{eq1.2}) and the fact that $\int_{Y}\frac{%
\partial u_{1}^{k}}{\partial y_{j}}\left( x,y,\tau \right) dy=0$.

Here is one fundamental lemma.

\begin{lemma}
\label{lem2.3} Suppose $N=2$. Suppose also that there exists a function $%
\mathbf{u}=\left( \mathbf{u}_{0},\mathbf{u}_{1}\right) \in \mathbb{F}%
_{0}^{1} $ verifying 
\begin{equation}
\mathbf{u}_{0}\in \mathcal{W}\left( 0,T\right) \text{ with }\mathbf{u}%
_{0}\left( 0\right) =0\text{,}  \label{eq2.34}
\end{equation}%
and the variational equation%
\begin{equation}
\left\{ 
\begin{array}{c}
\int_{0}^{T}\left( \mathbf{u}_{0}^{\prime }\left( t\right) ,\mathbf{v}%
_{0}\left( t\right) \right) dt+\int_{0}^{T}\widehat{a}_{\Omega }\left( 
\mathbf{u}\left( t\right) ,\mathbf{v}\left( t\right) \right)
dt+\int_{0}^{T}b\left( \mathbf{u}_{0}\left( t\right) ,\mathbf{u}_{0}\left(
t\right) ,\mathbf{v}_{0}\left( t\right) \right) dt \\ 
=\int_{0}^{T}\left( \mathbf{f}\left( t\right) ,\mathbf{v}_{0}\left( t\right)
\right) dt\text{ for all }\mathbf{v}=\left( \mathbf{v}_{0},\mathbf{v}%
_{1}\right) \in \mathbb{F}_{0}^{1}\text{.}%
\end{array}%
\right.  \label{eq2.35}
\end{equation}%
Then $\mathbf{u}$ is unique.
\end{lemma}

\begin{proof}
Let $\mathbf{v}_{\ast }=\left( \mathbf{v}_{0},\mathbf{v}_{1}\right) \in 
\mathbb{U}$ and $\varphi \in \mathcal{D}\left( ]0,T[\right) $. By taking $%
\mathbf{v=}\varphi \otimes \mathbf{v}_{\ast }$ in (\ref{eq2.35}), we arrive
at 
\begin{equation}
\left( \mathbf{u}_{0}^{\prime }\left( t\right) ,\mathbf{v}_{0}\right) +%
\widehat{a}_{\Omega }\left( \mathbf{u}\left( t\right) ,\mathbf{v}_{\ast
}\right) +b\left( \mathbf{u}_{0}\left( t\right) ,\mathbf{u}_{0}\left(
t\right) ,\mathbf{v}_{0}\right) =\left( \mathbf{f}\left( t\right) ,\mathbf{v}%
_{0}\right) \text{\qquad }\left( \mathbf{v}_{\ast }\in \mathbb{U}\right)
\label{eq2.36}
\end{equation}%
for almost all $t\in \left( 0,T\right) $. Suppose that $\mathbf{u}_{\ast }$
and\textbf{\ }$\mathbf{u}_{\ast \ast }$ are two solutions of (\ref{eq2.34})-(%
\ref{eq2.35}) with $\mathbf{u}_{\ast }=\left( \mathbf{u}_{\ast 0},\mathbf{u}%
_{\ast 1}\right) $ and $\mathbf{u}_{\ast \ast }=\left( \mathbf{u}_{\ast \ast
0},\mathbf{u}_{\ast \ast 1}\right) $. Let $\mathbf{u=u}_{\ast }-\mathbf{u}%
_{\ast \ast }=\left( \mathbf{u}_{0},\mathbf{u}_{1}\right) $ with $\mathbf{u}%
_{0}=\mathbf{u}_{\ast 0}-\mathbf{u}_{\ast \ast 0}$ and $\mathbf{u}_{1}=%
\mathbf{u}_{\ast 1}-\mathbf{u}_{\ast \ast 1}$. Let us show that $\mathbf{u=}%
0 $. Using (\ref{eq2.36}), we see that $\mathbf{u}$ verifies: 
\begin{equation}
\left( \mathbf{u}_{0}^{\prime }\left( t\right) ,\mathbf{v}_{0}\right) +%
\widehat{a}_{\Omega }\left( \mathbf{u}\left( t\right) ,\mathbf{v}_{\ast
}\right) +b\left( \mathbf{u}_{\ast 0}\left( t\right) ,\mathbf{u}_{0}\left(
t\right) ,\mathbf{v}_{0}\right) +b\left( \mathbf{u}_{0}\left( t\right) ,%
\mathbf{u}_{\ast \ast 0}\left( t\right) ,\mathbf{v}_{0}\right) =0
\label{eq2.37}
\end{equation}%
for all $\mathbf{v}_{\ast }\in \mathbb{U}$ and for almost all $t\in \left(
0,T\right) $. But, by virtue of \cite[p. 261]{bib31} 
\begin{equation}
\frac{d}{dt}\left\vert \mathbf{u}_{0}\left( t\right) \right\vert
^{2}=2\left( \mathbf{u}_{0}^{\prime }\left( t\right) ,\mathbf{u}_{0}\left(
t\right) \right) \text{\qquad }  \label{eq2.38}
\end{equation}%
for almost all $t\in \left( 0,T\right) $. Hence, taking $\mathbf{v}_{\ast }%
\mathbf{=u}\left( t\right) $ in (\ref{eq2.37}), we obtain by\ (\ref{eq2.4}),
(\ref{eq2.5}) and (\ref{eq2.33}) 
\begin{equation}
\frac{d}{dt}\left\vert \mathbf{u}_{0}\left( t\right) \right\vert
^{2}+2\alpha \left\Vert \mathbf{u}\left( t\right) \right\Vert _{\mathbb{U}%
}^{2}\leq -2b\left( \mathbf{u}_{0}\left( t\right) ,\mathbf{u}_{\ast \ast
0}\left( t\right) ,\mathbf{u}_{0}\left( t\right) \right) \text{\qquad }
\label{eq2.39}
\end{equation}%
for almost all $t\in \left( 0,T\right) $. Furthermore, using part (\ref%
{eq2.7}) of Lemma \ref{lem2.1} yields 
\begin{equation*}
2\left\vert b\left( \mathbf{u}_{0}\left( t\right) ,\mathbf{u}_{\ast \ast
0}\left( t\right) ,\mathbf{u}_{0}\left( t\right) \right) \right\vert \leq 2^{%
\frac{3}{2}}\left\vert \mathbf{u}_{0}\left( t\right) \right\vert \left\Vert 
\mathbf{u}_{0}\left( t\right) \right\Vert \left\Vert \mathbf{u}_{\ast \ast
0}\left( t\right) \right\Vert \text{.}
\end{equation*}%
Thus, by the inequality 
\begin{equation*}
2^{\frac{3}{2}}\left\vert \mathbf{u}_{0}\left( t\right) \right\vert
\left\Vert \mathbf{u}_{0}\left( t\right) \right\Vert \left\Vert \mathbf{u}%
_{\ast \ast 0}\left( t\right) \right\Vert \leq 2\alpha \left\Vert \mathbf{u}%
_{0}\left( t\right) \right\Vert ^{2}+\frac{1}{\alpha }\left\vert \mathbf{u}%
_{0}\left( t\right) \right\vert ^{2}\left\Vert \mathbf{u}_{\ast \ast
0}\left( t\right) \right\Vert ^{2}
\end{equation*}%
we have: 
\begin{equation*}
\frac{d}{dt}\left\vert \mathbf{u}_{0}\left( t\right) \right\vert
^{2}+2\alpha \left\Vert \mathbf{u}\left( t\right) \right\Vert _{\mathbb{U}%
}^{2}\leq 2\alpha \left\Vert \mathbf{u}\left( t\right) \right\Vert _{\mathbb{%
U}}^{2}+\frac{1}{\alpha }\left\vert \mathbf{u}_{0}\left( t\right)
\right\vert ^{2}\left\Vert \mathbf{u}_{\ast \ast 0}\left( t\right)
\right\Vert ^{2}
\end{equation*}%
for almost all $t\in \left( 0,T\right) $, and then 
\begin{equation*}
\frac{d}{dt}\left\vert \mathbf{u}_{0}\left( t\right) \right\vert ^{2}-\frac{1%
}{\alpha }\left\vert \mathbf{u}_{0}\left( t\right) \right\vert
^{2}\left\Vert \mathbf{u}_{\ast \ast 0}\left( t\right) \right\Vert ^{2}\leq 0
\end{equation*}%
for almost all $t\in \left( 0,T\right) $. Multiplying the preceding
inequality by

\noindent $\exp \left( -\frac{1}{\alpha }\int_{0}^{t}\left\Vert \mathbf{u}%
_{\ast \ast 0}\left( s\right) \right\Vert ^{2}ds\right) $, we obtain%
\begin{equation*}
\frac{d}{dt}\left[ \left\vert \mathbf{u}_{0}\left( t\right) \right\vert
^{2}\exp \left( -\frac{1}{\alpha }\int_{0}^{t}\left\Vert \mathbf{u}_{\ast
\ast 0}\left( s\right) \right\Vert ^{2}ds\right) \right] \leq 0
\end{equation*}%
for almost all $t\in \left( 0,T\right) $. As $\mathbf{u}_{0}\left( 0\right)
=0$, integrating on $\left[ 0,t\right] $, $\left( 0\leq t\leq T\right) $ the
preceding inequality yields: $\left\vert \mathbf{u}_{0}\left( t\right)
\right\vert ^{2}\leq 0$ for all $t\in \left( 0,T\right) $, thus $\mathbf{u}%
_{0}\left( t\right) =0$\quad $\left( t\in \left( 0,T\right) \right) $.
Finally, the inequality (\ref{eq2.39}) gives $\mathbf{u}=0$, \ and the lemma
follows.
\end{proof}

We are now able to prove the desired theorem. Throughout the remainder of
the present section, it is assumed that $a_{ij}$ is $Y$-periodic for any $%
1\leq i,j\leq N$.

\begin{theorem}
\label{th2.3} Suppose that the hypotheses of Lemma \ref{lem2.2} are
satisfied. For $0<\varepsilon <1$, let $\mathbf{u}_{\varepsilon }$ be
defined by (\ref{eq1.3})-(\ref{eq1.6}). Then, as $\varepsilon \rightarrow 0$
we have%
\begin{equation}
\mathbf{u}_{\varepsilon }\rightarrow \mathbf{u}_{0}\text{ in }\mathcal{W}%
\left( 0,T\right) \text{-weak,\qquad \qquad \qquad \qquad \qquad \qquad }
\label{eq2.40}
\end{equation}%
\begin{equation}
\frac{\partial u_{\varepsilon }^{k}}{\partial x_{j}}\rightarrow \frac{%
\partial u_{0}^{k}}{\partial x_{j}}+\frac{\partial u_{1}^{k}}{\partial y_{j}}%
\text{ in }L^{2}\left( Q\right) \text{-weak }\Sigma \text{ }\left( 1\leq
j,k\leq N\right)  \label{eq2.41}
\end{equation}%
where $\mathbf{u=}\left( \mathbf{u}_{0},\mathbf{u}_{1}\right) $ (with $%
\mathbf{u}_{0}=\left( u_{0}^{k}\right) $ and $\mathbf{u}_{1}=\left(
u_{1}^{k}\right) $) is the unique solution of (\ref{eq2.34})-(\ref{eq2.35}).
\end{theorem}

\begin{proof}
By Proposition \ref{pr2.1}, we see that the sequences $\left( p_{\varepsilon
}\right) _{0<\varepsilon <1}$ and $\left( \mathbf{u}_{\varepsilon }\right)
_{0<\varepsilon <1}=\left( u_{\varepsilon }^{1},...,u_{\varepsilon
}^{N}\right) _{0<\varepsilon <1}$ are bounded respectively in $L^{2}\left(
Q\right) $ and $\mathcal{W}\left( 0,T\right) $. Further, it follows from (%
\ref{eq2.11}) and (\ref{eq2.12}) that for $1\leq k\leq N$, the sequence $%
\left( u_{\varepsilon }^{k}\right) _{0<\varepsilon <1}$ is bounded in $%
\mathcal{Y}\left( 0,T\right) $. Let $E$ be a fundamental sequence. Then, by
Theorems \ref{th2.1}-\ref{th2.2} and the fact that $\mathcal{W}\left(
0,T\right) $ is compactly embedded in $L^{2}\left( Q\right) ^{N}$, there
exist a subsequence $E^{\prime }$ extracted from $E$ and functions $\mathbf{u%
}_{0}=\left( u_{0}^{k}\right) _{1\leq k\leq N}\in \mathcal{W}\left(
0,T\right) $, $\mathbf{u}_{1}=\left( u_{1}^{k}\right) _{1\leq k\leq N}\in
L^{2}\left( Q;L_{per}^{2}\left( Z;H_{\#}^{1}\left( Y;\mathbb{R}\right)
^{N}\right) \right) $, and $p\in L^{2}\left( Q;L_{per}^{2}\left( Y\times Z;%
\mathbb{R}\right) \right) $ such that as $E^{\prime }\ni \varepsilon
\rightarrow 0$, we have (\ref{eq2.40})-(\ref{eq2.41}) and 
\begin{equation}
\mathbf{u}_{\varepsilon }\rightarrow \mathbf{u}_{0}\text{ in }L^{2}\left(
Q\right) ^{N}\text{-strong,}  \label{eq2.42}
\end{equation}%
\begin{equation}
p_{\varepsilon }\rightarrow p\text{ in }L^{2}\left( Q\right) \text{-weak }%
\Sigma \text{.}  \label{eq2.43}
\end{equation}%
But, by virtue of Lemma \ref{lem2.3}, the theorem will be entirely proved if
we show that $\mathbf{u=}\left( \mathbf{u}_{0},\mathbf{u}_{1}\right) $
verifies (\ref{eq2.35}). Indeed, according to (\ref{eq1.4}), we have $div%
\mathbf{u}_{0}=0$ and $div_{y}\mathbf{u}_{1}=0$. Therefore $\mathbf{u=}%
\left( \mathbf{u}_{0},\mathbf{u}_{1}\right) \in \mathbb{F}_{0}^{1}$. Let us
recall that $\mathbf{u}_{0}$ can be considered as a continuous function of $%
\left[ 0,T\right] $ into $H$ since $\mathcal{W}\left( 0,T\right) $ is
continuously embedded in $\mathcal{C}\left( \left[ 0,T\right] ;H\right) $.
Let us show that $\mathbf{u}_{0}\left( 0\right) =0$. For $\mathbf{v}\in V$
and $\varphi \in \mathcal{C}^{1}\left( \left[ 0,T\right] \right) $ with $%
\varphi \left( T\right) =0$ and $\varphi \left( 0\right) =1$, we have by an
integration by part 
\begin{equation*}
\int_{0}^{T}\left( \mathbf{u}_{\varepsilon }^{\prime }\left( t\right) ,%
\mathbf{v}\right) \varphi \left( t\right) dt+\int_{0}^{T}\left( \mathbf{u}%
_{\varepsilon }\left( t\right) ,\mathbf{v}\right) \varphi ^{\prime }\left(
t\right) dt=-\left( \mathbf{u}_{\varepsilon }\left( 0\right) ,\mathbf{v}%
\right) \text{.}
\end{equation*}%
According to (\ref{eq1.6}), we have by passing to the limit in the preceding
equality as $E^{\prime }\ni \varepsilon \rightarrow 0$%
\begin{equation*}
\int_{0}^{T}\left( \mathbf{u}_{0}^{\prime }\left( t\right) ,\mathbf{v}%
\right) \varphi \left( t\right) dt+\int_{0}^{T}\left( \mathbf{u}_{0}\left(
t\right) ,\mathbf{v}\right) \varphi ^{\prime }\left( t\right) dt=0\text{.}
\end{equation*}%
Hence $\left( \mathbf{u}_{0}\left( 0\right) ,\mathbf{v}\right) =0$ for all $%
\mathbf{v}\in V$, and as $V$ is dense in $H$ we conclude that $\mathbf{u}%
_{0}\left( 0\right) =0$. Now, let us check that $\mathbf{u=}\left( \mathbf{u}%
_{0},\mathbf{u}_{1}\right) $ verifies the variational equation of (\ref%
{eq2.35}). For $0<\varepsilon <1$, let%
\begin{equation}
\begin{array}{c}
\mathbf{\Phi }_{\varepsilon }=\mathbf{\psi }_{0}+\varepsilon \mathbf{\psi }%
_{1}^{\varepsilon }\text{ with }\mathbf{\psi }_{0}\in \mathcal{D}\left( Q;%
\mathbb{R}\right) ^{N}\text{ and } \\ 
\mathbf{\psi }_{1}\in \mathcal{D}\left( Q;\mathbb{R}\right) \otimes \left[ 
\mathcal{C}_{per}^{\infty }\left( Z;\mathbb{R}\right) \otimes \mathcal{V}_{Y}%
\right] \text{,}%
\end{array}
\label{eq2.44}
\end{equation}%
i.e., $\mathbf{\Phi }_{\varepsilon }\left( x,t\right) =\mathbf{\psi }%
_{0}\left( x,t\right) +\varepsilon \mathbf{\psi }_{1}\left( x,t,\frac{x}{%
\varepsilon },\frac{t}{\varepsilon }\right) $ for $\left( x,t\right) \in Q$.
We have $\mathbf{\Phi }_{\varepsilon }\in \mathcal{D}\left( Q;\mathbb{R}%
\right) ^{N}$. Thus, multiplying (\ref{eq1.3}) by $\mathbf{\Phi }%
_{\varepsilon }$ yields 
\begin{equation}
\begin{array}{c}
\int_{0}^{T}\left( \mathbf{u}_{\varepsilon }^{\prime }\left( t\right) ,%
\mathbf{\Phi }_{\varepsilon }\left( t\right) \right)
dt+\int_{0}^{T}a^{\varepsilon }\left( \mathbf{u}_{\varepsilon }\left(
t\right) ,\mathbf{\Phi }_{\varepsilon }\left( t\right) \right)
dt+\int_{0}^{T}b\left( \mathbf{u}_{\varepsilon }\left( t\right) ,\mathbf{u}%
_{\varepsilon }\left( t\right) ,\mathbf{\Phi }_{\varepsilon }\left( t\right)
\right) dt \\ 
-\int_{Q}p_{\varepsilon }div\mathbf{\Phi }_{\varepsilon
}dxdt=\int_{0}^{T}\left( \mathbf{f}\left( t\right) ,\mathbf{\Phi }%
_{\varepsilon }\left( t\right) \right) dt\text{.}%
\end{array}
\label{eq2.45}
\end{equation}%
Let us note at once that 
\begin{equation*}
\int_{0}^{T}\left( \mathbf{u}_{\varepsilon }^{\prime }\left( t\right) ,%
\mathbf{\Phi }_{\varepsilon }\left( t\right) \right)
dt=-\sum_{l=1}^{N}\int_{Q}u_{\varepsilon }^{l}\left[ \frac{\partial \psi
_{0}^{l}}{\partial t}+\varepsilon \left( \frac{\partial \psi _{1}^{l}}{%
\partial t}\right) ^{\varepsilon }+\left( \frac{\partial \psi _{1}^{l}}{%
\partial \tau }\right) ^{\varepsilon }\right] dxdt\text{.}
\end{equation*}%
Then by virtue of (\ref{eq2.42}) we have%
\begin{equation}
\int_{0}^{T}\left( \mathbf{u}_{\varepsilon }^{\prime }\left( t\right) ,%
\mathbf{\Phi }_{\varepsilon }\left( t\right) \right) dt\rightarrow
-\sum_{l=1}^{N}\int_{Q}u_{0}^{l}\frac{\partial \psi _{0}^{l}}{\partial t}%
dxdt=\int_{0}^{T}\left( \mathbf{u}_{0}^{\prime }\left( t\right) ,\mathbf{%
\psi }_{0}\left( t\right) \right) dt  \label{eq2.46}
\end{equation}%
as $E^{\prime }\ni \varepsilon \rightarrow 0$. In fact, on one hand 
\begin{equation*}
\sum_{l=1}^{N}\int_{Q}u_{\varepsilon }^{l}\left[ \frac{\partial \psi _{0}^{l}%
}{\partial t}+\varepsilon \left( \frac{\partial \psi _{1}^{l}}{\partial t}%
\right) ^{\varepsilon }+\left( \frac{\partial \psi _{1}^{l}}{\partial \tau }%
\right) ^{\varepsilon }\right] dxdt
\end{equation*}%
\begin{equation*}
\rightarrow \sum_{l=1}^{N}\left[ \int_{Q}u_{0}^{l}\frac{\partial \psi
_{0}^{l}}{\partial t}dxdt+\int \int \int_{Q\times Y\times Z}u_{0}^{l}\frac{%
\partial \psi _{1}^{l}}{\partial \tau }dxdtdyd\tau \right]
\end{equation*}%
as $E^{\prime }\ni \varepsilon \rightarrow 0$, on the other hand

$\int \int \int_{Q\times Y\times Z}u_{0}^{l}\frac{\partial \psi _{1}^{l}}{%
\partial \tau }dxdtdyd\tau =\int_{Q}u_{0}^{l}\left( \int \int_{Y\times Z}%
\frac{\partial \psi _{1}^{l}}{\partial \tau }dyd\tau \right) dxdt=0$ by
virtue of the $Y\times Z$-periodicity. The next point is to pass to the
limit in (\ref{eq2.45}) as $E^{\prime }\ni \varepsilon \rightarrow 0$. To
this end, we note that as $E^{\prime }\ni \varepsilon \rightarrow 0$, 
\begin{equation*}
\int_{0}^{T}a^{\varepsilon }\left( \mathbf{u}_{\varepsilon }\left( t\right) ,%
\mathbf{\Phi }_{\varepsilon }\left( t\right) \right) dt\rightarrow
\int_{0}^{T}\widehat{a}_{\Omega }\left( \mathbf{u}\left( t\right) ,\mathbf{%
\Phi }\left( t\right) \right) dt\text{,}
\end{equation*}%
where $\mathbf{\Phi }=\left( \mathbf{\psi }_{0},\mathbf{\psi }_{1}\right) $
(proceed as in the proof of the analogous result in \cite[p.179]{bib19}).
Further, in view of (\ref{eq2.42}) and (\ref{eq2.41}), it follows from \cite[%
Proposition 4.7]{bib17} (see also \cite{bib16}) that 
\begin{equation*}
\int_{0}^{T}b\left( \mathbf{u}_{\varepsilon }\left( t\right) ,\mathbf{u}%
_{\varepsilon }\left( t\right) ,\mathbf{\Phi }_{\varepsilon }\left( t\right)
\right) dt\rightarrow \int_{0}^{T}b\left( \mathbf{u}_{0}\left( t\right) ,%
\mathbf{u}_{0}\left( t\right) ,\mathbf{\psi }_{0}\left( t\right) \right) dt
\end{equation*}%
as $E^{\prime }\ni \varepsilon \rightarrow 0$. Now, based on (\ref{eq2.43}),
there is no difficulty in showing that as $E^{\prime }\ni \varepsilon
\rightarrow 0$,%
\begin{equation*}
\int_{Q}p_{\varepsilon }div\mathbf{\Phi }_{\varepsilon }dxdt\rightarrow \int
\int \int_{Q\times Y\times Z}pdiv\mathbf{\psi }_{0}dxdtdyd\tau \text{.}
\end{equation*}%
On the other hand, let us check that as $\varepsilon \rightarrow 0$%
\begin{equation}
\int_{0}^{T}\left( \mathbf{f}\left( t\right) ,\mathbf{\Phi }_{\varepsilon
}\left( t\right) \right) dt\rightarrow \int_{0}^{T}\left( \mathbf{f}\left(
t\right) ,\mathbf{\psi }_{0}\left( t\right) \right) dt\text{.}
\label{eq2.47}
\end{equation}%
Indeed, if $\mathbf{f}\in L^{2}\left( 0,T;L^{2}\left( \Omega ;\mathbb{R}%
\right) ^{N}\right) $ (\ref{eq2.47}) is immediate by using the classical
fact that $\mathbf{\Phi }_{\varepsilon }\rightarrow \mathbf{\psi }_{0}$ in $%
L^{2}\left( Q\right) ^{N}$-weak and $\frac{\partial \mathbf{\Phi }%
_{\varepsilon }}{\partial x_{j}}\rightarrow \frac{\partial \mathbf{\psi }_{0}%
}{\partial x_{j}}$ in $L^{2}\left( Q\right) ^{N}$-weak $\left( 1\leq j\leq
N\right) $ as $\varepsilon \rightarrow 0$. In the general case, (\ref{eq2.47}%
) follows by the density of $L^{2}\left( 0,T;L^{2}\left( \Omega ;\mathbb{R}%
\right) ^{N}\right) $ in $L^{2}\left( 0,T;H^{-1}\left( \Omega ;\mathbb{R}%
\right) ^{N}\right) $.

Having made this point, we can pass to the limit in (\ref{eq2.45}) when $%
E^{\prime }\ni \varepsilon \rightarrow 0$, and the result is that%
\begin{equation}
\begin{array}{c}
\int_{0}^{T}\left( \mathbf{u}_{0}^{\prime }\left( t\right) ,\mathbf{\psi }%
_{0}\left( t\right) \right) dt+\int_{0}^{T}\widehat{a}_{\Omega }\left( 
\mathbf{u}\left( t\right) ,\mathbf{\Phi }\left( t\right) \right)
dt+\int_{0}^{T}b\left( \mathbf{u}_{0}\left( t\right) ,\mathbf{u}_{0}\left(
t\right) ,\mathbf{\psi }_{0}\left( t\right) \right) dt \\ 
-\int_{Q}p_{0}div\mathbf{\psi }_{0}dxdt=\int_{0}^{T}\left( \mathbf{f}\left(
t\right) ,\mathbf{\psi }_{0}\left( t\right) \right) dt\text{,}%
\end{array}
\label{eq2.48}
\end{equation}%
where $p_{0}$ denotes the mean of $p$, i.e., $p_{0}\in L^{2}\left(
0,T;L^{2}\left( \Omega ;\mathbb{R}\right) \right) $ and $p_{0}\left(
x,t\right) =\int \int_{Y\times Z}p\left( x,t,y,\tau \right) dyd\tau $ a.e.
in $\left( x,t\right) \in Q$, and where $\mathbf{\Phi }=\left( \mathbf{\psi }%
_{0},\mathbf{\psi }_{1}\right) $, $\mathbf{\psi }_{0}$ ranging over $%
\mathcal{D}\left( Q;\mathbb{R}\right) ^{N}$ and $\mathbf{\psi }_{1}$ ranging
over $\mathcal{D}\left( Q;\mathbb{R}\right) \otimes \left[ \mathcal{C}%
_{per}^{\infty }\left( Z;\mathbb{R}\right) \otimes \mathcal{V}_{Y}\right] $.
Taking in particular $\mathbf{\psi }_{0}$ in $\mathcal{D}\left( 0,T;\mathcal{%
V}\right) $ and using the density of $\mathbf{\tciFourier }_{0}^{\infty }$
in $\mathbb{F}_{0}^{1}$, one quickly arrives at (\ref{eq2.35}). The unicity
of $\mathbf{u=}\left( \mathbf{u}_{0},\mathbf{u}_{1}\right) $ follows by
Lemma \ref{lem2.3}. Consequently, (\ref{eq2.40}) and (\ref{eq2.41}) still
hold when $E\ni \varepsilon \rightarrow 0$. Hence when $0<\varepsilon
\rightarrow 0$, by virtue of the arbitrariness of $E$. The theorem is proved.
\end{proof}

Now, we wish to give a simple representation of the vector function $\mathbf{%
u}_{1}$ in Theorem \ref{th2.3} for further uses. For this purpose we
introduce the bilinear form $\widehat{a}$ on $L_{per}^{2}\left(
Z;V_{Y}\right) \times L_{per}^{2}\left( Z;V_{Y}\right) $ defined by%
\begin{equation*}
\widehat{a}\left( \mathbf{u},\mathbf{v}\right) =\sum_{i,j,k=1}^{N}\int
\int_{Y\times Z}a_{ij}\frac{\partial u^{k}}{\partial y_{j}}\frac{\partial
v^{k}}{\partial y_{i}}dyd\tau
\end{equation*}%
for $\mathbf{u=}\left( u^{k}\right) $ and $\mathbf{v=}\left( v^{k}\right)
\in L_{per}^{2}\left( Z;V_{Y}\right) $. Next, for each pair of indices $%
1\leq i,k\leq N$, we consider the variational problem%
\begin{equation}
\left\{ 
\begin{array}{c}
\mathbf{\chi }_{ik}\in L_{per}^{2}\left( Z;V_{Y}\right) :\qquad \qquad \qquad
\\ 
\widehat{a}\left( \mathbf{\chi }_{ik},\mathbf{w}\right)
=\sum_{l=1}^{N}\int_{Y\times Z}a_{li}\frac{\partial w^{k}}{\partial y_{l}}%
dyd\tau \\ 
\text{for all }\mathbf{w}=\left( w^{j}\right) \in L_{per}^{2}\left(
Z;V_{Y}\right) \text{,}%
\end{array}%
\right.  \label{eq2.49}
\end{equation}%
which determines $\mathbf{\chi }_{ik}$ in a unique manner.

\begin{lemma}
\label{lem2.4} Under the hypotheses and notation of Theorem \ref{th2.3}, we
have%
\begin{equation}
\mathbf{u}_{1}\left( x,t,y,\tau \right) =-\sum_{i,k=1}^{N}\frac{\partial
u_{0}^{k}}{\partial x_{i}}\left( x,t\right) \mathbf{\chi }_{ik}\left( y,\tau
\right)  \label{eq2.50}
\end{equation}%
almost everywhere in $\left( x,t,y,\tau \right) \in Q\times Y\times Z$.
\end{lemma}

\begin{proof}
In (\ref{eq2.35}), we choose the test functions $\mathbf{v}=\left( \mathbf{v}%
_{0},\mathbf{v}_{1}\right) $ such that $\mathbf{v}_{0}=0$ and $\mathbf{v}%
_{1}\left( x,t,y,\tau \right) =\varphi \left( x,t\right) \mathbf{w}\left(
y,\tau \right) $ for $\left( x,t,y,\tau \right) \in Q\times Y\times Z$,
where $\varphi \in \mathcal{D}\left( Q;\mathbb{R}\right) $ and $\mathbf{w}%
\in L_{per}^{2}\left( Z;V_{Y}\right) $. Then for almost every $\left(
x,t\right) $ in $Q$, we have%
\begin{equation}
\left\{ 
\begin{array}{c}
\widehat{a}\left( \mathbf{u}_{1}\left( x,t\right) ,\mathbf{w}\right)
=-\sum_{l,j,k=1}^{N}\frac{\partial u_{0}^{k}}{\partial x_{j}}\left(
x,t\right) \int \int_{Y\times Z}a_{lj}\frac{\partial w^{k}}{\partial y_{l}}%
dyd\tau \\ 
\text{for all }\mathbf{w}\in L_{per}^{2}\left( Z;V_{Y}\right) \text{.}\qquad
\qquad \qquad \qquad \qquad \qquad%
\end{array}%
\right.  \label{eq2.51}
\end{equation}%
But it is clear that $\mathbf{u}_{1}\left( x,t\right) $ (for fixed $\left(
x,t\right) \in Q$) is the unique function in $L_{per}^{2}\left(
Z;V_{Y}\right) $ solving the variational equation (\ref{eq2.51}). On the
other hand, it is an easy exercise to verify that $\mathbf{z}\left(
x,t\right) =-\sum_{i,k=1}^{N}\frac{\partial u_{0}^{k}}{\partial x_{i}}\left(
x,t\right) \mathbf{\chi }_{ik}$ solves also (\ref{eq2.51}). Hence the lemma
follows immediately.
\end{proof}

\subsection{Macroscopic homogenized equations}

Our aim here is to derive a well-posed initial boundary value problem for $%
\left( \mathbf{u}_{0},p_{0}\right) $. To begin, for $1\leq i,j,k,h\leq N$,
let%
\begin{equation*}
q_{ijkh}=\delta _{kh}\int_{Y}a_{ij}\left( y\right) dy-\sum_{l=1}^{N}\int
\int_{Y\times Z}a_{il}\left( y\right) \frac{\partial \mathcal{\chi }_{jh}^{k}%
}{\partial y_{l}}\left( y,\tau \right) dyd\tau \text{,}
\end{equation*}%
where $\delta _{kh}$ is the Kronecker symbol and $\mathbf{\chi }_{jh}=\left( 
\mathcal{\chi }_{jh}^{k}\right) $ is defined by (\ref{eq2.49}). To the
coefficients $q_{ijkh}$ we associate the differential operator $\mathcal{Q}$
on $Q$ mapping $\mathcal{D}^{\prime }\left( Q\right) ^{N}$ into $\mathcal{D}%
^{\prime }\left( Q\right) ^{N}$ ($\mathcal{D}^{\prime }\left( Q\right) $
being the usual space of complex distributions on $Q$) as 
\begin{equation}
(\mathcal{Q}\mathbf{z})^{k}=-\sum_{i,j,h=1}^{N}q_{ijkh}\frac{\partial
^{2}z^{h}}{\partial x_{i}\partial x_{j}}\quad \left( 1\leq k\leq N\right) 
\text{ for }\mathbf{z}=\left( z^{h}\right) \text{, }z^{h}\in \mathcal{D}%
^{\prime }\left( Q\right) \text{.}  \label{eq2.52}
\end{equation}%
$\mathcal{Q}$ is the so-called homogenized operator associated to $%
P^{\varepsilon }$ $\left( 0<\varepsilon <1\right) $.

Now, let us consider the initial boundary value\ problem 
\begin{equation}
\frac{\partial \mathbf{u}_{0}}{\partial t}+\mathcal{Q}\mathbf{u}%
_{0}+\sum_{j=1}^{N}u_{0}^{j}\frac{\partial \mathbf{u}_{0}}{\partial x_{j}}+%
\mathbf{grad}p_{0}=\mathbf{f}\text{ in }Q=\Omega \times ]0,T[\text{,}
\label{eq2.53}
\end{equation}%
\begin{equation}
div\mathbf{u}_{0}=0\text{ in }Q\text{,}  \label{eq2.54}
\end{equation}%
\begin{equation}
\mathbf{u}_{0}=0\text{ on }\partial \Omega \times ]0,T[\text{,}
\label{eq2.55}
\end{equation}%
\begin{equation}
\mathbf{u}_{0}\left( 0\right) =0\text{ in }\Omega \text{.}  \label{eq2.56}
\end{equation}

\begin{lemma}
\label{lem2.5} Suppose $N=2$. The initial boundary value problem (\ref%
{eq2.53})-(\ref{eq2.56}) admits at most one weak solution $\left( \mathbf{u}%
_{0},p_{0}\right) $ with

$\mathbf{u}_{0}\in \mathcal{W}\left( 0,T\right) $ and $p_{0}\in L^{2}\left(
0,T;L^{2}\left( \Omega ;\mathbb{R}\right) \mathfrak{/}\mathbb{R}\right) $.
\end{lemma}

\begin{proof}
If $\left( \mathbf{u}_{0},p_{0}\right) \in \mathcal{W}\left( 0,T\right)
\times L^{2}\left( 0,T;L^{2}\left( \Omega ;\mathbb{R}\right) \right) $
verifies (\ref{eq2.53})-(\ref{eq2.56}), then we have%
\begin{equation*}
\begin{array}{c}
\int_{0}^{T}\left( \mathbf{u}_{0}^{\prime }\left( t\right) ,\mathbf{v}%
_{0}\left( t\right) \right) dt+\sum_{i,j,k,h=1}^{N}\int_{Q}q_{ijkh}\frac{%
\partial u_{0}^{h}}{\partial x_{j}}\frac{\partial v_{0}^{k}}{\partial x_{i}}%
dxdt \\ 
+\int_{0}^{T}b\left( \mathbf{u}_{0}\left( t\right) ,\mathbf{u}_{0}\left(
t\right) ,\mathbf{v}_{0}\left( t\right) \right) dt=\int_{0}^{T}\left( 
\mathbf{f}\left( t\right) ,\mathbf{v}_{0}\left( t\right) \right) dt%
\end{array}%
\end{equation*}%
for all $\mathbf{v}_{0}\in L^{2}\left( 0,T;V\right) $. From the previous
equality, one quickly arrives at 
\begin{equation}
\begin{array}{c}
\int_{0}^{T}\left( \mathbf{u}_{0}^{\prime }\left( t\right) ,\mathbf{v}%
_{0}\left( t\right) \right) dt+\sum_{i,j,k=1}^{N}\int \int \int_{Q\times
Y\times Z}a_{ij}\left( \frac{\partial u_{0}^{k}}{\partial x_{j}}+\frac{%
\partial u_{1}^{k}}{\partial y_{j}}\right) \frac{\partial v_{0}^{k}}{%
\partial x_{i}}dxdtdyd\tau \\ 
+\int_{0}^{T}b\left( \mathbf{u}_{0}\left( t\right) ,\mathbf{u}_{0}\left(
t\right) ,\mathbf{v}_{0}\left( t\right) \right) dt=\int_{0}^{T}\left( 
\mathbf{f}\left( t\right) ,\mathbf{v}_{0}\left( t\right) \right) dt%
\end{array}
\label{eq2.57}
\end{equation}%
where $u_{1}^{k}\left( x,t,y,\tau \right) =-\sum_{i,h=1}^{N}\frac{\partial
u_{0}^{h}}{\partial x_{i}}\left( x,t\right) \mathcal{\chi }_{ih}^{k}\left(
y,\tau \right) $ for $\left( x,t,y,\tau \right) \in Q\times Y\times Z$. Let
us check that $\mathbf{u}=\left( \mathbf{u}_{0},\mathbf{u}_{1}\right) $
(with $\mathbf{u}_{1}\left( x,t,y,\tau \right) =-\sum_{i,k=1}^{N}\frac{%
\partial u_{0}^{k}}{\partial x_{i}}\left( x,t\right) \mathbf{\chi }%
_{ik}\left( y,\tau \right) $ for $\left( x,t,y,\tau \right) \in Q\times
Y\times Z$) satisfies (\ref{eq2.35}). Indeed, we have 
\begin{equation}
\sum_{i,j,k=1}^{N}\int \int \int_{Q\times Y\times Z}a_{ij}\left( \frac{%
\partial u_{0}^{k}}{\partial x_{j}}+\frac{\partial u_{1}^{k}}{\partial y_{j}}%
\right) \frac{\partial v_{1}^{k}}{\partial y_{i}}dxdtdyd\tau =0
\label{eq2.58}
\end{equation}%
for all $\mathbf{v}_{1}=\left( v_{1}^{k}\right) \in L^{2}\left(
Q;L_{per}^{2}\left( Z;V_{Y}\right) \right) $, since $\mathbf{u}_{1}\left(
x,t\right) $ verifies (\ref{eq2.51}) for $\left( x,t\right) \in Q$. Thus, by
(\ref{eq2.57})-(\ref{eq2.58}), we see that $\mathbf{u}=\left( \mathbf{u}_{0},%
\mathbf{u}_{1}\right) $ verifies (\ref{eq2.35}). Hence, the unicity in (\ref%
{eq2.53})-(\ref{eq2.56}) follows by Lemma \ref{lem2.3}.
\end{proof}

This leads us to the following theorem.

\begin{theorem}
\label{th2.4} Suppose that the hypotheses of Theorem \ref{th2.3} are
satisfied. For each $0<\varepsilon <1$, let $\left( \mathbf{u}_{\varepsilon
},p_{\varepsilon }\right) \in \mathcal{W}\left( 0,T\right) \times
L^{2}\left( 0,T;L^{2}\left( \Omega ;\mathbb{R}\right) \mathfrak{/}\mathbb{R}%
\right) $ be defined by (\ref{eq1.3})-(\ref{eq1.6}). Then, as $\varepsilon
\rightarrow 0$, we have $\mathbf{u}_{\varepsilon }\rightarrow \mathbf{u}_{0}$
in $\mathcal{W}\left( 0,T\right) $-weak and $p_{\varepsilon }\rightarrow
p_{0}$ in $L^{2}\left( 0,T;L^{2}\left( \Omega \right) \right) $-weak, where
the pair $\left( \mathbf{u}_{0},p_{0}\right) $\ lies in $\mathcal{W}\left(
0,T\right) \times L^{2}\left( 0,T;L^{2}\left( \Omega ;\mathbb{R}\right) 
\mathfrak{/}\mathbb{R}\right) $ and is the unique solution of (\ref{eq2.53}%
)-(\ref{eq2.56}).
\end{theorem}

\begin{proof}
Let $E$ be a fundamental sequence. As in the proof of Theorem \ref{th2.3},
there exists a subsequence $E^{\prime }$ extracted from $E$ such that as $%
E^{\prime }\ni \varepsilon \rightarrow 0$, we have (\ref{eq2.40})-(\ref%
{eq2.41}) and (\ref{eq2.43}) with $\mathbf{u=}\left( \mathbf{u}_{0},\mathbf{u%
}_{1}\right) \in \mathbb{F}_{0}^{1}$ and $\mathbf{u}_{0}\left( 0\right) =0$.
Then, from (\ref{eq2.43}) we have $p_{\varepsilon }\rightarrow p_{0}$ in $%
L^{2}\left( 0,T;L^{2}\left( \Omega \right) \right) $-weak when $E^{\prime
}\ni \varepsilon \rightarrow 0$, where $p_{0}$ is the mean of $p$. Hence, it
follows that $p_{0}\in L^{2}\left( 0,T;L^{2}\left( \Omega ;\mathbb{R}\right) 
\mathfrak{/}\mathbb{R}\right) $. Furher, (\ref{eq2.48}) holds for all $%
\mathbf{\Phi }=\left( \mathbf{\psi }_{0},\mathbf{\psi }_{1}\right) \in 
\mathcal{D}\left( Q;\mathbb{R}\right) ^{N}\times \mathcal{D}\left( Q;\mathbb{%
R}\right) \otimes \left[ \mathcal{C}_{per}^{\infty }\left( Z;\mathbb{R}%
\right) \otimes \mathcal{V}_{Y}\right] $. Then, substituting (\ref{eq2.50})
in (\ref{eq2.48}) and choosing therein the $\mathbf{\Phi }$'s such that $%
\mathbf{\psi }_{1}=0$, a simple computation leads to (\ref{eq2.53}) with
evidently (\ref{eq2.54})-(\ref{eq2.56}). Hence the theorem follows by Lemma %
\ref{lem2.5} since $E$ is arbitrarily chosen.
\end{proof}

\begin{remark}
\label{rem2.2} The operator $\mathcal{Q}$ is elliptic, i.e., there is some $%
\alpha _{0}>0$ such that%
\begin{equation*}
\sum_{i,j,k,h=1}^{N}q_{ijkh}\xi _{ik}\xi _{jh}\geq \alpha
_{0}\sum_{k,h=1}^{N}\left\vert \xi _{kh}\right\vert ^{2}
\end{equation*}%
for all $\mathbf{\xi =}\left( \xi _{ij}\right) $ with $\xi _{ij}\in \mathbb{R%
}$. Indeed, by following a classical line of argument (see, e.g., \cite{bib4}%
), we can give a suitable expression of $q_{ijkh}$, viz.%
\begin{equation*}
q_{ijkh}=\widehat{a}\left( \mathbf{\chi }_{ik}-\mathbf{\pi }_{ik},\mathbf{%
\chi }_{jh}-\mathbf{\pi }_{jh}\right) \text{,}
\end{equation*}%
where, for each pair of indices $1\leq i,k\leq N$, the vector function $%
\mathbf{\pi }_{ik}=\left( \pi _{ik}^{1},...,\pi _{ik}^{N}\right) :\mathbb{R}%
_{y}^{N}\rightarrow \mathbb{R}$ is given by $\pi _{ik}^{r}\left( y\right)
=y_{i}\delta _{kr}$ $\left( r=1,...,N\right) $ for $y=\left(
y_{1},...,y_{N}\right) \in \mathbb{R}^{N}$. Hence, the above ellipticity
property follows in a classical fashion.
\end{remark}

\section{General deterministic homogenization of unsteady Navier-Stokes type
equations}

Our goal here is to extend the results of Section 2 to a more general
setting beyond the periodic framework. The basic notation and hypotheses
(except the periodicity assumption) stated before are still valid.

\subsection{\textbf{Preliminaries and statement of the homogenization problem%
}}

We recall that $\mathcal{B}\left( \mathbb{R}_{y}^{N}\right) $, $\mathcal{B}%
\left( \mathbb{R}_{\tau }\right) $ and $\mathcal{B}\left( \mathbb{R}%
_{y}^{N}\times \mathbb{R}_{\tau }\right) $ denote respectively the spaces of
bounded continuous complex functions on $\mathbb{R}_{y}^{N}$, $\mathbb{R}%
_{\tau }$ and $\mathbb{R}_{y}^{N}\times \mathbb{R}_{\tau }$. It is well
known that the above spaces with the supremum norm and the usual algebra
operations are commutative $\mathcal{C}^{\ast }$-algebras with identity (the
involution is here the usual complex conjugation).

Throughout the present Section 3, $A_{y}$\ and $A_{\tau }$ denote
respectively the separable closed subalgebras of the Banach algebras $%
\mathcal{B}\left( \mathbb{R}_{y}^{N}\right) $ and $\mathcal{B}\left( \mathbb{%
R}_{\tau }\right) $, $A$ denotes the closure of $A_{y}\otimes A_{\tau }$ in $%
\mathcal{B}\left( \mathbb{R}_{y}^{N}\times \mathbb{R}_{\tau }\right) $,
which is also a separable closed subalgebra of $\mathcal{B}\left( \mathbb{R}%
_{y}^{N}\times \mathbb{R}_{\tau }\right) $. Further, we assume that $A_{y}$
and $A_{\tau }$ contain the constants, $A_{y}$ and $A_{\tau }$ are stable
under complex conjugation, and finally, $A_{y}$ and $A_{\tau }$ have the
following properties: For all $u\in A_{y}$ and $v\in A_{\tau }$, we have $%
u^{\varepsilon }\rightarrow M\left( u\right) $ in $L^{\infty }\left( \mathbb{%
R}_{x}^{N}\right) $-weak $\ast $ and $v^{\varepsilon }\rightarrow M\left(
v\right) $ in $L^{\infty }\left( \mathbb{R}_{t}\right) $-weak $\ast $ as $%
\varepsilon \rightarrow 0$ $\left( \varepsilon >0\right) $, where:%
\begin{equation*}
u^{\varepsilon }\left( x\right) =u\left( \frac{x}{\varepsilon }\right) \text{
}\left( x\in \mathbb{R}^{N}\right) \text{,}
\end{equation*}%
\begin{equation*}
v^{\varepsilon }\left( t\right) =v\left( \frac{t}{\varepsilon }\right) \text{
}\left( t\in \mathbb{R}\right) \text{,}
\end{equation*}%
the mapping $u\rightarrow M\left( u\right) $ of $A_{y}$ (resp. $A_{\tau }$)
into $\mathbb{C}$, denoted by $M$, is a positive continuous linear form on $%
A_{y}$ (resp. $A_{\tau }$) with $M\left( 1\right) =1$ (see \cite{bib17}).
Then, under those assumptions on $A_{y}$ and $A_{\tau }$, $A$ contains the
constants, $A$ is stable under complex conjugation and for any $w\in A$, we
have $w^{\varepsilon }\rightarrow M\left( w\right) $ in $L^{\infty }\left( 
\mathbb{R}_{\left( x,t\right) }^{N+1}\right) $-weak $\ast $ as $\varepsilon
\rightarrow 0$ $\left( \varepsilon >0\right) $ where%
\begin{equation*}
w^{\varepsilon }\left( x,t\right) =w\left( \frac{x}{\varepsilon },\frac{t}{%
\varepsilon }\right) \quad \left( \left( x,t\right) \in \mathbb{R}^{N}\times 
\mathbb{R}\right) \text{.}
\end{equation*}%
For the details, see \cite{bib17}.

$A_{y}$, $A_{\tau }$ and $A$ are called $H$-algebras. $A$ is the $H$-algebra
product of $A_{y}$ and $A_{\tau }$. It is clear that $A_{y}$, $A_{\tau }$
and $A$ are the commutative $\mathcal{C}^{\ast }$-algebras with identity. We
denote by $\Delta \left( A_{y}\right) $, $\Delta \left( A_{\tau }\right) $
and $\Delta \left( A\right) $ the spectra of $A_{y}$, $A_{\tau }$ and $A$
respectively, and by $\mathcal{G}$ the Gelfand transformation on $A_{y}$, $%
A_{\tau }$ and $A$. We recall that if $B$ is a commutative $\mathcal{C}%
^{\ast }$-algebras with identity, $\Delta \left( B\right) $ is the set of
all nonzero multiplicative linear forms on $B$, and $\mathcal{G}$ is the
mapping of $B$ into $\mathcal{C}\left( \Delta \left( B\right) \right) $ such
that $\mathcal{G}\left( u\right) \left( s\right) =\left\langle
s,u\right\rangle $ $\left( s\in \Delta \left( B\right) \right) $, where $%
\left\langle ,\right\rangle $ denotes the duality pairing between $B^{\prime
}$ (the topological dual of $B$) and $B$. The appropriate topology on $%
\Delta \left( B\right) $ is the relative weak $\ast $ topology on $B^{\prime
}$. So topologized, $\Delta \left( B\right) $ is a metrizable compact space,
and the Gelfand transformation is an isometric isomorphism of the $\mathcal{C%
}^{\ast }$-algebra $B$ onto the $\mathcal{C}^{\ast }$-algebra $\mathcal{C}%
\left( \Delta \left( B\right) \right) $. See, e.g., \cite{bib9} for further
details concerning the Banach algebras theory.

The appropriate measures on $\Delta \left( A_{y}\right) $, $\Delta \left(
A_{\tau }\right) $ and $\Delta \left( A\right) $ are the so-called $M$%
-measures, namely the positive Radon measures $\beta _{y}$, $\beta _{\tau }$
and $\beta $ (of total mass $1$) on $\Delta \left( A_{y}\right) $, $\Delta
\left( A_{\tau }\right) $ and $\Delta \left( A\right) $ respectively, such
that $M\left( u\right) =\int_{\Delta \left( A_{y}\right) }\mathcal{G}\left(
u\right) d\beta _{y}$ for $u\in A_{y}$, $M\left( v\right) =\int_{\Delta
\left( A_{\tau }\right) }\mathcal{G}\left( v\right) d\beta _{\tau }$ for $%
v\in A_{\tau }$ and $M\left( w\right) =\int_{\Delta \left( A\right) }%
\mathcal{G}\left( w\right) d\beta $ for $w\in A$ (see \cite[Proposition 2.1]%
{bib17}). Points in $\Delta \left( A_{y}\right) $ (resp. $\Delta \left(
A_{\tau }\right) $) are denoted by $s$ (resp. $s_{0}$). Furthermore, we have 
$\Delta \left( A\right) =\Delta \left( A_{y}\right) \times \Delta \left(
A_{\tau }\right) $ (Cartesian product) and $\beta =\beta _{y}\otimes \beta
_{\tau }$.

The partial derivative of index $i$ $\left( 1\leq i\leq N\right) $ on $%
\Delta \left( A_{y}\right) $ is defined to be the mapping $\partial _{i}=%
\mathcal{G}\circ D_{y_{i}}\circ \mathcal{G}^{-1}$ (usual composition) of

\begin{equation*}
\mathcal{D}^{1}\left( \Delta \left( A_{y}\right) \right) =\left\{ \varphi
\in \mathcal{C}\left( \Delta \left( A_{y}\right) \right) :\mathcal{G}%
^{-1}\left( \varphi \right) \in A_{y}^{1}\right\}
\end{equation*}%
into $\mathcal{C}\left( \Delta \left( A_{y}\right) \right) $, where $%
A_{y}^{1}=\left\{ \psi \in \mathcal{C}^{1}\left( \mathbb{R}_{y}^{N}\right)
:\psi ,\text{ }D_{y_{i}}\psi \in A_{y}\text{ }\left( 1\leq i\leq N\right)
\right\} $, $D_{y_{i}}=\frac{\partial }{\partial y_{i}}$. Higher order
derivatives can be defined analogously (see \cite{bib17}). Now, let $%
A_{y}^{\infty }$ be the space of $\psi \in \mathcal{C}^{\infty }\left( 
\mathbb{R}_{y}^{N}\right) $ such that 
\begin{equation*}
D_{y}^{\alpha }\psi =\frac{\partial ^{\left\vert \alpha \right\vert }\psi }{%
\partial y_{1}^{\alpha _{1}}...\partial y_{N}^{\alpha _{N}}}\in A_{y}
\end{equation*}%
for every multi-index $\alpha =\left( \alpha _{1},...,\alpha _{N}\right) \in 
\mathbb{N}^{N}$, and let 
\begin{equation*}
\mathcal{D}\left( \Delta \left( A_{y}\right) \right) =\left\{ \varphi \in 
\mathcal{C}\left( \Delta \left( A_{y}\right) \right) :\mathcal{G}^{-1}\left(
\varphi \right) \in A_{y}^{\infty }\right\} \text{.}
\end{equation*}%
Endowed with a suitable locally convex topology (see \cite{bib17}), $%
A_{y}^{\infty }$ (respectively $\mathcal{D}\left( \Delta \left( A_{y}\right)
\right) $) is a Fr\'{e}chet space and further, $\mathcal{G}$ viewed as
defined on $A_{y}^{\infty }$ is a topological isomorphism of $A_{y}^{\infty
} $ onto $\mathcal{D}\left( \Delta \left( A_{y}\right) \right) $.

By a distribution on $\Delta \left( A_{y}\right) $ is understood any
continuous linear form on $\mathcal{D}\left( \Delta \left( A_{y}\right)
\right) $. The space of all distributions on $\Delta \left( A_{y}\right) $
is then the dual, $\mathcal{D}^{\prime }\left( \Delta \left( A_{y}\right)
\right) $, of $\mathcal{D}\left( \Delta \left( A_{y}\right) \right) $. We
endow $\mathcal{D}^{\prime }\left( \Delta \left( A_{y}\right) \right) $ with
the strong dual topology. In the sequel it is assumed that $A_{y}^{\infty }$
is dense in $A_{y}$, which amounts to assuming that $\mathcal{D}\left(
\Delta \left( A_{y}\right) \right) $ is dense in $\mathcal{C}\left( \Delta
\left( A_{y}\right) \right) $. Then $L^{p}\left( \Delta \left( A_{y}\right)
\right) \subset \mathcal{D}^{\prime }\left( \Delta \left( A_{y}\right)
\right) $ $\left( 1\leq p\leq \infty \right) $ with continuous embedding
(see \cite{bib17} for more details). Hence we may define 
\begin{equation*}
H^{1}\left( \Delta \left( A_{y}\right) \right) =\left\{ u\in L^{2}\left(
\Delta \left( A_{y}\right) \right) :\partial _{i}u\in L^{2}\left( \Delta
\left( A_{y}\right) \right) \text{ }\left( 1\leq i\leq N\right) \right\} 
\text{,}
\end{equation*}%
where the derivative $\partial _{i}u$ is taken in the distribution sense on $%
\Delta \left( A_{y}\right) $ (exactly as the Schwartz derivative is defined
in the classical case). This is a Hilbert space with norm%
\begin{equation*}
\left\Vert u\right\Vert _{H^{1}\left( \Delta \left( A_{y}\right) \right)
}=\left( \left\Vert u\right\Vert _{L^{2}\left( \Delta \left( A_{y}\right)
\right) }^{2}+\sum_{i=1}^{N}\left\Vert \partial _{i}u\right\Vert
_{L^{2}\left( \Delta \left( A_{y}\right) \right) }^{2}\right) ^{\frac{1}{2}}%
\text{ }\left( u\in H^{1}\left( \Delta \left( A_{y}\right) \right) \right) 
\text{.}
\end{equation*}

However, in practice the appropriate space is not $H^{1}\left( \Delta \left(
A_{y}\right) \right) $ but its closed subspace 
\begin{equation*}
H^{1}\left( \Delta \left( A_{y}\right) \right) \mathfrak{/}\mathbb{C}%
=\left\{ u\in H^{1}\left( \Delta \left( A_{y}\right) \right) :\int_{\Delta
\left( A_{y}\right) }u\left( s\right) d\beta \left( s\right) =0\right\}
\end{equation*}%
equipped with the seminorm%
\begin{equation*}
\left\Vert u\right\Vert _{H^{1}\left( \Delta \left( A_{y}\right) \right) 
\mathfrak{/}\mathbb{C}}=\left( \sum_{i=1}^{N}\left\Vert \partial
_{i}u\right\Vert _{L^{2}\left( \Delta \left( A_{y}\right) \right)
}^{2}\right) ^{\frac{1}{2}}\text{ }\left( u\in H^{1}\left( \Delta \left(
A_{y}\right) \right) \mathfrak{/}\mathbb{C}\right) \text{.}
\end{equation*}%
Unfortunately, the pre-Hilbert space $H^{1}\left( \Delta \left( A_{y}\right)
\right) \mathfrak{/}\mathbb{C}$ is in general nonseparated and noncomplete.
We introduce the separated completion, $H_{\#}^{1}\left( \Delta \left(
A_{y}\right) \right) $, of $H^{1}\left( \Delta \left( A_{y}\right) \right) 
\mathfrak{/}\mathbb{C}$, and the canonical mapping $J_{y}$ of $H^{1}\left(
\Delta \left( A_{y}\right) \right) \mathfrak{/}\mathbb{C}$ into its
separated completion. See \cite{bib17} (and in particular Remark 2.4 and
Proposition 2.6 there) for more details.

In the sequel, we assume also $A_{\tau }^{\infty }$ to be dense in $A_{\tau
} $, where $A_{\tau }^{\infty }$ is the space of $w\in \mathcal{C}^{\infty
}\left( \mathbb{R}_{\tau }\right) $ such that 
\begin{equation*}
\frac{d^{\alpha }w}{d\tau ^{\alpha }}\in A_{\tau }\quad \left( \alpha \in 
\mathbb{N}\right) \text{.}
\end{equation*}

We will now recall the notion of $\Sigma $-convergence in the present
context. Let $1\leq p<\infty $, and let $E$ be as in Section 2.

\begin{definition}
\label{def3.1} A sequence $\left( u_{\varepsilon }\right) _{\varepsilon \in
E}\subset L^{p}\left( Q\right) $ is said to:

(i) weakly $\Sigma $-converge in $L^{p}\left( Q\right) $ to some $u_{0}\in
L^{p}\left( Q\times \Delta \left( A\right) \right) =$

\noindent $L^{p}\left( Q;L^{p}\left( \Delta \left( A\right) \right) \right) $
if as $E\ni \varepsilon \rightarrow 0$, 
\begin{equation*}
\int_{Q}u_{\varepsilon }\left( x\right) \psi ^{\varepsilon }\left(
x,t\right) dxdt\rightarrow \int \int_{Q\times \Delta \left( A\right)
}u_{0}\left( x,t,s,s_{0}\right) \widehat{\psi }\left( x,s,s_{0}\right)
dxdtd\beta \left( s,s_{0}\right)
\end{equation*}%
for all $\psi \in L^{p^{\prime }}\left( Q;A\right) $ $\left( \frac{1}{%
p^{\prime }}=1-\frac{1}{p}\right) $, where $\psi ^{\varepsilon }$ is as in
Definition \ref{def2.1}, and where $\widehat{\psi }\left( x,t,.,.\right) =%
\mathcal{G}\left( \psi \left( x,t,.,.\right) \right) $ a.e. in $\left(
x,t\right) \in Q$;

(ii) stronly $\Sigma $-converge in $L^{p}\left( Q\right) $ to some $u_{0}\in
L^{p}\left( Q\times \Delta \left( A\right) \right) $ if the following
property is verified:%
\begin{equation*}
\left\{ 
\begin{array}{c}
\text{Given }\eta >0\text{ and }v\in L^{p}\left( Q;A\right) \text{ with\quad
\qquad \qquad \qquad } \\ 
\left\Vert u_{0}-\widehat{v}\right\Vert _{L^{p}\left( Q\times \Delta \left(
A\right) \right) }\leq \frac{\eta }{2}\text{, there is some }\alpha >0\qquad
\quad \\ 
\text{such that }\left\Vert u_{\varepsilon }-v^{\varepsilon }\right\Vert
_{L^{p}\left( Q\right) }\leq \eta \text{ provided }E\ni \varepsilon \leq
\alpha \text{.}%
\end{array}%
\right.
\end{equation*}
\end{definition}

\begin{remark}
\label{rem3.1} The existence of such $v$'s as in (ii) results from the
density of $L^{p}\left( Q;\mathcal{C}\left( \Delta \left( A\right) \right)
\right) $ in $L^{p}\left( Q;L^{p}\left( \Delta \left( A\right) \right)
\right) $.
\end{remark}

We will use the same notation as in Section 2 to briefly express weak and
strong $\Sigma $-convergence.

Theorem \ref{th2.1} (together with its proof) carries over to the present
setting. Instead of Theorem \ref{th2.2}, we have here the following notion.

\begin{definition}
\label{def3.2} The $H$-algebra $A$ is said to be quasi-\textit{proper if the
following conditions are fulfilled.}

\noindent (QPR)$_{1}$ $\mathcal{D}\left( \Delta \left( A_{y}\right) \right) $
is dense in $H^{1}\left( \Delta \left( A_{y}\right) \right) $.

\noindent (QPR)$_{2}$ Given a fundamental sequence $E$, and a sequence $%
\left( u_{\varepsilon }\right) _{\varepsilon \in E}$ which is bounded in $%
\mathcal{Y}\left( 0,T\right) $, one can extract a subsequence $E^{\prime }$
from $E$ such that as $E^{\prime }\ni \varepsilon \rightarrow 0$, we have $%
u_{\varepsilon }\rightarrow u_{0}$ in $\mathcal{Y}\left( 0,T\right) $-weak
and $\frac{\partial u_{\varepsilon }}{\partial x_{j}}\rightarrow \frac{%
\partial u_{0}}{\partial x_{j}}+\partial _{j}u_{1}$ in $L^{2}\left( Q\right) 
$-weak $\Sigma $ $\left( 1\leq j\leq N\right) $, where $u_{0}\in \mathcal{Y}%
\left( 0,T\right) $, $u_{1}\in L^{2}\left( Q;L^{2}\left( \Delta \left(
A_{\tau }\right) ;H_{\#}^{1}\left( \Delta \left( A_{y}\right) \right)
\right) \right) $.
\end{definition}

The $H$-algebra $A=\mathcal{C}_{per}\left( Y\times Z\right) $ (see Section
2) is quasi-proper. Other examples of quasi-proper $H$-algebras can be found
in \cite{bib24}.

Having made the above preliminaries, let us turn now to the statement of the
general deterministic homogenization problem for (\ref{eq1.3})-(\ref{eq1.6}%
). For this purpose, let $\Xi ^{2}\left( \mathbb{R}^{N}\right) $ be the
space of functions $u\in L_{loc}^{2}\left( \mathbb{R}_{y}^{N}\right) $ such
that 
\begin{equation*}
\left\Vert u\right\Vert _{\Xi ^{2}}=\underset{0<\varepsilon \leq 1}{\sup }%
\left( \int_{B_{N}}\left\vert u\left( \frac{x}{\varepsilon }\right)
\right\vert ^{2}dx\right) ^{\frac{1}{2}}<\infty \text{,}
\end{equation*}%
where $B_{N}$ denotes the open unit ball in $\mathbb{R}^{N}$. $\Xi ^{2}$ is
a complex vector space, and the mapping $u\rightarrow \left\Vert
u\right\Vert _{\Xi ^{2}}$, denoted by $\left\Vert .\right\Vert _{\Xi ^{2}}$,
is a norm on $\Xi ^{2}$ which makes it a Banach space (this is a simple
exercise left to the reader). We define $\mathfrak{X}_{y}^{2}$ and $%
\mathfrak{X}^{2}$ to be the closure of $A_{y}$ and $A$ in $\Xi ^{2}\left( 
\mathbb{R}^{N}\right) $ and $\Xi ^{2}\left( \mathbb{R}^{N+1}\right) $
respectively. We provide $\mathfrak{X}_{y}^{2}$ (resp. $\mathfrak{X}^{2}$)
with the $\Xi ^{2}\left( \mathbb{R}^{N}\right) $-norm (resp. $\Xi ^{2}\left( 
\mathbb{R}^{N+1}\right) $-norm), which makes it a Banach space.

\begin{remark}
\label{rem3.2} Any function $u\in \mathfrak{X}_{y}^{2}$ can be consider as a
function in $\mathfrak{X}^{2}$ which is independent of the variable $\tau $.
Indeed, let $u\in \mathfrak{X}_{y}^{p}$ and $\eta >0$. There exists a
function $v\in A_{y}$\ such that 
\begin{equation*}
\left\Vert u-v\right\Vert _{\Xi ^{p}\left( \mathbb{R}^{N}\right) }\leq \frac{%
\eta }{2}\text{, i.e., }\sup_{0<\varepsilon \leq 1}\left(
\int_{B_{N}}\left\vert u^{\varepsilon }-v^{\varepsilon }\right\vert
^{p}dy\right) ^{\frac{1}{p}}\leq \frac{\eta }{2}\text{.}
\end{equation*}%
but $v=v\otimes 1\in A_{y}\otimes A_{\tau }\subset A$ and 
\begin{equation*}
\int_{B_{N+1}}\left\vert u^{\varepsilon }-v^{\varepsilon }\right\vert
^{p}dyd\tau \leq 2\int_{B_{N}}\left\vert u^{\varepsilon }-v^{\varepsilon
}\right\vert ^{p}dy\leq 2\left\Vert u-v\right\Vert _{\Xi ^{p}\left( \mathbb{R%
}^{N}\right) }\text{.}
\end{equation*}%
It follows from the preceding inequalities that $u\in \Xi ^{p}\left( \mathbb{%
R}^{N+1}\right) $\ and $\left\Vert u-v\right\Vert _{\Xi ^{p}\left( \mathbb{R}%
^{N+1}\right) }\leq \eta $.
\end{remark}

Our main purpose in the present section is to discuss the homogenization of (%
\ref{eq1.3})-(\ref{eq1.6}) under the assumption 
\begin{equation}
a_{ij}\in \mathfrak{X}^{2}\text{ }\left( 1\leq i,j\leq N\right) \text{.}
\label{eq3.1}
\end{equation}

As is pointed out in \cite{bib17}, \cite{bib18} and \cite{bib19}, assumption
(\ref{eq3.1}) covers a great variety of concrete behaviors. In particular, (%
\ref{eq3.1}) generalizes the usual periodicity hypothesis (see Section 2).
Indeed, for $A=\mathcal{C}_{per}\left( Y\times Z\right) $, we have $%
\mathfrak{X}^{2}=L_{per}^{2}\left( Y\times Z\right) $ (use Lemma 1 of \cite%
{bib16}).

The approach we follow here is analogous to the one which was presented in
Section 2. Throughout the rest of the section, it is assumed that (\ref%
{eq3.1}) is satisfied, and $A$, the closure of $A_{y}\otimes A_{\tau }$ in $%
\mathcal{B}\left( \mathbb{R}_{y}^{N}\times \mathbb{R}_{\tau }\right) $ is
quasi-proper.

\subsection{\textbf{A global homogenization theorem}}

We need a few preliminaries. To begin, we set 
\begin{equation*}
\mathcal{G}\left( \mathbf{\psi }\right) =\left( \mathcal{G}\left( \psi
^{i}\right) \right) _{1\leq i\leq N}
\end{equation*}%
for any $\mathbf{\psi =}\left( \psi ^{i}\right) $ with $\psi ^{i}\in A$ $%
\left( 1\leq i\leq N\right) $. We have $\mathcal{G}\left( \mathbf{\psi }%
\right) \in \mathcal{C}\left( \Delta \left( A\right) \right) ^{N}$, and the
transformation $\mathbf{\psi }\rightarrow \mathcal{G}\left( \mathbf{\psi }%
\right) $ of $A^{N}$ into $\mathcal{C}\left( \Delta \left( A\right) \right)
^{N}$ maps in particular $\left( A_{\mathbb{R}}^{\infty }\right) ^{N}$
isomorphically onto $\mathcal{D}\left( \Delta \left( A\right) ;\mathbb{R}%
\right) ^{N}$, where we denote 
\begin{equation*}
A_{\mathbb{R}}^{\infty }=A^{\infty }\cap \mathcal{C}\left( \mathbb{R}^{N};%
\mathbb{R}\right) \text{.}
\end{equation*}%
Likewise, letting $\mathbf{J}_{y}\left( \mathbf{u}\right) =\left(
J_{y}\left( u^{i}\right) \right) _{1\leq i\leq N}$ for $\mathbf{u=}\left(
u^{i}\right) $ with $u^{i}\in H^{1}\left( \Delta \left( A_{y}\right) \right) 
\mathfrak{/}\mathbb{C}$ $\left( 1\leq i\leq N\right) $, we have $\mathbf{J}%
_{y}\left( \mathbf{u}\right) \in H_{\#}^{1}\left( \Delta \left( A_{y}\right)
\right) ^{N}$ and the transformation $\mathbf{u}\rightarrow \mathbf{J}%
_{y}\left( \mathbf{u}\right) $ of $\left[ H^{1}\left( \Delta \left(
A_{y}\right) \right) \mathfrak{/}\mathbb{C}\right] ^{N}$ into $%
H_{\#}^{1}\left( \Delta \left( A_{y}\right) \right) ^{N}$ maps in particular 
$\left[ H^{1}\left( \Delta \left( A_{y}\right) ;\mathbb{R}\right) \mathfrak{/%
}\mathbb{C}\right] ^{N}$ isometrically into $H_{\#}^{1}\left( \Delta \left(
A_{y}\right) ;\mathbb{R}\right) ^{N}$, where we denote 
\begin{equation*}
H_{\#}^{1}\left( \Delta \left( A_{y}\right) ;\mathbb{R}\right) =\left\{ u\in
H_{\#}^{1}\left( \Delta \left( A_{y}\right) \right) :\partial _{i}u\in
L^{2}\left( \Delta \left( A_{y}\right) ;\mathbb{R}\right) \text{ }\left(
1\leq i\leq N\right) \right\} \text{.}
\end{equation*}%
For any $u\in L^{2}\left( \Delta \left( A_{\tau }\right) ;H^{1}\left( \Delta
\left( A_{y}\right) \right) /\mathbb{C}\right) $, we put%
\begin{equation*}
J\left( u\right) \left( s_{0}\right) =J_{y}\left( u\left( s_{0}\right)
\right) \quad \left( s_{0}\in \Delta \left( A_{\tau }\right) \right) \text{.}
\end{equation*}%
This defines a continuous linear mapping $J$ of

\noindent $L^{2}\left( \Delta \left( A_{\tau }\right) ;H^{1}\left( \Delta
\left( A_{y}\right) \right) /\mathbb{C}\right) $ into $L^{2}\left( \Delta
\left( A_{\tau }\right) ;H_{\#}^{1}\left( \Delta \left( A_{y}\right) \right)
\right) $ with the equality 
\begin{equation*}
J\left( \chi \otimes v\right) =\chi \otimes J_{y}\left( v\right) \text{ for }%
\chi \in L^{2}\left( \Delta \left( A_{\tau }\right) \right) \text{ and }v\in
H^{1}\left( \Delta \left( A_{y}\right) \right) /\mathbb{C}\text{.}
\end{equation*}%
Furthermore, letting $\mathbf{J}\left( \mathbf{u}\right) =\left( J\left(
u^{i}\right) \right) $ for $\mathbf{u=}\left( u^{i}\right) $ with $u^{i}\in
L^{2}\left( \Delta \left( A_{\tau }\right) ;H^{1}\left( \Delta \left(
A_{y}\right) \right) /\mathbb{C}\right) $ $\left( 1\leq i\leq N\right) $, we
have $\mathbf{J}\left( \mathbf{u}\right) \in L^{2}\left( \Delta \left(
A_{\tau }\right) ;H_{\#}^{1}\left( \Delta \left( A_{y}\right) \right)
^{N}\right) $ and the transformation $\mathbf{u}\rightarrow \mathbf{J}\left( 
\mathbf{u}\right) $ of $L^{2}\left( \Delta \left( A_{\tau }\right) ;\left[
H^{1}\left( \Delta \left( A_{y}\right) \right) /\mathbb{C}\right]
^{N}\right) $ into $L^{2}\left( \Delta \left( A_{\tau }\right)
;H_{\#}^{1}\left( \Delta \left( A_{y}\right) \right) ^{N}\right) $ maps in
particular $L^{2}\left( \Delta \left( A_{\tau }\right) ;\left[ H^{1}\left(
\Delta \left( A_{y}\right) ;\mathbb{R}\right) /\mathbb{C}\right] ^{N}\right) 
$ isometrically into

\noindent $L^{2}\left( \Delta \left( A_{\tau }\right) ;H_{\#}^{1}\left(
\Delta \left( A_{y}\right) ;\mathbb{R}\right) ^{N}\right) $.

We will set 
\begin{equation*}
\mathbb{E}_{0}^{1}=L^{2}\left( 0,T;H_{0}^{1}\left( \Omega ;\mathbb{R}\right)
^{N}\right) \times L^{2}\left( Q;L^{2}\left( \Delta \left( A_{\tau }\right)
;H_{\#}^{1}\left( \Delta \left( A_{y}\right) ;\mathbb{R}\right) ^{N}\right)
\right) \text{,}
\end{equation*}%
\begin{equation*}
\mathcal{E}_{0}^{\infty }=\mathcal{D}\left( Q;\mathbb{R}\right) ^{N}\times
\left( \mathcal{D}\left( Q;\mathbb{R}\right) \otimes \left[ \mathcal{D}%
\left( \Delta \left( A_{\tau }\right) ;\mathbb{R}\right) \otimes \mathbf{J}%
_{y}\left[ \left( \mathcal{D}\left( \Delta \left( A\right) ;\mathbb{R}%
\right) \mathfrak{/}\mathbb{C}\right) ^{N}\right] \right] \right) \text{,}
\end{equation*}%
where $\mathcal{D}\left( \Delta \left( A_{y}\right) ;\mathbb{R}\right) 
\mathfrak{/}\mathbb{C=}\mathcal{D}\left( \Delta \left( A_{y}\right) ;\mathbb{%
R}\right) \cap \left[ H^{1}\left( \Delta \left( A_{y}\right) \right) 
\mathfrak{/}\mathbb{C}\right] $. $\mathbb{E}_{0}^{1}$ is topologized in an
obvious way and $\mathcal{E}_{0}^{\infty }$ is considered without topology.
It is clear that $\mathcal{E}_{0}^{\infty }$ is dense in $\mathbb{E}_{0}^{1}$%
.

At the present time, let us consider the vector space%
\begin{equation*}
\mathbb{U}_{0}^{1}=H_{0}^{1}\left( \Omega ;\mathbb{R}\right) ^{N}\times
L^{2}\left( \Omega ;L^{2}\left( \Delta \left( A_{\tau }\right)
;H_{\#}^{1}\left( \Delta \left( A_{y}\right) ;\mathbb{R}\right) ^{N}\right)
\right)
\end{equation*}%
topologized in an obvious way. We put

\begin{equation*}
\widehat{a}_{\Omega }\left( \mathbf{u},\mathbf{v}\right)
=\sum_{i,j,k=1}^{N}\int \int_{\Omega \times \Delta \left( A\right) }\widehat{%
a}_{ij}\left( \frac{\partial u_{0}^{k}}{\partial x_{j}}+\partial
_{j}u_{1}^{k}\right) \left( \frac{\partial v_{0}^{k}}{\partial x_{i}}%
+\partial _{i}v_{1}^{k}\right) dxd\beta
\end{equation*}%
for $\mathbf{u}=\left( \mathbf{u}_{0},\mathbf{u}_{1}\right) $ and $\mathbf{v}%
=\left( \mathbf{v}_{0},\mathbf{v}_{1}\right) $ in $\mathbb{U}_{0}^{1}$ with,
of course, $\mathbf{u}_{0}=\left( \mathbf{u}_{0}^{k}\right) $, $\mathbf{u}%
_{1}=\left( \mathbf{u}_{1}^{k}\right) $ (and analogous expressions for $%
\mathbf{v}_{0}$ and $\mathbf{v}_{1}$), where $\widehat{a}_{ij}=\mathcal{G}%
\left( a_{ij}\right) $. This gives a bilinear form $\widehat{a}_{\Omega }$
on $\mathbb{U}_{0}^{1}\times \mathbb{U}_{0}^{1}$, which is symmetric,
continuous, and coercive (see \cite{bib17}).

Now, let 
\begin{equation*}
V_{A_{y}}=\left\{ \mathbf{u=}\left( u^{i}\right) \in H_{\#}^{1}\left( \Delta
\left( A_{y}\right) ;\mathbb{R}\right) ^{N}:\widehat{div}\mathbf{u=}%
0\right\} \text{,}
\end{equation*}%
where 
\begin{equation*}
\widehat{div}\mathbf{u}=\sum_{i=1}^{N}\partial _{i}u^{i}\text{.}
\end{equation*}%
Equipped with the $H_{\#}^{1}\left( \Delta \left( A_{y}\right) \right) ^{N}$%
-norm, $V_{A_{y}}$ is a Hilbert space. We next put 
\begin{equation*}
\mathbb{F}_{0}^{1}=L^{2}\left( 0,T;V\right) \times L^{2}\left( Q;L^{2}\left(
\Delta \left( A_{\tau }\right) ;V_{A_{y}}\right) \right)
\end{equation*}%
provided with an obvious norm. It is an easy exercise to check that Lemma %
\ref{lem2.3} together with its proof can be carried over mutatis mutandis to
the present setting. This leads us to the analogue of Theorem \ref{th2.3}.

\begin{theorem}
\label{th3.1} Suppose (\ref{eq3.1}) holds and further $A$ is quasi-proper.
On the other hand, suppose that the hypotheses of Lemma \ref{lem2.2} are
satisfied. For each real $0<\varepsilon <1$, let $\mathbf{u}_{\varepsilon
}=\left( u_{\varepsilon }^{k}\right) $ be defined by (\ref{eq1.3})-(\ref%
{eq1.6}). Then, as $\varepsilon \rightarrow 0$, 
\begin{equation}
\mathbf{u}_{\varepsilon }\rightarrow \mathbf{u}_{0}\text{ in }\mathcal{W}%
\left( 0,T\right) \text{-weak,\qquad \qquad \qquad \qquad \qquad \quad }
\label{eq3.2}
\end{equation}%
\begin{equation}
\frac{\partial u_{\varepsilon }^{k}}{\partial x_{j}}\rightarrow \frac{%
\partial u_{0}^{k}}{\partial x_{j}}+\partial _{j}u_{1}^{k}\text{ in }%
L^{2}\left( Q\right) \text{-weak }\Sigma \text{ }\left( 1\leq j,k\leq
N\right) \text{,}  \label{eq3.3}
\end{equation}%
where $\mathbf{u}=\left( \mathbf{u}_{0},\mathbf{u}_{1}\right) $ [with $%
\mathbf{u}_{0}=\left( \mathbf{u}_{0}^{k}\right) $ and $\mathbf{u}_{1}=\left( 
\mathbf{u}_{1}^{k}\right) $] is the unique solution of (\ref{eq2.34})-(\ref%
{eq2.35}).
\end{theorem}

\begin{proof}
This is an adaptation of the proof of Theorem \ref{th2.3} and we will not go
too deeply into details. Starting from (\ref{eq2.11})-(\ref{eq2.13}), we see
that the generalized sequences $\left( \mathbf{u}_{\varepsilon }\right)
_{0<\varepsilon <1}$ and $\left( p_{\varepsilon }\right) _{0<\varepsilon <1}$
are bounded in $\mathcal{W}\left( 0,T\right) $ and $L^{2}\left( Q\right) $,
respectively. Moreover, for $1\leq k\leq N$, the sequence $\left(
u_{\varepsilon }^{k}\right) _{0<\varepsilon <1}$ is bounded in $\mathcal{Y}%
\left( 0,T\right) $. Hence, from any given fundamental sequence $E$ one can
extract a subsequence $E^{\prime }$ such that as $E^{\prime }\ni \varepsilon
\rightarrow 0$, we have (\ref{eq2.43}), (\ref{eq3.2}) and (\ref{eq3.3}),
where $p$ lies in $L^{2}\left( Q;L^{2}\left( \Delta \left( A\right) ;\mathbb{%
R}\right) \right) $ and $\mathbf{u}=\left( \mathbf{u}_{0},\mathbf{u}%
_{1}\right) $ lies in $\mathbb{F}_{0}^{1}$ with (\ref{eq2.34}).

Now, for each real $0<\varepsilon <1$, let 
\begin{equation}
\left\{ 
\begin{array}{c}
\mathbf{\Phi }_{\varepsilon }=\mathbf{\psi }_{0}+\varepsilon \mathbf{\psi }%
_{1}^{\varepsilon }\text{ with} \\ 
\mathbf{\psi }_{0}\in \mathcal{D}\left( Q;\mathbb{R}\right) ^{N}\text{, }%
\mathbf{\psi }_{1}\in \mathcal{D}\left( Q;\mathbb{R}\right) \otimes \left[ _{%
\mathbb{R}}A_{\tau }^{\infty }\otimes \left( _{\mathbb{R}}A_{y}^{\infty }%
\mathfrak{/}\mathbb{C}\right) ^{N}\right]%
\end{array}%
\right.  \label{eq3.4}
\end{equation}%
and%
\begin{equation*}
\mathbf{\Phi }=\left( \mathbf{\psi }_{0},\mathbf{J}\left( \widehat{\mathbf{%
\psi }}_{1}\right) \right) \text{,}
\end{equation*}%
where: $_{\mathbb{R}}A_{y}^{\infty }\mathfrak{/}\mathbb{C=}\left\{ \psi \in
A_{y}^{\infty }:M\left( \psi \right) =0\right\} $, $_{\mathbb{R}}A_{\tau
}^{\infty }=\left\{ w\in A_{\tau }^{\infty }:M\left( w\right) =0\right\} $, $%
\widehat{\mathbf{\psi }}_{1}$ stands for the function

\noindent $\left( x,t\right) \rightarrow \mathcal{G}\left( \mathbf{\psi }%
_{1}\left( x,t,.\right) \right) $ of $Q$ into $\left[ \mathcal{D}\left(
\Delta \left( A\right) ;\mathbb{R}\right) \mathfrak{/}\mathbb{C}\right] ^{N}$
($\mathbf{\psi }_{1}$ being viewed as a function in $\mathcal{C}\left(
Q;A^{N}\right) $), $\mathbf{J}\left( \widehat{\mathbf{\psi }}_{1}\right) $
stands for the function $\left( x,t\right) \rightarrow \mathbf{J}\left( 
\widehat{\mathbf{\psi }}_{1}\left( x,t,.\right) \right) $ of $Q$ into $%
L^{2}\left( \Delta \left( A_{\tau }\right) ;H_{\#}^{1}\left( \Delta \left(
A_{y}\right) ;\mathbb{R}\right) ^{N}\right) $. It is clear that $\mathbf{%
\Phi }\in \mathcal{E}_{0}^{\infty }$. With this in mind, we can pass to the
limit in (\ref{eq2.45}) (with $\mathbf{\Phi }_{\varepsilon }$ given by (\ref%
{eq3.4})) as $E^{\prime }\ni \varepsilon \rightarrow 0$, and we obtain 
\begin{equation*}
\begin{array}{c}
\int_{0}^{T}\left( \mathbf{u}_{0}^{\prime }\left( t\right) ,\mathbf{\psi }%
_{0}\left( t\right) \right) dt+\int_{0}^{T}\widehat{a}_{\Omega }\left( 
\mathbf{u}\left( t\right) ,\mathbf{\Phi }\left( t\right) \right)
dt+\int_{0}^{T}b\left( \mathbf{u}_{0}\left( t\right) ,\mathbf{u}_{0}\left(
t\right) ,\mathbf{\psi }_{0}\left( t\right) \right) dt \\ 
-\int \int_{Q\times \Delta \left( A\right) }p\left( div\mathbf{\psi }_{0}+%
\widehat{div}\widehat{\mathbf{\psi }}_{1}\right) dxdt=\int_{0}^{T}\left( 
\mathbf{f}\left( t\right) ,\mathbf{\psi }_{0}\left( t\right) \right) dt\text{%
.}%
\end{array}%
\end{equation*}%
Therefore, thanks to the density of $\mathcal{E}_{0}^{\infty }$ in $\mathbb{E%
}_{0}^{1}$, 
\begin{equation}
\begin{array}{c}
\int_{0}^{T}\left( \mathbf{u}_{0}^{\prime }\left( t\right) ,\mathbf{v}%
_{0}\left( t\right) \right) dt+\int_{0}^{T}\widehat{a}_{\Omega }\left( 
\mathbf{u}\left( t\right) ,\mathbf{v}\left( t\right) \right)
dt+\int_{0}^{T}b\left( \mathbf{u}_{0}\left( t\right) ,\mathbf{u}_{0}\left(
t\right) ,\mathbf{v}_{0}\left( t\right) \right) dt \\ 
-\int \int_{Q\times \Delta \left( A\right) }p\left( div\mathbf{v}_{0}+%
\widehat{div}\mathbf{v}_{1}\right) dxdt=\int_{0}^{T}\left( \mathbf{f}\left(
t\right) ,\mathbf{v}_{0}\left( t\right) \right) dt\text{,}%
\end{array}
\label{eq3.5}
\end{equation}%
and that for all $\mathbf{v=}\left( \mathbf{v}_{0},\mathbf{v}_{1}\right) \in 
\mathbb{E}_{0}^{1}$. Taking in particular $\mathbf{v}\in \mathbb{F}_{0}^{1}$
leads us immediately to (\ref{eq2.35}). Hence the theorem follows by the
same argument as used in the proof of Theorem \ref{th2.3}.
\end{proof}

As was pointed out in Section 2, it is of interest to give a suitable
representation of $\mathbf{u}_{1}$ (in Theorem \ref{th3.1}). To this end,
let 
\begin{equation*}
\widehat{a}\left( \mathbf{v},\mathbf{w}\right)
=\sum_{i,j,k=1}^{N}\int_{\Delta \left( A\right) }\widehat{a}_{ij}\partial
_{j}v^{k}\partial _{i}w^{k}d\beta
\end{equation*}%
for $\mathbf{v=}\left( v^{k}\right) $ and $\mathbf{w=}\left( w^{k}\right) $
in $L^{2}\left( \Delta \left( A_{\tau }\right) ;H_{\#}^{1}\left( \Delta
\left( A_{y}\right) ;\mathbb{R}\right) ^{N}\right) $. This defines a
bilinear form $\widehat{a}$ on $L^{2}\left( \Delta \left( A_{\tau }\right)
;H_{\#}^{1}\left( \Delta \left( A_{y}\right) ;\mathbb{R}\right) ^{N}\right)
\times L^{2}\left( \Delta \left( A_{\tau }\right) ;H_{\#}^{1}\left( \Delta
\left( A_{y}\right) ;\mathbb{R}\right) ^{N}\right) $, which is symmetric,
continuous and coercive. For each couple of indices $1\leq i,k\leq N$, we
consider the variational problem%
\begin{equation}
\left\{ 
\begin{array}{c}
\mathbf{\chi }_{ik}\in L^{2}\left( \Delta \left( A_{\tau }\right)
;V_{A_{y}}\right) :\qquad \qquad \qquad \qquad \quad \\ 
\widehat{a}\left( \mathbf{\chi }_{ik},\mathbf{w}\right)
=\sum_{l=1}^{N}\int_{\Delta \left( A\right) }\widehat{a}_{li}\partial
_{l}w^{k}d\beta \qquad \quad \qquad \\ 
\text{for all }\mathbf{w=}\left( w^{k}\right) \in L^{2}\left( \Delta \left(
A_{\tau }\right) ;V_{A_{y}}\right) \text{,\quad \qquad \quad }%
\end{array}%
\right.  \label{eq3.6}
\end{equation}%
which uniquely determines $\mathbf{\chi }_{ik}$.

\begin{lemma}
\label{lem3.1} Under the assumptions and notation of Theorem \ref{th3.1}, we
have 
\begin{equation}
\mathbf{u}_{1}\left( x,t,s,s_{0}\right) =-\sum_{i,k=1}^{N}\frac{\partial
u_{0}^{k}}{\partial x_{i}}\left( x,t\right) \mathbf{\chi }_{ik}\left(
s,s_{0}\right)  \label{eq3.7}
\end{equation}%
almost everywhere in $\left( x,t,s,s_{0}\right) \in Q\times \Delta \left(
A\right) =\Omega \times ]0,T[\times \Delta \left( A_{y}\right) \times \Delta
\left( A_{\tau }\right) $.
\end{lemma}

\begin{proof}
This is a simple adaptation of the proof of Lemma \ref{lem2.4}; the
verification is left to the reader.
\end{proof}

\subsection{\textbf{Macroscopic homogenized equations}}

The aim here is to derive from (\ref{eq3.5}) a well-posed initial boundary
value problem for the couple $\left( \mathbf{u}_{0},p_{0}\right) $, where $%
\mathbf{u}_{0}$ is the weak limit in (\ref{eq3.2}) and $p_{0}$ is the mean
of $p$ (in (\ref{eq3.5})), i.e., $p_{0}\left( x,t\right) =\int_{\Delta
\left( A\right) }p\left( x,t,s,s_{0}\right) d\beta \left( s,s_{0}\right) $
for $\left( x,t\right) \in Q$. We will proceed exactly as in Subsection 2.3.

First, for $1\leq i,j,k,h\leq N$, let 
\begin{equation*}
q_{ijkh}=\delta _{kh}\int_{\Delta \left( A_{y}\right) }\widehat{a}%
_{ij}\left( s\right) d\beta _{y}\left( s\right) -\sum_{l=1}^{N}\int_{\Delta
\left( A\right) }\widehat{a}_{il}\left( s\right) \partial _{{\footnotesize l}%
}\mathcal{\chi }_{jh}^{k}\left( s,s_{0}\right) d\beta \left( s,s_{0}\right) 
\text{,}
\end{equation*}%
where $\mathbf{\chi }_{jh}=\left( \mathcal{\chi }_{jh}^{k}\right) $ is
defined as in (\ref{eq3.6}). To these coefficients we associate the
differential operator $\mathcal{Q}$ on $Q$ given by (\ref{eq2.52}). Finally,
we consider the boundary value problem (\ref{eq2.53})-(\ref{eq2.56}).

\begin{lemma}
\label{lem3.2} Under the hypotheses of Theorem \ref{th3.1}, the boundary
value problem (\ref{eq2.53})-(\ref{eq2.56}) admits at most one weak solution 
$\left( \mathbf{u}_{0},p_{0}\right) $ with $\mathbf{u}_{0}\in \mathcal{W}%
\left( 0,T\right) $, $p_{0}\in L^{2}\left( 0,T;L^{2}\left( \Omega ;\mathbb{R}%
\right) \mathfrak{/}\mathbb{R}\right) $.
\end{lemma}

\begin{proof}
It is an easy exercise to show that if a

\noindent couple $\left( \mathbf{u}_{0},p_{0}\right) \in \mathcal{W}\left(
0,T\right) \times L^{2}\left( 0,T;L^{2}\left( \Omega ;\mathbb{R}\right) 
\mathfrak{/}\mathbb{R}\right) $ is a solution of (\ref{eq2.53})-(\ref{eq2.56}%
), then the couple $\mathbf{u=}\left( \mathbf{u}_{0},\mathbf{u}_{1}\right) $
[in which $\mathbf{u}_{1}$ is given by (\ref{eq3.7})] satisfies (\ref{eq2.34}%
)-(\ref{eq2.35}) and is therefore unique. Hence Lemma \ref{lem3.2} follows
at once.
\end{proof}

We are now in a position to state and prove

\begin{theorem}
\label{th3.2} Let the hypotheses of Theorem \ref{th3.1} be satisfied. For
each real $0<\varepsilon <1$, let $\left( \mathbf{u}_{\varepsilon
},p_{\varepsilon }\right) \in \mathcal{W}\left( 0,T\right) \times
L^{2}\left( 0,T;L^{2}\left( \Omega ;\mathbb{R}\right) \mathfrak{/}\mathbb{R}%
\right) $ be defined by (\ref{eq1.3})-(\ref{eq1.6}). Then, as $\varepsilon
\rightarrow 0$, we have $\mathbf{u}_{\varepsilon }\rightarrow \mathbf{u}_{0}$
in $\mathcal{W}\left( 0,T\right) $-weak and $p_{\varepsilon }\rightarrow
p_{0}$ in $L^{2}\left( Q\right) $-weak, where the couple $\left( \mathbf{u}%
_{0},p_{0}\right) $ lies in $\mathcal{W}\left( 0,T\right) \times L^{2}\left(
0,T;L^{2}\left( \Omega ;\mathbb{R}\right) \mathfrak{/}\mathbb{R}\right) $
and is the unique weak solution of (\ref{eq2.53})-(\ref{eq2.56}).
\end{theorem}

\begin{proof}
As was pointed out above, from any arbitrarily given fundamental sequence $E$
one can extract a subsequence $E^{\prime }$ such that as $E^{\prime }\ni
\varepsilon \rightarrow 0$, we have (\ref{eq3.2})-(\ref{eq3.3}) and (\ref%
{eq2.43}) hence $p_{\varepsilon }\rightarrow p_{0}$ in $L^{2}\left( Q\right) 
$-weak, where $p_{0}$ is the mean of $p$ and thus $p_{0}\in L^{2}\left(
0,T;L^{2}\left( \Omega ;\mathbb{R}\right) \mathfrak{/}\mathbb{R}\right) $,
and where $\mathbf{u=}\left( \mathbf{u}_{0},\mathbf{u}_{1}\right) \in 
\mathbb{F}_{0}^{1}$. Furthermore, (\ref{eq3.5}) holds for all $\mathbf{v=}%
\left( \mathbf{v}_{0},\mathbf{v}_{1}\right) \in \mathbb{E}_{0}^{1}$.
Substituting (\ref{eq3.7}) in (\ref{eq3.5}) and then choosing therein the
particular test functions $\mathbf{v=}\left( \mathbf{v}_{0},\mathbf{v}%
_{1}\right) \in \mathbb{E}_{0}^{1}$ with $\mathbf{v}_{1}=0$ leads to Theorem %
\ref{th3.2}, thanks to Lemma \ref{lem3.2}.
\end{proof}

We can present $q_{ijkh}$ in a suitable form as in Remark \ref{rem2.2}. For
this purpose, we introduce the space $\mathcal{M}$ of all $N\times N$ matrix
functions with entries in $L^{2}\left( \Delta \left( A\right) ;\mathbb{R}%
\right) $. Specifically $\mathcal{M}$ denotes the space of $\mathbf{F=}%
\left( F^{ij}\right) _{1\leq i,j\leq N}$ with $F^{ij}\in L^{2}\left( \Delta
\left( A\right) ;\mathbb{R}\right) $. Provided with the norm 
\begin{equation*}
\left\Vert \mathbf{F}\right\Vert _{\mathcal{M}}=\left(
\sum_{i,j=1}^{N}\left\Vert F^{ij}\right\Vert _{L^{2}\left( \Delta \left(
A\right) \right) }^{2}\right) ^{\frac{1}{2}}\text{, }\mathbf{F=}\left(
F^{ij}\right) \in \mathcal{M}\text{, }
\end{equation*}%
$\mathcal{M}$ is a Hilbert space. Now, let 
\begin{equation*}
\mathcal{A}\left( \mathbf{F},\mathbf{G}\right)
=\sum_{i,j,k=1}^{N}\int_{\Delta \left( A\right) }\widehat{a}_{ij}\left(
s\right) F^{jk}\left( s,s_{0}\right) G^{ik}\left( s,s_{0}\right) d\beta
\left( s,s_{0}\right)
\end{equation*}%
for $\mathbf{F=}\left( F^{jk}\right) $ and $\mathbf{G=}\left( G^{ik}\right) $
in $\mathcal{M}$. This gives a bilinear form $\mathcal{A}$ on $\mathcal{M}%
\times \mathcal{M}$, which is symmetric, continuous and coercive.
Furthermore 
\begin{equation*}
\widehat{a}\left( \mathbf{u},\mathbf{v}\right) =\mathcal{A}\left( \widehat{%
\nabla }\mathbf{u},\widehat{\nabla }\mathbf{v}\right) \text{, }\mathbf{u},%
\mathbf{v}\in L^{2}\left( \Delta \left( A_{\tau }\right) ;H_{\#}^{1}\left(
\Delta \left( A_{y}\right) ;\mathbb{R}\right) ^{N}\right) \text{,}
\end{equation*}%
where $\widehat{\nabla }\mathbf{u=}\left( \partial _{j}u^{k}\right) $ for
any $\mathbf{u}=\left( u^{k}\right) \in L^{2}\left( \Delta \left( A_{\tau
}\right) ;H_{\#}^{1}\left( \Delta \left( A_{y}\right) ;\mathbb{R}\right)
^{N}\right) $. Now, by the same line of proceeding as followed in \cite{bib4}
(see also \cite{bib15}) one can quickly show that%
\begin{equation*}
q_{ijkh}=\mathcal{A}\left( \widehat{\nabla }\mathbf{\chi }_{ik}-\mathbf{%
\theta }_{ik},\widehat{\nabla }\mathbf{\chi }_{jh}-\mathbf{\theta }%
_{jh}\right) \text{,}
\end{equation*}%
where, for any couple of indices $1\leq i,k\leq N$, $\mathbf{\chi }_{ik}$ is
defined by (\ref{eq3.6}), and $\mathbf{\theta }_{ik}=\left( \theta
_{ik}^{lm}\right) \in \mathcal{M}$ with $\theta _{ik}^{lm}=\delta
_{il}\delta _{km}$. Having made this point, Remark \ref{rem2.2} can then be
carried over to the present setting.

\subsection{\textbf{Some concrete examples}}

In the present subsection we consider a few examples of homogenization
problems for (\ref{eq1.3})-(\ref{eq1.6}) in a concrete setting (as opposed
to the abstract assumption (\ref{eq3.1})) and we show how their study leads
to the abstract setting of Subsection 3.1 and so we may conclude by merely
applying Theorems \ref{th3.1} and \ref{th3.2}.

\begin{example}
\label{ex1} (Almost periodic homogenization). We study here the
homogenization of (\ref{eq1.3})-(\ref{eq1.6}) under the concrete hypothesis
that the family $\left( a_{ij}\right) _{1\leq i,j\leq N}$ verifies:%
\begin{equation}
a_{ij}\in L_{AP}^{2}\left( \mathbb{R}_{y}^{N}\right) \text{\quad }\left(
1\leq i,j\leq N\right) \text{,}  \label{eq3.8}
\end{equation}%
where $L_{AP}^{2}\left( \mathbb{R}_{y}^{N}\right) $ denotes the space of all
functions $w\in L_{loc}^{2}\left( \mathbb{R}_{y}^{N}\right) $ that are
almost periodic in the sense of Stepanoff (see, e.g., \cite[Section 4]{bib24}%
). According to \cite[Proposition 4.1]{bib24}, the hypothesis (\ref{eq3.8})
yields a countable subgroup $\mathcal{R}_{y}$ of $\mathbb{R}_{y}^{N}$ such
that $a_{ij}\in L_{AP,\mathcal{R}_{y}}^{2}\left( \mathbb{R}_{y}^{N}\right) $ 
$\left( 1\leq i,j\leq N\right) $, where

$L_{AP,\mathcal{R}_{y}}^{2}\left( \mathbb{R}_{y}^{N}\right) =\left\{ u\in
L_{AP}^{2}\left( \mathbb{R}_{y}^{N}\right) :Sp\left( u\right) \subset 
\mathcal{R}_{y}\right\} $, $Sp\left( u\right) $ being the spectrum of $u$,
i.e., $Sp\left( u\right) =\left\{ k\in \mathbb{R}^{N}:M\left( u\overline{%
\gamma }_{k}\right) \neq 0\right\} $ with $\gamma _{k}\left( y\right) =\exp
\left( 2i\pi k.y\right) $ $\left( y\in \mathbb{R}^{N}\right) $. Let us
consider the $H$-algebra

$A_{y}=AP_{\mathcal{R}_{y}}\left( \mathbb{R}_{y}^{N}\right) =\left\{ u\in
AP\left( \mathbb{R}_{y}^{N}\right) :Sp\left( u\right) \subset \mathcal{R}%
_{y}\right\} $, where $AP\left( \mathbb{R}_{y}^{N}\right) $ denotes the
space of almost periodic continuous complex functions on $\mathbb{R}_{y}^{N}$
(see, e.g., \cite[Chapter 5]{bib7} and \cite[Chapter 10]{bib9}). We have%
\begin{equation*}
a_{ij}\in L_{AP,\mathcal{R}_{y}}^{2}\left( \mathbb{R}_{y}^{N}\right) \subset 
\mathfrak{X}_{y}^{2}\text{ }\left( 1\leq i,j\leq N\right)
\end{equation*}%
as it can be seen by using \cite[Lemma 1]{bib16}. In view of Remark \ref%
{eq3.2}, we consider the $a_{ij}$ as functions in $\mathfrak{X}^{2}$ which
are independent of the variable $\tau $. We see that for any countable
subgoup $\mathcal{R}_{\tau }$ of $\mathbb{R}_{\tau }$, we have (\ref{eq3.1})
with $A_{y}=AP_{\mathcal{R}_{y}}\left( \mathbb{R}_{y}^{N}\right) $ and $%
A_{\tau }=AP_{\mathcal{R}_{\tau }}\left( \mathbb{R}_{\tau }\right) $ where $%
AP_{\mathcal{R}_{\tau }}\left( \mathbb{R}_{\tau }\right) =\left\{ u\in
AP\left( \mathbb{R}_{\tau }\right) :Sp\left( u\right) \subset \mathcal{R}%
_{\tau }\right\} $. Moreover, for $\mathcal{R=R}_{y}\times \mathcal{R}_{\tau
}$, $AP_{\mathcal{R}}\left( \mathbb{R}_{y}^{N}\times \mathbb{R}_{\tau
}\right) $ coincides with the closure of $AP_{\mathcal{R}_{y}}\left( \mathbb{%
R}_{y}^{N}\right) \otimes AP_{\mathcal{R}_{\tau }}\left( \mathbb{R}_{\tau
}\right) $ in $\mathcal{B}\left( \mathbb{R}_{y}^{N}\times \mathbb{R}_{\tau
}\right) $ (see \cite[Proposition 3.2]{bib17}), hence $A=$ $AP_{\mathcal{R}%
}\left( \mathbb{R}_{y}^{N}\times \mathbb{R}_{\tau }\right) $. On the other
hand, by virtue of \cite[Proposition 3.2]{bib24}, the $H$-algebra $A$ is
quasi-proper. Thus, the homogenization of (\ref{eq1.3})-(\ref{eq1.6})
follows.
\end{example}

\begin{example}
\label{ex2} Let $Y^{\prime }=\left( -\frac{1}{2},\frac{1}{2}\right) ^{N-1}$.
We study the homogenization of (\ref{eq1.3})-(\ref{eq1.6}) under the
following hypothesis:%
\begin{equation}
a_{ij}\in L^{2}\left( \mathbb{R};L_{per}^{2}\left( Y^{\prime }\right)
\right) \text{\quad }\left( 1\leq i,j\leq N\right)  \label{eq3.9}
\end{equation}%
where $L_{per}^{2}\left( Y^{\prime }\right) $ is the space of functions in $%
L_{loc}^{2}\left( \mathbb{R}^{N-1}\right) $ that are $Y^{\prime }$-periodic.
Let $\mathcal{C}_{per}\left( Y^{\prime }\right) $ denotes the space of $%
Y^{\prime }$-periodic complex continuous functions on $\mathbb{R}^{N-1}$.
Let us consider the $H$-algebra $A_{y}=\mathcal{B}_{\infty }\left( \mathbb{R}%
;\mathcal{C}_{per}\left( Y^{\prime }\right) \right) $. We recall that $%
\mathcal{B}_{\infty }\left( \mathbb{R};\mathcal{C}_{per}\left( Y^{\prime
}\right) \right) $ is the space of continuous functions $\psi :\mathbb{R}%
^{N-1}\times \mathbb{R\rightarrow C}$ such that the mapping $%
y_{N}\rightarrow \psi \left( .,y_{N}\right) $ send continuously $\mathbb{R}$
into $\mathcal{C}_{per}\left( Y^{\prime }\right) $ and $\psi \left(
.,y_{N}\right) $ has a limit in $\mathcal{C}_{per}\left( Y^{\prime }\right) $
(with the norm $\left\Vert {\small \cdot }\right\Vert _{\infty }$) as $%
\left\vert y_{N}\right\vert \rightarrow +\infty $. Under the hypothesis (\ref%
{eq3.9}), the condition (\ref{eq3.1}) is satisfied with $A=\mathcal{B}%
_{\infty }\left( \mathbb{R}_{\tau };A_{y}\right) $, where $\mathcal{B}%
_{\infty }\left( \mathbb{R}_{\tau };A_{y}\right) $ is the space of functions 
$\varphi :\mathbb{R}^{N}\times \mathbb{R\rightarrow C}$ such that the
mapping $\tau \rightarrow \varphi \left( .,\tau \right) $ send continuously $%
\mathbb{R}$ into $A_{y}$ and $\varphi \left( .,\tau \right) $ has a limit in 
$A_{y}$ (with the norm $\left\Vert {\small \cdot }\right\Vert _{\infty }$)
as $\left\vert \tau \right\vert \rightarrow +\infty $. Indeed, on one hand
the space $\mathcal{K}\left( \mathbb{R};\mathcal{C}_{per}\left( Y^{\prime
}\right) \right) $ is contained in $A_{y}=\mathcal{B}_{\infty }\left( 
\mathbb{R};\mathcal{C}_{per}\left( Y^{\prime }\right) \right) $, and $%
\mathcal{K}\left( \mathbb{R};\mathcal{C}_{per}\left( Y^{\prime }\right)
\right) $ is dense in $L^{2}\left( \mathbb{R};L_{per}^{2}\left( Y^{\prime
}\right) \right) $ as it's easily seen by\ using the fact that $\mathcal{K}%
\left( \mathbb{R}\right) $ and $\mathcal{C}_{per}\left( Y^{\prime }\right) $
are dense in $L^{2}\left( \mathbb{R}\right) $ and $L_{per}^{2}\left(
Y^{\prime }\right) $ respectively. On the other hand, let $\left( L^{2},\ell
^{\infty }\right) $ be the space of all $u\in L_{loc}^{2}\left( \mathbb{R}%
_{y}^{N}\right) $ such that 
\begin{equation*}
\left\Vert u\right\Vert _{2,\infty }=\underset{k\in \mathbb{Z}^{N}}{\sup }%
\left[ \int_{k+Y}\left\vert u\left( y\right) \right\vert ^{2}dy\right] ^{%
\frac{1}{2}}<\infty \text{,}
\end{equation*}%
where $Y=\left( -\frac{1}{2},\frac{1}{2}\right) ^{N}$. This is a Banach
space under the norm $\left\Vert .\right\Vert _{2,\infty }$. $L^{2}\left( 
\mathbb{R};L_{per}^{2}\left( Y^{\prime }\right) \right) $ is continuously
embedded in $\left( L^{2},\ell ^{\infty }\right) \left( \mathbb{R}%
^{N}\right) $ and the later is continuously embedded in $\Xi ^{2}\left( 
\mathbb{R}^{N}\right) $, thus $L^{2}\left( \mathbb{R};L_{per}^{2}\left(
Y^{\prime }\right) \right) $ is continuously embedded $\Xi ^{2}\left( 
\mathbb{R}^{N}\right) $. It follows that $L^{2}\left( \mathbb{R}%
;L_{per}^{2}\left( Y^{\prime }\right) \right) \subset \mathfrak{X}_{y}^{2}$.
Thus, by Remark \ref{rem3.2}, we see that (\ref{eq3.9}) implies (\ref{eq3.1}%
) for the H-alg\`{e}bre $A$, closure of $A_{y}\otimes \mathcal{B}_{\infty
}\left( \mathbb{R}_{\tau }\right) $ in $\mathcal{B}\left( \mathbb{R}%
_{y}^{N}\times \mathbb{R}_{\tau }\right) $ ($A_{\tau }$ is here $\mathcal{B}%
_{\infty }\left( \mathbb{R}_{\tau }\right) $, the space of continuous
functions $w:\mathbb{R}\rightarrow \mathbb{C}$ such that $w\left( \tau
\right) $ has a limit in $\mathbb{C}$ as $\left\vert \tau \right\vert
\rightarrow +\infty $). Further, $\mathcal{B}_{\infty }\left( \mathbb{R}%
_{\tau };A_{y}\right) $ coincides with the closure of $A_{y}\otimes \mathcal{%
B}_{\infty }\left( \mathbb{R}_{\tau }\right) $ in $\mathcal{B}\left( \mathbb{%
R}_{y}^{N}\times \mathbb{R}_{\tau }\right) $, hence $A=\mathcal{B}_{\infty
}\left( \mathbb{R}_{\tau };A_{y}\right) $. Moreover if we denote by\ $A_{2y}$
the closure of $\mathcal{C}_{per}\left( Y^{\prime }\right) \otimes \mathbb{C}
$ in $\mathcal{B}\left( \mathbb{R}_{y}^{N}\right) $, then the paire $\left\{
A_{y},A_{2y}\right\} $ satisfies the hypothesis (\textbf{H}) of \cite[%
Subsection 4.1]{bib25} by virtue of the proof in \cite[Corollary 4.4]{bib17}%
. But, $A_{2y}=AP_{\mathcal{R}}\left( \mathbb{R}_{y}^{N}\right) $ with $%
\mathcal{R}=\mathbb{Z}^{N-1}\times \left\{ 0\right\} $, and $A_{2}=\mathcal{B%
}_{\infty }\left( \mathbb{R}_{\tau };AP_{\mathcal{R}}\left( \mathbb{R}%
_{y}^{N}\right) \right) $ is quasi-proper (see \cite[Example 3.1]{bib24}).%
\emph{\ }Therefore, by \cite[Proposition 4.1]{bib25} we see that $A$ is
quasi-proper. Hence, the homogenization of (\ref{eq1.3})-(\ref{eq1.6})
follows by Theorems \ref{th3.1} and \ref{th3.2}.
\end{example}

\begin{example}
\label{ex3} Let us suppose that%
\begin{equation}
a_{ij}\in L_{\infty ,AP}^{2}\left( \mathbb{R}^{N}\right) \text{\quad }\left(
1\leq i,j\leq N\right) \text{,}  \label{eq3.10}
\end{equation}%
where $L_{\infty ,AP}^{2}\left( \mathbb{R}^{N}\right) $ denotes the closure
in $\left( L^{2},\ell ^{\infty }\right) \left( \mathbb{R}^{N}\right) \mathbb{%
\ }$of the space of finite sums $\sum_{finite}\varphi _{i}u_{i}$ with $%
\varphi _{i}\in \mathcal{B}_{\infty }\left( \mathbb{R}^{N}\right) $ and $%
u_{i}\in AP\left( \mathbb{R}^{N}\right) $. We have the continuous embedding
of $L^{\infty }\left( \mathbb{R}^{N}\right) $ in $\left( L^{2},\ell ^{\infty
}\right) \left( \mathbb{R}^{N}\right) $, thus 
\begin{equation*}
a_{ij}\in \left( L^{2},\ell ^{\infty }\right) \left( \mathbb{R}^{N}\right) 
\text{\qquad }\left( 1\leq i,j\leq N\right) \text{.}
\end{equation*}%
One can also see that $L_{\infty ,AP}^{2}\left( \mathbb{R}^{N}\right) $
coincides with the closure of $\mathcal{B}_{\infty }\left( \mathbb{R}%
^{N}\right) +AP\left( \mathbb{R}^{N}\right) $ in $\left( L^{2},\ell ^{\infty
}\right) \left( \mathbb{R}^{N}\right) $. Let $1\leq i,j\leq N$. By density
there exists a sequence $\left( a_{n}^{ij}\right) _{n\in \mathbb{N}}$
defined by $a_{n}^{ij}=b_{n}^{ij}+c_{n}^{ij}$ with $b_{n}^{ij}\in \mathcal{B}%
_{\infty }\left( \mathbb{R}^{N}\right) $ and $c_{n}^{ij}\in AP\left( \mathbb{%
R}^{N}\right) $ such that%
\begin{equation}
a_{n}^{ij}\rightarrow a_{ij}\text{ in }\left( L^{2},l^{\infty }\right)
\left( \mathbb{R}^{N}\right) \text{ as }n\rightarrow \infty \text{.}
\label{eq3.11}
\end{equation}%
The family $\left\{ c_{n}^{ij}:1\leq i,j\leq N\text{ and }n\in \mathbb{N}%
\right\} $ is a countable set in $AP\left( \mathbb{R}^{N}\right) $, thus by 
\cite[Proposition 5.1]{bib15} there exists a countable subgroup $\mathcal{R}$
of $\mathbb{R}^{N}$ such that 
\begin{equation*}
\left\{ c_{n}^{ij}:1\leq i,j\leq N\text{ and }n\in \mathbb{N}\right\}
\subset AP_{\mathcal{R}}\left( \mathbb{R}_{y}^{N}\right) \text{.}
\end{equation*}%
Then the sequence $\left( a_{n}^{ij}\right) _{n\in \mathbb{N}}$ lies in $%
\mathcal{B}_{\infty ,\mathcal{R}}\left( \mathbb{R}_{y}^{N}\right) $, where $%
\mathcal{B}_{\infty ,\mathcal{R}}\left( \mathbb{R}_{y}^{N}\right) $ denotes
the closure in $\mathcal{B}\left( \mathbb{R}_{y}^{N}\right) $ of $\mathcal{B}%
_{\infty }\left( \mathbb{R}_{y}^{N}\right) +AP_{\mathcal{R}}\left( \mathbb{R}%
_{y}^{N}\right) $. We denote by $L_{\infty ,\mathcal{R}}^{2}\left( \mathbb{R}%
_{y}^{N}\right) $ the closure in $\left( L^{2},\ell ^{\infty }\right) \left( 
\mathbb{R}_{y}^{N}\right) $ of $\mathcal{B}_{\infty ,\mathcal{R}}\left( 
\mathbb{R}_{y}^{N}\right) $. It follows from (\ref{eq3.11}) that 
\begin{equation}
a_{ij}\in L_{\infty ,\mathcal{R}}^{2}\left( \mathbb{R}_{y}^{N}\right) \text{ 
}\left( 1\leq i,j\leq N\right) \text{.}  \label{eq3.12}
\end{equation}%
Let $A_{y}=\mathcal{B}_{\infty ,\mathcal{R}}\left( \mathbb{R}_{y}^{N}\right) 
$ and $A_{\tau }=\mathcal{B}_{\infty }\left( \mathbb{R}_{\tau }\right) $. We
denote by $A$ the closure of $A_{y}\otimes A_{\tau }$ in $\mathcal{B}\left( 
\mathbb{R}_{y}^{N}\times \mathbb{R}_{\tau }\right) $. As $\left( L^{2},\ell
^{\infty }\right) \left( \mathbb{R}_{y}^{N}\right) $ is continuously
embedded in $\Xi ^{2}\left( \mathbb{R}^{N}\right) $, we have $L_{\infty ,%
\mathcal{R}}^{2}\left( \mathbb{R}_{y}^{N}\right) \subset \mathfrak{X}%
_{y}^{2} $. Il follows from (\ref{eq3.12}) and Remark \ref{rem3.2} that (\ref%
{eq3.1}) is satisfied for $A$. Moreover, the paire $\left\{ A_{2y}=AP_{%
\mathcal{R}}\left( \mathbb{R}_{y}^{N}\right) \text{, }A_{y}=\mathcal{B}%
_{\infty ,\mathcal{R}}\left( \mathbb{R}_{y}^{N}\right) \right\} $ satisfies
the hypothesis (\textbf{H}) of \cite[Subsection 4.1]{bib25} (see the proof
of \cite[Corollary 4.2]{bib15}). Furthermore, the H-algebra $A_{2}$ closure
of $A_{2y}\otimes A_{\tau }$ in $\mathcal{B}\left( \mathbb{R}_{y}^{N}\times 
\mathbb{R}_{\tau }\right) $ is quasi-proper (see Example \ref{ex1}). Thus,
by virtue of \cite[Proposition 4.1]{bib25} $A$ is quasi-proper. Therefore,
the homogenization of (\ref{eq1.3})-(\ref{eq1.6}) follows by Theorems \ref%
{th3.1}-\ref{th3.2}.
\end{example}

\section{Homogenization of unsteady Navier-Stokes equations in periodic
porous media}

The basic notation and hypotheses are those which are stated in Section 1,
especially in \textit{Problem} II. Throughout the present section, vector
spaces are considered over $\mathbb{R}$ and scalar functions are assumed to
take real values. Thus, for the sake of convenience, we will put $\mathcal{C}%
\left( X\right) =\mathcal{C}\left( X;\mathbb{R}\right) $, $\mathcal{B}\left(
X\right) =\mathcal{B}\left( X;\mathbb{R}\right) $, $L^{p}\left( X\right)
=L^{p}\left( X;\mathbb{R}\right) $, $H^{1}\left( X\right) =H^{1}\left( X;%
\mathbb{R}\right) $, etc., $X$ being an open set in $\mathbb{R}^{N}$.

The purpose here is to investigate the asymptotic behaviour, as $\varepsilon
\rightarrow 0$, of the solution, $\left( \mathbf{u}_{\varepsilon
},p_{\varepsilon }\right) $, of (\ref{eq1.7})-(\ref{eq1.10}) with $N=2$. As
was mentioned earlier, the hypothesis $N=2$ guarantees the unicity in (\ref%
{eq1.7})-(\ref{eq1.10}).

\subsection{\textbf{Preliminaries}}

Before we can study the asymptotic behavior of $\mathbf{u}_{\varepsilon }$
and $p_{\varepsilon }$ as $\varepsilon \rightarrow 0$, we require a few
basic results.

\begin{lemma}
\label{lem4.1} (Friedrichs inequality). There is a constant $c=c\left(
Y_{f}\right) >0$ such that 
\begin{equation}
\int_{\Omega _{\varepsilon }}\left\vert u\right\vert ^{2}dx\leq c\varepsilon
^{2}\int_{\Omega _{\varepsilon }}\left\vert \nabla u\right\vert ^{2}dx
\label{eq4.3}
\end{equation}%
for all $u\in H_{0}^{1}\left( \Omega _{\varepsilon }\right) $ and all real $%
0<\varepsilon <1$.
\end{lemma}

\begin{proof}
See \cite{bib23}
\end{proof}

Now, if $\mathbf{w=}\left( w^{k}\right) _{1\leq k\leq N}$ with $w^{k}\in
L^{p}\left( \mathcal{O}\right) $, or if $\mathbf{w=}\left( w^{ij}\right)
_{1\leq i,j\leq N}$ with $w^{ij}\in L^{p}\left( \mathcal{O}\right) $, where $%
\mathcal{O}$ is an open set in $\mathbb{R}^{N}$, we will sometimes write $%
\left\Vert \mathbf{w}\right\Vert _{L^{p}\left( \mathcal{O}\right) }$ for $%
\left\Vert \mathbf{w}\right\Vert _{L^{p}\left( \mathcal{O}\right) ^{N}}$ or
for $\left\Vert \mathbf{w}\right\Vert _{L^{p}\left( \mathcal{O}\right)
^{N\times N}}$. This abuse is convenient and in common use.

The next lemma will allow us study the behaviour of the pressure $%
p_{\varepsilon }$.

\begin{lemma}
\label{lem4.2} There exists a linear operator $R_{\varepsilon
}:H_{0}^{1}\left( \Omega \right) ^{N}\rightarrow H_{0}^{1}\left( \Omega
_{\varepsilon }\right) ^{N}$ with the following properties:

\noindent (P$_{1}$) If $\mathbf{w}\in H_{0}^{1}\left( \Omega \right) ^{N}$
and $\mathbf{w}$ is zero on $\Omega \backslash \Omega _{\varepsilon }$, then 
$R_{\varepsilon }\mathbf{w}=\mathbf{w}\mathfrak{\mid }_{\Omega _{\varepsilon
}}$.

\noindent (P$_{2}$) If $\mathbf{w}\in H_{0}^{1}\left( \Omega \right) ^{N}$
and $div\mathbf{w}=0$, then $divR_{\varepsilon }\mathbf{w}=0$.

\noindent (P$_{3}$) There is a constant $c>0$ (independent of $\mathbf{w}$
and $\varepsilon $, as well) such that%
\begin{equation}
\left\Vert R_{\varepsilon }\mathbf{w}\right\Vert _{L^{2}\left( \Omega
_{\varepsilon }\right) }\leq c\left\Vert \mathbf{w}\right\Vert _{L^{2}\left(
\Omega \right) }+c\varepsilon \left\Vert \nabla \mathbf{w}\right\Vert
_{L^{2}\left( \Omega \right) }  \label{eq4.7}
\end{equation}%
and 
\begin{equation}
\left\Vert \nabla R_{\varepsilon }\mathbf{w}\right\Vert _{L^{2}\left( \Omega
_{\varepsilon }\right) }\leq \frac{c}{\varepsilon }\left\Vert \mathbf{w}%
\right\Vert _{L^{2}\left( \Omega \right) }+c\left\Vert \nabla \mathbf{w}%
\right\Vert _{L^{2}\left( \Omega \right) }  \label{eq4.8}
\end{equation}%
for all $\mathbf{w}\in H_{0}^{1}\left( \Omega \right) ^{N}$ and all $%
0<\varepsilon <1$, where $\nabla \mathbf{w=}\left( \frac{\partial \mathbf{w}%
}{\partial x_{i}}\right) _{1\leq i\leq N}$.
\end{lemma}

\begin{proof}
See \cite{bib23}
\end{proof}

\begin{lemma}
\label{lem4.3} Assume that $\mathbf{f}$, $\mathbf{f}^{\prime }\in
L^{2}\left( 0,T;L^{2}\left( \Omega \right) ^{N}\right) $ and $\mathbf{f}%
\left( 0\right) \in L^{2}\left( \Omega \right) ^{N}$. Let $\left( \mathbf{u}%
_{\varepsilon },p_{\varepsilon }\right) $ be the solution of (\ref{eq1.7})-(%
\ref{eq1.10}). Let $\mathbf{u}_{\varepsilon }$ be identified with its
extension by zero in $\left( \Omega \backslash \Omega _{\varepsilon }\right)
\times ]0,T[$, which lies in $L^{2}\left( 0,T;H_{0}^{1}\left( \Omega \right)
^{N}\right) $. The following assertions are true.

\noindent There is a constant $C>0$ such that%
\begin{equation}
\left\Vert \nabla \mathbf{u}_{\varepsilon }\right\Vert _{L^{2}\left(
Q\right) }\leq C\varepsilon \text{, }\left\Vert \mathbf{u}_{\varepsilon
}\right\Vert _{L^{2}\left( Q\right) }\leq C\varepsilon ^{2}\text{ and }%
\left\Vert \frac{\partial \mathbf{u}_{\varepsilon }}{\partial t}\right\Vert
_{L^{2}\left( Q\right) }\leq C\text{ \quad }\left( 0<\varepsilon <1\right) 
\text{.}  \label{eq4.11}
\end{equation}

\noindent For each $0<\varepsilon <1$, there is a unique $\overline{p}%
_{\varepsilon }\in L^{2}\left( 0,T;L^{2}\left( \Omega \right) \mathfrak{/}%
\mathbb{R}\right) $ such that%
\begin{equation}
\int_{\Omega }\overline{p}_{\varepsilon }div\mathbf{w}dx=\int_{\Omega
_{\varepsilon }}p_{\varepsilon }div\left( R_{\varepsilon }\mathbf{w}\right)
dx\text{\quad }\left( \mathbf{w}\in H_{0}^{1}\left( \Omega \right)
^{N}\right)  \label{eq4.12}
\end{equation}%
for almost all $t\in \left[ 0,T\right] $.

Furthermore, 
\begin{equation}
\frac{\partial \overline{p}_{\varepsilon }}{\partial x_{i}}{\LARGE \mid }%
_{\Omega _{\varepsilon }\times ]0,T[}=\frac{\partial p_{\varepsilon }}{%
\partial x_{i}}\text{\quad }\left( 1\leq i\leq N\right)  \label{eq4.13}
\end{equation}%
in the sense of distributions on $\Omega _{\varepsilon }\times ]0,T[$,%
\begin{equation}
\left\Vert \nabla \overline{p}_{\varepsilon }\right\Vert _{L^{2}\left(
0,T;H^{-1}\left( \Omega \right) ^{N}\right) }\leq C\text{ and }\left\Vert 
\overline{p}_{\varepsilon }\right\Vert _{L^{2}\left( Q\right) }\leq C\text{
\qquad \qquad \qquad }\left( 0<\varepsilon <1\right)  \label{eq4.14}
\end{equation}%
where the constant $C>0$ is independent of $\varepsilon $.
\end{lemma}

\begin{proof}
Let us prove the inequalities in (\ref{eq4.11}). By (\ref{eq1.7}) we have 
\begin{equation}
\left( \mathbf{u}_{\varepsilon }^{\prime }\left( t\right) ,\mathbf{v}\right)
+\nu \int_{\Omega _{\varepsilon }}\nabla \mathbf{u}_{\varepsilon }\left(
t\right) {\small \cdot }\nabla \mathbf{v}dx+b\left( \mathbf{u}_{\varepsilon
}\left( t\right) ,\mathbf{u}_{\varepsilon }\left( t\right) \mathbf{,v}%
\right) =\int_{\Omega _{\varepsilon }}\mathbf{f}\left( t\right) {\small %
\cdot }\mathbf{v}dx  \label{eq4.15}
\end{equation}%
for all $\mathbf{v}\in H_{0}^{1}\left( \Omega _{\varepsilon }\right) ^{N}$
with $div\mathbf{v}=0$, where the dot stands for the usual Euclidean inner
product. Choosing the particular test function $\mathbf{v}=\mathbf{u}%
_{\varepsilon }\left( t\right) $ and noting that%
\begin{equation*}
b\left( \mathbf{u}_{\varepsilon }\left( t\right) ,\mathbf{u}_{\varepsilon
}\left( t\right) \mathbf{,u}_{\varepsilon }\left( t\right) \right) =0\text{
(see, e.g., \cite[p.163]{bib31}),}
\end{equation*}%
we get immediately%
\begin{equation*}
\frac{d}{dt}\left\Vert \mathbf{u}_{\varepsilon }\left( t\right) \right\Vert
_{L^{2}\left( \Omega _{\varepsilon }\right) }^{2}+2\nu \left\Vert \nabla 
\mathbf{u}_{\varepsilon }\left( t\right) \right\Vert _{L^{2}\left( \Omega
_{\varepsilon }\right) }^{2}=2\int_{\Omega _{\varepsilon }}\mathbf{f}\left(
t\right) {\small \cdot }\mathbf{u}_{\varepsilon }\left( t\right) dx\text{ in 
}]0,T[\text{.}
\end{equation*}%
Integrating on $\left[ 0,t\right] $ (with $t\in \left[ 0,T\right] $) the
preceding equality, we arrive at%
\begin{equation}
\left\Vert \mathbf{u}_{\varepsilon }\left( t\right) \right\Vert
_{L^{2}\left( \Omega _{\varepsilon }\right) }^{2}+2\nu \left\Vert \nabla 
\mathbf{u}_{\varepsilon }\right\Vert _{L^{2}\left( 0,T;L^{2}\left( \Omega
_{\varepsilon }\right) \right) }^{2}\leq 2\left\Vert \mathbf{f}\right\Vert
_{L^{2}\left( Q\right) }\left\Vert \mathbf{u}_{\varepsilon }\right\Vert
_{L^{2}\left( 0,T;L^{2}\left( \Omega _{\varepsilon }\right) \right) }\text{
in }]0,T[\text{.}  \label{eq4.15a}
\end{equation}%
From (\ref{eq4.15a}) we get 
\begin{equation*}
\nu \left\Vert \nabla \mathbf{u}_{\varepsilon }\right\Vert _{L^{2}\left(
0,T;L^{2}\left( \Omega _{\varepsilon }\right) \right) }^{2}\leq \left\Vert 
\mathbf{f}\right\Vert _{L^{2}\left( Q\right) }\left\Vert \mathbf{u}%
_{\varepsilon }\right\Vert _{L^{2}\left( 0,T;L^{2}\left( \Omega
_{\varepsilon }\right) \right) }\text{.}
\end{equation*}%
Therefore, using Lemma \ref{lem4.1} in the preceding inequality leads to%
\begin{equation}
\left\Vert \nabla \mathbf{u}_{\varepsilon }\right\Vert _{L^{2}\left(
0,T;L^{2}\left( \Omega _{\varepsilon }\right) \right) }\leq \frac{%
c\varepsilon }{\nu }\left\Vert \mathbf{f}\right\Vert _{L^{2}\left( Q\right) }%
\text{ and }\left\Vert \mathbf{u}_{\varepsilon }\right\Vert _{L^{2}\left(
0,T;L^{2}\left( \Omega _{\varepsilon }\right) \right) }\leq \frac{%
c^{2}\varepsilon ^{2}}{\nu }\left\Vert \mathbf{f}\right\Vert _{L^{2}\left(
Q\right) }\text{,}  \label{eq4.15b}
\end{equation}%
$c$ being the constant in (\ref{eq4.3}). On the other hand, let us
differentiate (\ref{eq4.15}) in the distribution sense on $]0,T[$. We have%
\begin{equation*}
\left( \mathbf{u}_{\varepsilon }^{\prime \prime }\left( t\right) ,\mathbf{v}%
\right) +\nu \int_{\Omega _{\varepsilon }}\nabla \mathbf{u}_{\varepsilon
}^{\prime }\left( t\right) {\small \cdot }\nabla \mathbf{v}dx+b\left( 
\mathbf{u}_{\varepsilon }^{\prime }\left( t\right) ,\mathbf{u}_{\varepsilon
}\left( t\right) \mathbf{,v}\right) +b\left( \mathbf{u}_{\varepsilon }\left(
t\right) ,\mathbf{u}_{\varepsilon }^{\prime }\left( t\right) \mathbf{,v}%
\right) =\int_{\Omega _{\varepsilon }}\mathbf{f}^{\prime }\left( t\right) 
{\small \cdot }\mathbf{v}dx
\end{equation*}%
for all $\mathbf{v}\in H_{0}^{1}\left( \Omega _{\varepsilon }\right) ^{N}$
with $div\mathbf{v}=0$ and in particular for $\mathbf{v}=\mathbf{u}%
_{\varepsilon }^{\prime }\left( t\right) $, 
\begin{equation}
\left( \mathbf{u}_{\varepsilon }^{\prime \prime }\left( t\right) ,\mathbf{u}%
_{\varepsilon }^{\prime }\left( t\right) \right) +\nu \left\Vert \nabla 
\mathbf{u}_{\varepsilon }^{\prime }\left( t\right) \right\Vert _{L^{2}\left(
\Omega _{\varepsilon }\right) }^{2}+b\left( \mathbf{u}_{\varepsilon
}^{\prime }\left( t\right) ,\mathbf{u}_{\varepsilon }\left( t\right) \mathbf{%
,u}_{\varepsilon }^{\prime }\left( t\right) \right) =\int_{\Omega
_{\varepsilon }}\mathbf{f}^{\prime }\left( t\right) {\small \cdot }\mathbf{u}%
_{\varepsilon }^{\prime }\left( t\right) dx\text{.}  \label{eq4.15c}
\end{equation}%
In fact, since $\mathbf{f}$, $\mathbf{f}^{\prime }\in L^{2}\left(
0,T;L^{2}\left( \Omega \right) ^{N}\right) $ and $\mathbf{f}\left( 0\right)
\in L^{2}\left( \Omega \right) ^{N}$, $\mathbf{u}_{\varepsilon }^{\prime }$
belongs to $L^{2}\left( 0,T;V\right) \cap L^{\infty }\left( 0,T;H\right) $
by virtue of Lemma \ref{lem2.2}, and further we recall that $b\left( \mathbf{%
u}_{\varepsilon }\left( t\right) ,\mathbf{u}_{\varepsilon }^{\prime }\left(
t\right) \mathbf{,u}_{\varepsilon }^{\prime }\left( t\right) \right) =0$.
Moreover, by (\ref{eq4.15}) we get%
\begin{equation*}
\left\Vert \mathbf{u}_{\varepsilon }^{\prime }\left( t\right) \right\Vert
_{L^{2}\left( \Omega _{\varepsilon }\right) }^{2}+\nu \int_{\Omega
_{\varepsilon }}\nabla \mathbf{u}_{\varepsilon }\left( t\right) {\small %
\cdot }\nabla \mathbf{u}_{\varepsilon }^{\prime }\left( t\right) dx+b\left( 
\mathbf{u}_{\varepsilon }\left( t\right) ,\mathbf{u}_{\varepsilon }\left(
t\right) \mathbf{,u}_{\varepsilon }^{\prime }\left( t\right) \right)
=\int_{\Omega _{\varepsilon }}\mathbf{f}\left( t\right) {\small \cdot }%
\mathbf{u}_{\varepsilon }^{\prime }\left( t\right) dx
\end{equation*}%
and in particular%
\begin{equation*}
\left\Vert \mathbf{u}_{\varepsilon }^{\prime }\left( 0\right) \right\Vert
_{L^{2}\left( \Omega _{\varepsilon }\right) }^{2}=\int_{\Omega _{\varepsilon
}}\mathbf{f}\left( 0\right) {\small \cdot }\mathbf{u}_{\varepsilon }^{\prime
}\left( 0\right) dx\text{.}
\end{equation*}%
This implies%
\begin{equation}
\left\Vert \mathbf{u}_{\varepsilon }^{\prime }\left( 0\right) \right\Vert
_{L^{2}\left( \Omega _{\varepsilon }\right) }\leq \left\Vert \mathbf{f}%
\left( 0\right) \right\Vert _{L^{2}\left( \Omega \right) }\text{.}
\label{eq4.16}
\end{equation}%
On the other hand, by (\ref{eq4.15c}) we have%
\begin{equation}
\frac{d}{dt}\left\Vert \mathbf{u}_{\varepsilon }^{\prime }\left( t\right)
\right\Vert _{L^{2}\left( \Omega _{\varepsilon }\right) }^{2}+2\nu
\left\Vert \nabla \mathbf{u}_{\varepsilon }^{\prime }\left( t\right)
\right\Vert _{L^{2}\left( \Omega _{\varepsilon }\right) }^{2}+2b\left( 
\mathbf{u}_{\varepsilon }^{\prime }\left( t\right) ,\mathbf{u}_{\varepsilon
}\left( t\right) \mathbf{,u}_{\varepsilon }^{\prime }\left( t\right) \right)
=2\int_{\Omega _{\varepsilon }}\mathbf{f}^{\prime }\left( t\right) {\small %
\cdot }\mathbf{u}_{\varepsilon }^{\prime }\left( t\right) dx\text{.}
\label{eq4.16a}
\end{equation}%
Further, by virtue of (\ref{eq2.7}) of Lemma \ref{lem2.1}, we have 
\begin{equation*}
\left\vert 2b\left( \mathbf{u}_{\varepsilon }^{\prime }\left( t\right) ,%
\mathbf{u}_{\varepsilon }\left( t\right) \mathbf{,u}_{\varepsilon }^{\prime
}\left( t\right) \right) \right\vert \leq 2^{\frac{3}{2}}\left\Vert \mathbf{u%
}_{\varepsilon }^{\prime }\left( t\right) \right\Vert _{L^{2}\left( \Omega
_{\varepsilon }\right) }\left\Vert \nabla \mathbf{u}_{\varepsilon }^{\prime
}\left( t\right) \right\Vert _{L^{2}\left( \Omega _{\varepsilon }\right)
}\left\Vert \nabla \mathbf{u}_{\varepsilon }\left( t\right) \right\Vert
_{L^{2}\left( \Omega _{\varepsilon }\right) }
\end{equation*}%
\begin{equation*}
\leq \nu \left\Vert \nabla \mathbf{u}_{\varepsilon }^{\prime }\left(
t\right) \right\Vert _{L^{2}\left( \Omega _{\varepsilon }\right) }^{2}+\frac{%
2}{\nu }\left\Vert \mathbf{u}_{\varepsilon }^{\prime }\left( t\right)
\right\Vert _{L^{2}\left( \Omega _{\varepsilon }\right) }^{2}\left\Vert
\nabla \mathbf{u}_{\varepsilon }\left( t\right) \right\Vert _{L^{2}\left(
\Omega _{\varepsilon }\right) }^{2}\text{.}
\end{equation*}%
Moreover, 
\begin{equation*}
2\left\vert \int_{\Omega _{\varepsilon }}\mathbf{f}^{\prime }\left( t\right) 
{\small \cdot }\mathbf{u}_{\varepsilon }^{\prime }\left( t\right)
dx\right\vert \leq 2\left\Vert \mathbf{f}^{\prime }\left( t\right)
\right\Vert _{L^{2}\left( \Omega \right) }\left\Vert \mathbf{u}_{\varepsilon
}^{\prime }\left( t\right) \right\Vert _{L^{2}\left( \Omega _{\varepsilon
}\right) }
\end{equation*}%
\begin{equation*}
\leq 2c\varepsilon \left\Vert \mathbf{f}^{\prime }\left( t\right)
\right\Vert _{L^{2}\left( \Omega \right) }\left\Vert \nabla \mathbf{u}%
_{\varepsilon }^{\prime }\left( t\right) \right\Vert _{L^{2}\left( \Omega
_{\varepsilon }\right) }
\end{equation*}%
\begin{equation*}
\leq \frac{c^{2}\varepsilon ^{2}}{\nu }\left\Vert \mathbf{f}^{\prime }\left(
t\right) \right\Vert _{L^{2}\left( \Omega \right) }^{2}+\nu \left\Vert
\nabla \mathbf{u}_{\varepsilon }^{\prime }\left( t\right) \right\Vert
_{L^{2}\left( \Omega _{\varepsilon }\right) }^{2}\text{,}
\end{equation*}%
where $c>0$ is the constant in (\ref{eq4.3}). Therefore, by (\ref{eq4.16a})
we get 
\begin{equation*}
\frac{d}{dt}\left\Vert \mathbf{u}_{\varepsilon }^{\prime }\left( t\right)
\right\Vert _{L^{2}\left( \Omega _{\varepsilon }\right) }^{2}-\frac{2}{\nu }%
\left\Vert \mathbf{u}_{\varepsilon }^{\prime }\left( t\right) \right\Vert
_{L^{2}\left( \Omega _{\varepsilon }\right) }^{2}\left\Vert \nabla \mathbf{u}%
_{\varepsilon }\left( t\right) \right\Vert _{L^{2}\left( \Omega
_{\varepsilon }\right) }^{2}\leq \frac{c^{2}\varepsilon ^{2}}{\nu }%
\left\Vert \mathbf{f}^{\prime }\left( t\right) \right\Vert _{L^{2}\left(
\Omega \right) }^{2}\text{.}
\end{equation*}%
Hence, 
\begin{equation*}
\frac{d}{dt}\left\{ \left\Vert \mathbf{u}_{\varepsilon }^{\prime }\left(
t\right) \right\Vert _{L^{2}\left( \Omega _{\varepsilon }\right) }^{2}\exp
\left( -\int_{0}^{t}\frac{2}{\nu }\left\Vert \nabla \mathbf{u}_{\varepsilon
}\left( s\right) \right\Vert _{L^{2}\left( \Omega _{\varepsilon }\right)
}^{2}ds\right) \right\} \leq \frac{c^{2}\varepsilon ^{2}}{\nu }\left\Vert 
\mathbf{f}^{\prime }\left( t\right) \right\Vert _{L^{2}\left( \Omega \right)
}^{2}\text{.}
\end{equation*}%
In view of (\ref{eq4.16}), integrating the preceding inequality on $\left[
0,t\right] $ (with $t\in \left[ 0,T\right] $) and using the first inequality
of (\ref{eq4.15b}), one quickly arrives at 
\begin{equation}
\left\Vert \mathbf{u}_{\varepsilon }^{\prime }\left( t\right) \right\Vert
_{L^{2}\left( \Omega _{\varepsilon }\right) }^{2}\leq \left\{ \left\Vert 
\mathbf{f}\left( 0\right) \right\Vert _{L^{2}\left( \Omega \right) }^{2}+%
\frac{c^{2}\varepsilon ^{2}}{\nu }\left\Vert \mathbf{f}^{\prime }\right\Vert
_{L^{2}\left( Q\right) }^{2}\right\} \exp \left( \frac{2Tc^{2}\varepsilon
^{2}}{\nu ^{3}}\left\Vert \mathbf{f}\right\Vert _{L^{2}\left( Q\right)
}^{2}\right) \text{.}  \label{eq4.16b}
\end{equation}%
It follows from (\ref{eq4.15b}), (\ref{eq4.16}) and (\ref{eq4.16b}) that (%
\ref{eq4.11}) holds for all $0<\varepsilon <1$ with an appropriate constant $%
C>0$.

Now, let us prove (\ref{eq4.12})-(\ref{eq4.14}). By (\ref{eq1.7}) it is
clear that%
\begin{equation}
\int_{\Omega _{\varepsilon }}\mathbf{u}_{\varepsilon }^{\prime }\left(
t\right) {\small \cdot }\mathbf{v}dx+\nu \int_{\Omega _{\varepsilon }}\nabla 
\mathbf{u}_{\varepsilon }\left( t\right) {\small \cdot }\nabla \mathbf{v}%
dx+b\left( \mathbf{u}_{\varepsilon }\left( t\right) ,\mathbf{u}_{\varepsilon
}\left( t\right) \mathbf{,v}\right) -\int_{\Omega _{\varepsilon
}}p_{\varepsilon }div\mathbf{v}dx=\int_{\Omega _{\varepsilon }}\mathbf{f}%
\left( t\right) {\small \cdot }\mathbf{v}dx  \label{eq4.17}
\end{equation}%
for almost all $t\in ]0,T[$ and for all $\mathbf{v}\in H_{0}^{1}\left(
\Omega _{\varepsilon }\right) ^{N}$. For fixed $0<\varepsilon <1$, let us
consider the mapping $F_{\varepsilon }$ of $L^{2}\left( 0,T;H_{0}^{1}\left(
\Omega \right) ^{N}\right) $ into $\mathbb{R}$ defined by 
\begin{equation}
F_{\varepsilon }\left( \mathbf{w}\right) =-\int_{0}^{T}\int_{\Omega
_{\varepsilon }}p_{\varepsilon }div\left( R_{\varepsilon }\mathbf{w}\right)
dx\text{ for all }\mathbf{w\in }L^{2}\left( 0,T;H_{0}^{1}\left( \Omega
\right) ^{N}\right) \text{,}  \label{eq4.17a}
\end{equation}%
where $R_{\varepsilon }$ is the restriction operator of Lemma \ref{lem4.2}.
It is straightforward that $F_{\varepsilon }$ is a continuous linear form on 
$L^{2}\left( 0,T;H_{0}^{1}\left( \Omega \right) ^{N}\right) $ and 
\begin{equation*}
\begin{array}{c}
\left\langle F_{\varepsilon },\mathbf{w}\right\rangle
=\int_{0}^{T}\int_{\Omega _{\varepsilon }}\mathbf{u}_{\varepsilon }^{\prime }%
{\small \cdot R}_{\varepsilon }\mathbf{w}dxdt+\nu \int_{0}^{T}\int_{\Omega
_{\varepsilon }}\nabla \mathbf{u}_{\varepsilon }{\small \cdot }\nabla
R_{\varepsilon }\mathbf{w}dxdt \\ 
+\int_{0}^{T}b\left( \mathbf{u}_{\varepsilon }\left( t\right) ,\mathbf{u}%
_{\varepsilon }\left( t\right) \mathbf{,}R_{\varepsilon }\mathbf{w}\left(
t\right) \right) dt-\int_{0}^{T}\int_{\Omega _{\varepsilon }}\mathbf{f}%
{\small \cdot }R_{\varepsilon }\mathbf{w}dxdt%
\end{array}%
\end{equation*}%
for all $\mathbf{w\in }L^{2}\left( 0,T;H_{0}^{1}\left( \Omega \right)
^{N}\right) $, in view of (\ref{eq4.17}). The aim is to estimate each of the
integrals on the right of the preceding equality. By (\ref{eq4.7}) and (\ref%
{eq4.11}), we see that%
\begin{equation*}
\left\vert \int_{0}^{T}\int_{\Omega _{\varepsilon }}\mathbf{u}_{\varepsilon
}^{\prime }{\small \cdot R}_{\varepsilon }\mathbf{w}dxdt\right\vert \leq 
\sqrt{2}Cc\left( \left\Vert \mathbf{w}\right\Vert _{L^{2}\left( Q\right)
}+\varepsilon \left\Vert \nabla \mathbf{w}\right\Vert _{L^{2}\left( Q\right)
}\right) \text{.}
\end{equation*}%
Next, combining (\ref{eq4.8}) with (\ref{eq4.11}) we get 
\begin{equation*}
\left\vert \nu \int_{0}^{T}\int_{\Omega _{\varepsilon }}\nabla \mathbf{u}%
_{\varepsilon }{\small \cdot }\nabla R_{\varepsilon }\mathbf{w}%
dxdt\right\vert \leq \sqrt{2}\nu Cc\left( \left\Vert \mathbf{w}\right\Vert
_{L^{2}\left( Q\right) }+\varepsilon \left\Vert \nabla \mathbf{w}\right\Vert
_{L^{2}\left( Q\right) }\right) \text{.}
\end{equation*}%
Further, by (\ref{eq4.7}) we get immediately%
\begin{equation*}
\left\vert \int_{0}^{T}\int_{\Omega _{\varepsilon }}\mathbf{f}{\small \cdot }%
R_{\varepsilon }\mathbf{w}dxdt\right\vert \leq \sqrt{2}c\left\Vert \mathbf{f}%
\right\Vert _{L^{2}\left( Q\right) }\left( \left\Vert \mathbf{w}\right\Vert
_{L^{2}\left( Q\right) }+\varepsilon \left\Vert \nabla \mathbf{w}\right\Vert
_{L^{2}\left( Q\right) }\right) \text{.}
\end{equation*}%
Finally, recalling that 
\begin{equation*}
b\left( \mathbf{u}_{\varepsilon }\left( t\right) ,\mathbf{u}_{\varepsilon
}\left( t\right) \mathbf{,}R_{\varepsilon }\mathbf{w}\left( t\right) \right)
=-b\left( \mathbf{u}_{\varepsilon }\left( t\right) ,R_{\varepsilon }\mathbf{w%
}\left( t\right) ,\mathbf{u}_{\varepsilon }\left( t\right) \right) \text{,}
\end{equation*}%
it follows from (\ref{eq2.7}) of Lemma \ref{lem2.1} that 
\begin{equation*}
\left\vert b\left( \mathbf{u}_{\varepsilon }\left( t\right) ,\mathbf{u}%
_{\varepsilon }\left( t\right) \mathbf{,}R_{\varepsilon }\mathbf{w}\left(
t\right) \right) \right\vert \leq \sqrt{2}\left\Vert \mathbf{u}_{\varepsilon
}\left( t\right) \right\Vert _{L^{2}\left( \Omega _{\varepsilon }\right)
}\left\Vert \nabla \mathbf{u}_{\varepsilon }\left( t\right) \right\Vert
_{L^{2}\left( \Omega _{\varepsilon }\right) }\left\Vert \nabla
R_{\varepsilon }\mathbf{w}\left( t\right) \right\Vert _{L^{2}\left( \Omega
_{\varepsilon }\right) }\text{.}
\end{equation*}%
Thus, 
\begin{equation*}
\left\vert \int_{0}^{T}b\left( \mathbf{u}_{\varepsilon }\left( t\right) ,%
\mathbf{u}_{\varepsilon }\left( t\right) \mathbf{,}R_{\varepsilon }\mathbf{w}%
\left( t\right) \right) dt\right\vert \leq \qquad \qquad \qquad \qquad
\qquad \qquad \qquad \qquad \qquad \qquad \qquad \qquad
\end{equation*}%
\begin{equation*}
\qquad \qquad \qquad \qquad \qquad \qquad \sqrt{2}\left\Vert \mathbf{u}%
_{\varepsilon }\right\Vert _{L^{\infty }\left( 0,T;L^{2}\left( \Omega
_{\varepsilon }\right) \right) }\int_{0}^{T}\left\Vert \nabla \mathbf{u}%
_{\varepsilon }\left( t\right) \right\Vert _{L^{2}\left( \Omega
_{\varepsilon }\right) }\left\Vert \nabla R_{\varepsilon }\mathbf{w}\left(
t\right) \right\Vert _{L^{2}\left( \Omega _{\varepsilon }\right) }dt\text{.}
\end{equation*}%
Using (\ref{eq4.8}) and (\ref{eq4.11}), we arrive at 
\begin{equation*}
\left\vert \int_{0}^{T}b\left( \mathbf{u}_{\varepsilon }\left( t\right) ,%
\mathbf{u}_{\varepsilon }\left( t\right) \mathbf{,}R_{\varepsilon }\mathbf{w}%
\left( t\right) \right) dt\right\vert \leq 2Cc\left\Vert \mathbf{u}%
_{\varepsilon }\right\Vert _{L^{\infty }\left( 0,T;L^{2}\left( \Omega
_{\varepsilon }\right) \right) }\left( \left\Vert \mathbf{w}\right\Vert
_{L^{2}\left( Q\right) }+\varepsilon \left\Vert \nabla \mathbf{w}\right\Vert
_{L^{2}\left( Q\right) }\right) \text{.}
\end{equation*}%
Furthermore, by (\ref{eq4.15a})%
\begin{equation}
\left\Vert \mathbf{u}_{\varepsilon }\right\Vert _{L^{\infty }\left(
0,T;L^{2}\left( \Omega _{\varepsilon }\right) \right) }^{2}\leq 2\left\Vert 
\mathbf{f}\right\Vert _{L^{2}\left( Q\right) }\left\Vert \mathbf{u}%
_{\varepsilon }\right\Vert _{L^{2}\left( 0,T;L^{2}\left( \Omega
_{\varepsilon }\right) \right) }\text{.}  \label{eq4.18}
\end{equation}%
Thus, combining (\ref{eq4.11}) and the preceding inequality, we get%
\begin{equation*}
\left\vert \int_{0}^{T}b\left( \mathbf{u}_{\varepsilon }\left( t\right) ,%
\mathbf{u}_{\varepsilon }\left( t\right) \mathbf{,}R_{\varepsilon }\mathbf{w}%
\left( t\right) \right) dt\right\vert \leq \left( 2C\right) ^{\frac{3}{2}%
}c\varepsilon \left\Vert \mathbf{f}\right\Vert _{L^{2}\left( Q\right) }^{%
\frac{1}{2}}\left( \left\Vert \mathbf{w}\right\Vert _{L^{2}\left( Q\right)
}+\varepsilon \left\Vert \nabla \mathbf{w}\right\Vert _{L^{2}\left( Q\right)
}\right) \text{.}
\end{equation*}%
On the other hand, in view of (\ref{eq4.17a}) and property ($P_{2}$) of
Lemma \ref{lem4.2}, we have%
\begin{equation}
\left\langle F_{\varepsilon },\mathbf{w}\right\rangle =0\text{ for all }%
\mathbf{w}\in L^{2}\left( 0,T;V\right) \text{.}  \label{eq4.18a}
\end{equation}%
Further, (\ref{eq4.18a}) is satisfied in particular for all $\mathbf{w}\in 
\mathcal{D}\left( 0,T;\mathcal{V}\right) $. Thus, by a classical argument
(see, e.g., \cite[pp.14-15]{bib31}) we obtain%
\begin{equation*}
F_{\varepsilon }=\nabla \overline{p}_{\varepsilon }
\end{equation*}%
where $\overline{p}_{\varepsilon }\in L^{2}\left( 0,T;L^{2}\left( \Omega
\right) \mathfrak{/}\mathbb{R}\right) $ with 
\begin{equation}
\left\Vert \overline{p}_{\varepsilon }\right\Vert _{L^{2}\left( Q\right)
}\leq c\left\Vert \nabla \overline{p}_{\varepsilon }\right\Vert
_{L^{2}\left( 0,T;H^{-1}\left( \Omega \right) ^{N}\right) }\text{.}
\label{eq4.18c}
\end{equation}%
Moreover, taking in particular $\mathbf{w}=\mathbf{\varphi }\in \mathcal{D}%
\left( 0,T;\mathcal{D}\left( \Omega _{\varepsilon }\right) ^{N}\right) $ in (%
\ref{eq4.17a}) leads to%
\begin{equation*}
\left\langle \nabla \overline{p}_{\varepsilon },\mathbf{\varphi }%
\right\rangle =\left\langle \nabla p_{\varepsilon },\mathbf{\varphi }%
\right\rangle \text{.}
\end{equation*}%
Hence (\ref{eq4.13}) follows by the arbitrariness of $\mathbf{\varphi }$. In
view of the estimates for the right hand of (\ref{eq4.17a}),%
\begin{equation}
\left\vert \left\langle \nabla \overline{p}_{\varepsilon },\mathbf{w}%
\right\rangle \right\vert \leq c_{0}\left( \left\Vert \mathbf{w}\right\Vert
_{L^{2}\left( Q\right) }+\varepsilon \left\Vert \nabla \mathbf{w}\right\Vert
_{L^{2}\left( Q\right) }\right)  \label{eq4.18b}
\end{equation}%
for all $0<\varepsilon <1$ and for all $\mathbf{w\in }L^{2}\left(
0,T;H_{0}^{1}\left( \Omega \right) ^{N}\right) $, where the constant $%
c_{0}>0 $ is independent of $\varepsilon $. We deduce that 
\begin{equation*}
\left\vert \left\langle \nabla \overline{p}_{\varepsilon },\mathbf{w}%
\right\rangle \right\vert \leq c_{0}\left\Vert \mathbf{w}\right\Vert
_{L^{2}\left( 0,T;H_{0}^{1}\left( \Omega \right) ^{N}\right) }\qquad \left( 
\mathbf{w\in }L^{2}\left( 0,T;H_{0}^{1}\left( \Omega \right) ^{N}\right)
\right)
\end{equation*}%
for all $0<\varepsilon <1$. Hence (\ref{eq4.14}) follows (use (\ref{eq4.18c}%
)) with an appropriate constant $C>0$.
\end{proof}

\subsection{\textbf{Homogenization results}}

Let%
\begin{equation*}
W\left( Y\right) =\left\{ \mathbf{w}\in H_{per}^{1}\left( Y\right) ^{N}:%
\mathbf{w}=0\text{ on }Y_{s}\text{, }div_{y}\mathbf{w}=0\right\} \text{.}
\end{equation*}%
This is a Hilbert space with the norm%
\begin{equation*}
\left\Vert \mathbf{w}\right\Vert _{W\left( Y\right) }=\left(
\int_{Y}\left\vert \nabla _{y}\mathbf{w}\right\vert ^{2}dy\right) ^{\frac{1}{%
2}}\text{ \qquad }\left( \mathbf{w}\in W\left( Y\right) \right)
\end{equation*}%
equivalent to the $H_{per}^{1}\left( Y\right) ^{N}$-norm. Now, we define%
\begin{equation*}
\mathcal{H}_{0}\left( Q;L_{per}^{2}\left( Z;W\left( Y\right) \right) \right)
=\left\{ \mathbf{u}\in L^{2}\left( Q;L_{per}^{2}\left( Z;W\left( Y\right)
\right) \right) :div\widetilde{\mathbf{u}}=0\text{, }\widetilde{\mathbf{u}}%
{\small \cdot }\mathbf{n}=0\text{ on }\partial \Omega \times ]0,T[\right\}
\end{equation*}%
where:%
\begin{equation*}
\mathbf{n}=\left( n^{i}\right) _{1\leq i\leq N}\text{ is the outward unit
normal to }\partial \Omega \text{,\qquad \qquad }
\end{equation*}%
\begin{equation*}
\widetilde{\mathbf{u}}\left( x,t\right) =\int \int_{Y\times Z}\mathbf{u}%
\left( x,t,y,\tau \right) dyd\tau \text{ }\left( \left( x,t\right) \in
Q=\Omega \times ]0,T[\right) \text{.}
\end{equation*}%
Provided with the $L^{2}\left( Q;L_{per}^{2}\left( Z;W\left( Y\right)
\right) \right) $-norm, $\mathcal{H}_{0}\left( Q;L_{per}^{2}\left( Z;W\left(
Y\right) \right) \right) $ is a Hilbert space.

We will need the family of vector functions $\left( \mathbf{\chi }%
_{j}\right) _{1\leq j\leq N}$ defined, for each fixed $1\leq j\leq N$, by
the variational problem%
\begin{equation}
\left\{ 
\begin{array}{c}
\mathbf{\chi }_{j}\in L_{per}^{2}\left( Z;W\left( Y\right) \right) :\qquad
\qquad \quad \quad \\ 
\int \int_{Y\times Z}\nabla _{y}\mathbf{\chi }_{j}{\small \cdot }\nabla _{y}%
\mathbf{w}dyd\tau =\int \int_{Y\times Z}w^{j}dyd\tau \\ 
\text{for all }\mathbf{w}=\left( w^{i}\right) \in L_{per}^{2}\left(
Z;W\left( Y\right) \right) \text{.}%
\end{array}%
\right.  \label{eq4.19}
\end{equation}

Let $\left( K_{ij}\right) _{1\leq i,j\leq N}$ be the matrix defined by%
\begin{equation*}
K_{ij}=\int \int_{Y\times Z}\mathcal{\chi }_{i}^{j}\left( y,\tau \right)
dyd\tau \text{.}
\end{equation*}

\begin{lemma}
\label{lem4.6} The matrix $\left( K_{ij}\right) $ is symmetric and positive
definite.
\end{lemma}

\begin{proof}
It follows immediately from (\ref{eq4.19}) that 
\begin{equation*}
K_{ij}=\int \int_{Y\times Z}\nabla _{y}\mathbf{\chi }_{j}{\small \cdot }%
\nabla _{y}\mathbf{\chi }_{i}dyd\tau =\int_{Z}\left( \mathbf{\chi }_{j},%
\mathbf{\chi }_{i}\right) _{W\left( Y\right) }d\tau \text{,}
\end{equation*}%
where $\left( ,\right) _{W\left( Y\right) }$ denotes the inner product
associated with the norm $\left\Vert .\right\Vert _{W\left( Y\right) }$ on $%
W\left( Y\right) $. The symmetry property of $\left( K_{ij}\right) $ follows
at once. Now, we have

\noindent $\sum_{i,j=1}^{N}K_{ij}\xi _{i}\xi _{j}=\int_{Z}\left( \mathbf{u}%
\left( \xi \right) ,\mathbf{u}\left( \xi \right) \right) _{W\left( Y\right)
}d\tau $ for all $\xi =\left( \xi _{i}\right) \in \mathbb{R}^{N}$, where $%
\mathbf{u}\left( \xi \right) =\sum_{i=1}^{N}\xi _{i}\mathbf{\chi }_{i}$. The
positivity follows. Let us check the nondegeneracy. Suppose we are given
some $\xi =\left( \xi _{i}\right) $ such that $\sum_{i,j=1}^{N}K_{ij}\xi
_{i}\xi _{j}=0$. Then, $\mathbf{u}\left( \xi \right) =0$. We deduce by (\ref%
{eq4.19}) that 
\begin{equation}
\sum_{i=1}^{N}\xi _{i}\int \int_{Y\times Z}w^{i}dyd\tau =0  \label{eq4.19a}
\end{equation}%
for all $\mathbf{w}=\left( w^{i}\right) \in L_{per}^{2}\left( Z;W\left(
Y\right) \right) $, and in particular (\ref{eq4.19a}) holds true for all $%
\mathbf{w}=1\otimes \mathbf{v}$ with $\mathbf{v}=\left( w^{i}\right) \in
W\left( Y\right) $. Choosing in (\ref{eq4.19a}) $\mathbf{w=}1\otimes \mathbf{%
v}$ with $\mathbf{v}\in W\left( Y\right) $ such that $\int_{Y}\mathbf{v}%
dy=\xi $ (see \cite[Remark 4.3]{bib23}) leads to $\xi _{j}=0$ $\left(
j=1,...,N\right) $ and so the lemma is proved.
\end{proof}

We are now able to prove the following homogenization theorem.

\begin{theorem}
\label{th4.1} Assume that the hypotheses of Lemma \ref{lem4.3} are satisfied
and $\mathbf{f}\in L^{2}\left( 0,T;H^{1}\left( \Omega \right) ^{N}\right) $.
Let $\left( \mathbf{u}_{\varepsilon },p_{\varepsilon }\right) $ be the
solution of (\ref{eq1.7})-(\ref{eq1.10}). Let $\mathbf{u}_{\varepsilon }$ be
identified with its extension by zero in $\left( \Omega \backslash \Omega
_{\varepsilon }\right) \times ]0,T[$, and let $\overline{p}_{\varepsilon }$
be defined in Lemma \ref{lem4.3}. Then, as $\varepsilon \rightarrow 0$, 
\begin{equation}
\frac{u_{\varepsilon }^{i}}{\varepsilon ^{2}}\rightarrow u_{0}^{i}\text{ in }%
L^{2}\left( Q\right) \text{-weak }\Sigma \text{ \quad }\left( 1\leq i\leq
N\right) \text{,\qquad \quad }  \label{eq4.20}
\end{equation}%
\begin{equation}
\frac{1}{\varepsilon }\frac{\partial u_{\varepsilon }^{i}}{\partial x_{j}}%
\rightarrow \frac{\partial u_{0}^{i}}{\partial y_{j}}\text{ in }L^{2}\left(
Q\right) \text{-weak }\Sigma \quad \left( 1\leq i,j\leq N\right) \text{,}
\label{eq4.21}
\end{equation}%
\begin{equation}
\overline{p}_{\varepsilon }\rightarrow p_{0}\text{ in }L^{2}\left( Q\right) 
\text{,\qquad \qquad \qquad \qquad \qquad \qquad }  \label{eq4.22}
\end{equation}%
where $\mathbf{u}_{0}=\left( u_{0}^{i}\right) $ is uniquely defined by the
variational problem%
\begin{equation}
\left\{ 
\begin{array}{c}
\mathbf{u}_{0}\in \mathcal{H}_{0}\left( Q;L_{per}^{2}\left( Z;W\left(
Y\right) \right) \right) :\qquad \qquad \\ 
\nu \int_{Q}\int \int_{Y\times Z}\nabla _{y}\mathbf{u}_{0}{\small \cdot }%
\nabla _{y}\mathbf{v}dxdtdyd\tau =\int_{Q}\mathbf{f}{\small \cdot }%
\widetilde{\mathbf{v}}dxdt \\ 
\text{for all }\mathbf{v}\in \mathcal{H}_{0}\left( Q;L_{per}^{2}\left(
Z;W\left( Y\right) \right) \right) \text{,\qquad }%
\end{array}%
\right.  \label{eq4.23}
\end{equation}%
and $p_{0}$ is the unique function in $L^{2}\left( 0,T;L^{2}\left( \Omega
\right) \mathfrak{/}\mathbb{R}\right) $ such that 
\begin{equation}
\widetilde{u}_{0}^{i}=\sum_{j=1}^{N}\frac{K_{ij}}{\nu }\left( f^{j}-\frac{%
\partial p_{0}}{\partial x_{j}}\right) \qquad \left( 1\leq i\leq N\right) 
\text{.}  \label{eq4.24}
\end{equation}
\end{theorem}

\begin{proof}
Let us first observe that the existence and unicity of $\mathbf{u}_{0}$ in (%
\ref{eq4.23}) is trivial (use, e.g., the Lax-Milgram lemma). On the other
hand, taking account of $\widetilde{\mathbf{u}}_{0}\cdot \mathbf{n}=0$ on $%
\partial \Omega \times ]0,T[$ and $div\widetilde{\mathbf{u}}_{0}=0$, we see
that if $p_{0}$ lies in $L^{2}\left( 0,T;L^{2}\left( \Omega \right) 
\mathfrak{/}\mathbb{R}\right) $ and verifies (\ref{eq4.24}), then $p_{0}$
satisfies 
\begin{equation}
\left\{ 
\begin{array}{c}
-\sum_{i,j=1}^{N}\frac{K_{ij}}{\nu }\frac{\partial ^{2}p_{0}}{\partial
x_{i}\partial x_{j}}=f\text{\quad in }\Omega \times ]0,T[\text{,}\quad \\ 
\sum_{i,j=1}^{N}\frac{K_{ij}}{\nu }\frac{\partial p_{0}}{\partial x_{j}}%
n^{i}=g\quad \text{on }\partial \Omega \times ]0,T[\text{,}%
\end{array}%
\right.  \label{eq4.25}
\end{equation}%
where: 
\begin{equation*}
f=-\sum_{i,j=1}^{N}\frac{K_{ij}}{\nu }\frac{\partial f^{j}}{\partial x_{i}}%
\text{,}
\end{equation*}%
\begin{equation*}
g=\sum_{i,j=1}^{N}\frac{K_{ij}}{\nu }\left( f^{j}\mathfrak{\mid }_{\partial
\Omega \times ]0,T[}\right) n^{i}\text{.}
\end{equation*}%
But (\ref{eq4.25}) is a Neumann type problem admitting one, and only one,
solution $p_{0}$ in $L^{2}\left( 0,T;H^{1}\left( \Omega \right) \mathfrak{/}%
\mathbb{R}\right) $.

Now, as seen earlier, the sequences $\left( \frac{\mathbf{u}_{\varepsilon }}{%
\varepsilon ^{2}}\right) _{0<\varepsilon <1}$, $\left( \overline{p}%
_{\varepsilon }\right) _{0<\varepsilon <1}$, $\left( \frac{\nabla \mathbf{u}%
_{\varepsilon }}{\varepsilon }\right) _{0<\varepsilon <1}$ and

\noindent $\left( \frac{\partial \mathbf{u}_{\varepsilon }}{\partial t}%
\right) _{0<\varepsilon <1}$ are bounded in the $L^{2}\left( Q\right) $
norm. Thus, given a fundamental sequence $E$ (i.e., $E$ is an ordinary
sequence of reals $0<\varepsilon _{n}<1$ such that $\varepsilon
_{n}\rightarrow 0$ as $n\rightarrow \infty $), by well known compactness
results (see in particular \cite{bib1}, \cite{bib11}) we can extract a
subsequence $E^{\prime }$ from $E$ such that as $E^{\prime }\ni \varepsilon
\rightarrow 0$, we have (\ref{eq4.20}), $\overline{p}_{\varepsilon
}\rightarrow p_{0}$ in $L^{2}\left( Q\right) $-weak and%
\begin{equation}
\frac{1}{\varepsilon }\frac{\partial u_{\varepsilon }^{i}}{\partial x_{j}}%
\rightarrow z_{ij}\text{ in }L^{2}\left( Q\right) \text{-weak }\Sigma \quad
\left( 1\leq i,j\leq N\right) \text{,}  \label{eq4.26}
\end{equation}%
where $u_{0}^{i}$, $z_{ij}\in L^{2}\left( Q;L_{per}^{2}\left(
Z;L_{per}^{2}\left( Y\right) \right) \right) $ and $p_{0}\in L^{2}\left(
0,T;L^{2}\left( \Omega \right) \mathfrak{/}\mathbb{R}\right) $. But based on
(\ref{eq4.18b}), if $\left( \mathbf{w}_{\varepsilon }\right) _{\varepsilon
\in E^{\prime }}$ is a sequence in $\mathcal{W}\left( 0,T\right) $ such that 
$\mathbf{w}_{\varepsilon }\rightarrow \mathbf{w}$ in $\mathcal{W}\left(
0,T\right) $-weak as $E^{\prime }\ni \varepsilon \rightarrow 0$, then%
\begin{equation}
\left\vert \left\langle \nabla \overline{p}_{\varepsilon },\mathbf{w}%
_{\varepsilon }\right\rangle -\left\langle \nabla p_{0},\mathbf{w}%
\right\rangle \right\vert \leq c\left( \left\Vert \mathbf{w}_{\varepsilon }-%
\mathbf{w}\right\Vert _{L^{2}\left( Q\right) }+\varepsilon \left\Vert \nabla 
\mathbf{w}_{\varepsilon }-\nabla \mathbf{w}\right\Vert _{L^{2}\left(
Q\right) }\right) +\left\vert \left\langle \nabla \overline{p}_{\varepsilon
}-\nabla p_{0},\mathbf{w}\right\rangle \right\vert \text{,}  \label{eq4.26a}
\end{equation}%
and we see that the right hand of (\ref{eq4.26a}) tends to zero as $%
E^{\prime }\ni \varepsilon \rightarrow 0$, since $\nabla \overline{p}%
_{\varepsilon }\rightarrow \nabla p_{0}$ in $L^{2}\left( 0,T;H^{-1}\left(
\Omega \right) ^{N}\right) $-weak and $\mathbf{w}_{\varepsilon }\rightarrow 
\mathbf{w}$ in $L^{2}\left( Q\right) ^{N}$ as $E^{\prime }\ni \varepsilon
\rightarrow 0$ ($\mathcal{W}\left( 0,T\right) $ is compactly embedded in $%
L^{2}\left( Q\right) ^{N}$). Consequently, $\nabla \overline{p}_{\varepsilon
}\rightarrow \nabla p_{0}$ in $L^{2}\left( 0,T;H^{-1}\left( \Omega \right)
^{N}\right) $-strong as $E^{\prime }\ni \varepsilon \rightarrow 0$. Thus by (%
\ref{eq4.18c}) we see that the sequence $\left( \overline{p}_{\varepsilon
}\right) _{\varepsilon \in E^{\prime }}$ actually strongly converges in $%
L^{2}\left( Q\right) $ to $p_{0}$, so that (\ref{eq4.22}) holds when $%
E^{\prime }\ni \varepsilon \rightarrow 0$.

The next point is to check that $\mathbf{u}_{0}\in L^{2}\left(
Q;L_{per}^{2}\left( Z;H_{per}^{1}\left( Y\right) ^{N}\right) \right) $ with 
\begin{equation}
\frac{\partial u_{0}^{i}}{\partial y_{j}}=z_{ij}\text{ \qquad \quad }\left(
1\leq i,j\leq N\right) \text{.}  \label{eq4.27}
\end{equation}

To do this, consider a function $\mathcal{\psi }:Q\times \mathbb{R}%
_{y}^{N}\times \mathbb{R}_{\tau }\rightarrow \mathbb{R}$ of the form%
\begin{equation}
\mathcal{\psi }\left( x,t,y,\tau \right) =\varphi \left( x,t\right) \Phi
\left( y,\tau \right) \qquad \qquad \left( \left( x,t\right) \in Q\text{, }%
\left( y,\tau \right) \in \mathbb{R}^{N}\times \mathbb{R}\right)
\label{eq4.27a}
\end{equation}%
with $\varphi \in \mathcal{D}\left( Q\right) $, $\Phi \in \mathcal{C}%
_{per}^{\infty }\left( Y\times Z\right) =\mathcal{C}_{per}\left( Y\times
Z\right) \cap \mathcal{C}^{\infty }\left( \mathbb{R}_{y}^{N}\times \mathbb{R}%
_{\tau }\right) $.

\noindent We have%
\begin{equation*}
\frac{1}{\varepsilon }\int_{Q}\frac{\partial u_{\varepsilon }^{i}}{\partial
x_{j}}\mathcal{\psi }^{\varepsilon }dxdt=-\frac{1}{\varepsilon }%
\int_{Q}u_{\varepsilon }^{i}\frac{\partial \mathcal{\psi }^{\varepsilon }}{%
\partial x_{j}}dxdt\qquad \qquad \qquad \qquad \qquad \qquad \quad
\end{equation*}%
\begin{equation*}
\qquad \qquad =-\frac{1}{\varepsilon }\int_{Q}u_{\varepsilon }^{i}\left[
\left( \frac{\partial \mathcal{\psi }}{\partial x_{j}}\right) ^{\varepsilon
}+\frac{1}{\varepsilon }\left( \frac{\partial \mathcal{\psi }}{\partial y_{j}%
}\right) ^{\varepsilon }\right] dxdt
\end{equation*}%
\begin{equation*}
\qquad \quad \qquad =-\int_{Q}\frac{u_{\varepsilon }^{i}}{\varepsilon }%
\left( \frac{\partial \mathcal{\psi }}{\partial x_{j}}\right) ^{\varepsilon
}dxdt-\int_{Q}\frac{u_{\varepsilon }^{i}}{\varepsilon ^{2}}\left( \frac{%
\partial \mathcal{\psi }}{\partial y_{j}}\right) ^{\varepsilon }dxdt\text{.}
\end{equation*}%
Letting $E^{\prime }\ni \varepsilon \rightarrow 0$ and recalling (\ref%
{eq4.20}) and (\ref{eq4.26}), we are quickly led to%
\begin{equation*}
\frac{\partial u_{0}^{i}\left( x,t,.,.\right) }{\partial y_{j}}=z_{ij}\left(
x,t,.,.\right) \text{ a.e. in }\left( x,t\right) \in Q\text{\qquad }\left(
1\leq i,j\leq N\right) \text{,}
\end{equation*}%
which shows that $\mathbf{u}_{0}$ belongs to $L^{2}\left(
Q;L_{per}^{2}\left( Z;H_{per}^{1}\left( Y\right) ^{N}\right) \right) $ with (%
\ref{eq4.27}) hence (\ref{eq4.21}) (as $E^{\prime }\ni \varepsilon
\rightarrow 0$).

Now, let us check that $\mathbf{u}_{0}\left( x,t,.,.\right) =0$ in $%
Y_{s}\times Z$ for almost all $\left( x,t\right) \in Q$. We consider $\psi $
as in (\ref{eq4.27a}) such that $\Phi =0$ in $Y_{f}\times Z$. As $E^{\prime
}\ni \varepsilon \rightarrow 0$, (\ref{eq4.20}) yields 
\begin{equation*}
0=\frac{1}{\varepsilon ^{2}}\int_{Q}u_{\varepsilon }^{i}\psi ^{\varepsilon
}dxdt\rightarrow \int_{Q}\int \int_{Y_{s}\times Z}u_{0}^{i}\left( x,t,y,\tau
\right) \varphi \left( x,t\right) \Phi \left( y,\tau \right) dxdtdyd\tau 
\text{.}
\end{equation*}%
Consequently%
\begin{equation*}
\int \int_{Y_{s}\times Z}u_{0}^{i}\left( x,t,y,\tau \right) \Phi \left(
y,\tau \right) dyd\tau \text{ a.e. in }\left( x,t\right) \in Q\text{\quad }%
\left( 1\leq i\leq N\right)
\end{equation*}%
(since $\varphi $ is arbitrary) and for all $\Phi \in \mathcal{C}%
_{per}^{\infty }\left( Y\times Z\right) $ verifying $\Phi =0$ in $%
Y_{f}\times Z$. Thus, $\mathbf{u}_{0}\left( x,t,.,.\right) =0$ in $%
Y_{s}\times Z$ a.e. in $\left( x,t\right) \in Q$. Furthermore, let $\varphi
\in \mathcal{D}\left( Q\right) $ and $w\in \mathcal{C}_{per}^{\infty }\left(
Y\times Z\right) $. According to (\ref{eq4.20}), we have as $E^{\prime }\ni
\varepsilon \rightarrow 0$,%
\begin{equation*}
\sum_{i=1}^{N}\int_{Q}\frac{u_{\varepsilon }^{i}\left( x,t\right) }{%
\varepsilon ^{2}}\varphi \left( x,t\right) \frac{\partial w}{\partial y_{i}}%
\left( \frac{x}{\varepsilon },\frac{t}{\varepsilon }\right) dxdt\rightarrow
\sum_{i=1}^{N}\int_{Q}\int \int_{Y\times Z}u_{0}^{i}\varphi \frac{\partial w%
}{\partial y_{i}}dxdtdyd\tau \text{.}
\end{equation*}%
But the lefthand side reduces to the term%
\begin{equation*}
-\sum_{i=1}^{N}\int_{Q}\frac{u_{\varepsilon }^{i}\left( x,t\right) }{%
\varepsilon }w\left( \frac{x}{\varepsilon },\frac{t}{\varepsilon }\right) 
\frac{\partial \varphi }{\partial x_{i}}\left( x,t\right) dxdt
\end{equation*}%
which goes to zero as $\varepsilon \rightarrow 0$. Hence, by the
arbitrariness of $\varphi $, 
\begin{equation*}
\int \int_{Y\times Z}div_{y}\mathbf{u}_{0}\left( x,.\right) wdyd\tau =0\text{%
\qquad \quad a.e. in }\left( x,t\right) \in Q\text{,}
\end{equation*}%
and that for any $w\in \mathcal{C}_{per}^{\infty }\left( Y\times Z\right) $.
Therefore $div_{y}\mathbf{u}_{0}\left( x,t,.,.\right) =0$ a.e. in $\left(
x,t\right) \in Q$. Hence $\mathbf{u}_{0}\in L^{2}\left( Q;L_{per}^{2}\left(
Z;W\left( Y\right) \right) \right) $.

Let us verify that $\mathbf{u}_{0}$ belongs to $\mathcal{H}_{0}\left(
Q;L_{per}^{2}\left( Z;W\left( Y\right) \right) \right) $. Clearly as $%
E^{\prime }\ni \varepsilon \rightarrow 0$,%
\begin{equation*}
\frac{\mathbf{u}_{\varepsilon }}{\varepsilon ^{2}}\rightarrow \widetilde{%
\mathbf{u}}_{0}\text{\quad in\quad }L^{2}\left( Q\right) ^{N}\text{-weak.}
\end{equation*}%
Since the operator $\mathbf{u}\rightarrow div\mathbf{u}$ sends continuously $%
L^{2}\left( 0,T;L^{2}\left( \Omega \right) ^{N}\right) $ into $L^{2}\left(
0,T;H^{-1}\left( \Omega \right) \right) $, it follows that%
\begin{equation*}
\frac{div\mathbf{u}_{\varepsilon }}{\varepsilon ^{2}}\rightarrow div%
\widetilde{\mathbf{u}}_{0}\text{\quad in\quad }L^{2}\left( 0,T;H^{-1}\left(
\Omega \right) \right) \text{-weak.}
\end{equation*}%
Hence $div\widetilde{\mathbf{u}}_{0}=0$. On the other hand, by the Stokes
formula we have 
\begin{equation*}
\sum_{i=1}^{N}\int_{Q}\frac{u_{\varepsilon }^{i}}{\varepsilon ^{2}}\theta 
\frac{\partial \varphi }{\partial x_{i}}dxdt=0
\end{equation*}%
for all $\theta \in \mathcal{D}\left( ]0,T[\right) $ and all $\varphi \in 
\mathcal{D}\left( \overline{\Omega }\right) $. Hence, on letting $E^{\prime
}\ni \varepsilon \rightarrow 0$, it follows%
\begin{equation*}
\int_{\Omega }\widetilde{\mathbf{u}}_{0}\left( x,t\right) {\small \cdot }%
\nabla \varphi \left( x\right) dx=0\text{ a.e. in }t\in ]0,T[
\end{equation*}%
(since $\theta $ is arbitrary) for all $\varphi \in \mathcal{D}\left( 
\overline{\Omega }\right) $. This shows that $\widetilde{\mathbf{u}}_{0}%
{\small \cdot }\mathbf{n}=0$ on $\partial \Omega \times ]0,T[$ (use the
Stokes formula) and so $\mathbf{u}_{0}\in \mathcal{H}_{0}\left(
Q;L_{per}^{2}\left( Z;W\left( Y\right) \right) \right) $.

Finally, if we prove that $\mathbf{u}_{0}$ satisfies the variational
equation in (\ref{eq4.23}) and $p_{0}$ satisfies (\ref{eq4.24}) then, by
virtue of the unicity in (\ref{eq4.23}) and (\ref{eq4.24}), it will follow
that (\ref{eq4.20})-(\ref{eq4.22}) hold not only as $E^{\prime }\ni
\varepsilon \rightarrow 0$ but also as $E\ni \varepsilon \rightarrow 0$ and
further as $0<\varepsilon \rightarrow 0$, owing to the arbitrariness of $E$;
and so the proof will be complete.

For this purpose, we introduce the space%
\begin{equation*}
\mathcal{H}\left( Q;L_{per}^{2}\left( Z;W\left( Y\right) \right) \right)
=\left\{ \mathbf{u}\in L^{2}\left( Q;L_{per}^{2}\left( Z;W\left( Y\right)
\right) \right) :div\widetilde{\mathbf{u}}\in L^{2}\left( Q\right) \text{, }%
\widetilde{\mathbf{u}}{\small \cdot }\mathbf{n}=0\text{ on }\partial \Omega
\times ]0,T[\right\} \text{,}
\end{equation*}%
which is a Hilbert space with the norm%
\begin{equation*}
\left\Vert \mathbf{u}\right\Vert _{\mathcal{H}\left( Q;L_{per}^{2}\left(
Z;W\left( Y\right) \right) \right) }=\left( \left\Vert \mathbf{u}\right\Vert
_{L^{2}\left( Q;L_{per}^{2}\left( Z;W\left( Y\right) \right) \right)
}^{2}+\left\Vert div\widetilde{\mathbf{u}}\right\Vert _{L^{2}\left( Q\right)
}^{2}\right) ^{\frac{1}{2}}\text{.}
\end{equation*}%
Next, let $\mathbf{\Phi }\in \mathcal{D}\left( Q\right) \otimes \mathcal{C}%
_{per}^{\infty }\left( Z\right) \otimes W\left( Y\right) $. We recall the
standard notation 
\begin{equation*}
\mathbf{\Phi }^{\varepsilon }\left( x,t\right) =\mathbf{\Phi }\left( x,t,%
\frac{x}{\varepsilon },\frac{t}{\varepsilon }\right) \text{\qquad for\quad }%
\left( x,t\right) \in Q\text{,}
\end{equation*}%
and the formulas 
\begin{equation*}
\nabla \mathbf{\Phi }^{\varepsilon }=\left( \nabla _{x}\mathbf{\Phi }\right)
^{\varepsilon }+\frac{1}{\varepsilon }\left( \nabla _{y}\mathbf{\Phi }%
\right) ^{\varepsilon }
\end{equation*}%
\begin{equation*}
div\mathbf{\Phi }^{\varepsilon }=\left( div_{x}\mathbf{\Phi }\right)
^{\varepsilon }+\frac{1}{\varepsilon }\left( div_{y}\mathbf{\Phi }\right)
^{\varepsilon }\text{.}
\end{equation*}%
Having made this point, take now in (\ref{eq4.17}) the particular test
function $\mathbf{v}=\mathbf{\Phi }^{\varepsilon }\left( t\right) $. This
yields 
\begin{equation*}
\int_{Q}\frac{\partial \mathbf{u}_{\varepsilon }}{\partial t}\mathbf{\Phi }%
^{\varepsilon }dxdt+\nu \int_{Q}\nabla \mathbf{u}_{\varepsilon }{\small %
\cdot }\left( \nabla _{x}\mathbf{\Phi }\right) ^{\varepsilon }dxdt+\nu
\int_{Q}\frac{1}{\varepsilon }\nabla \mathbf{u}_{\varepsilon }{\small \cdot }%
\left( \nabla _{y}\mathbf{\Phi }\right) ^{\varepsilon }dxdt
\end{equation*}%
\begin{equation*}
+\int_{0}^{T}b\left( \mathbf{u}_{\varepsilon }\left( t\right) ,\mathbf{u}%
_{\varepsilon }\left( t\right) ,\mathbf{\Phi }^{\varepsilon }\left( t\right)
\right) dt-\int_{Q}\overline{p}_{\varepsilon }\left( div_{x}\mathbf{\Phi }%
\right) ^{\varepsilon }dxdt=\int_{Q}\mathbf{f}{\small \cdot }\mathbf{\Phi }%
^{\varepsilon }dxdt\text{.}
\end{equation*}

The aim is to pass to the limit as $E^{\prime }\ni \varepsilon \rightarrow 0$%
. In view of the inequalities in (\ref{eq4.11}), we clearly have%
\begin{equation*}
\int_{Q}\nabla \mathbf{u}_{\varepsilon }{\small \cdot }\left( \nabla _{x}%
\mathbf{\Phi }\right) ^{\varepsilon }dxdt\rightarrow 0\text{ and }\int_{Q}%
\frac{\partial \mathbf{u}_{\varepsilon }}{\partial t}\mathbf{\Phi }%
^{\varepsilon }dxdt\rightarrow 0
\end{equation*}%
as $0<\varepsilon \rightarrow 0$. On the other hand, using the equality 
\begin{equation*}
b\left( \mathbf{u}_{\varepsilon }\left( t\right) ,\mathbf{u}_{\varepsilon
}\left( t\right) ,\mathbf{\Phi }^{\varepsilon }\left( t\right) \right)
=-b\left( \mathbf{u}_{\varepsilon }\left( t\right) ,\mathbf{\Phi }%
^{\varepsilon }\left( t\right) ,\mathbf{u}_{\varepsilon }\left( t\right)
\right)
\end{equation*}%
and (\ref{eq2.7}) (of Lemma \ref{lem2.1}), we have 
\begin{equation*}
\left\vert \int_{0}^{T}b\left( \mathbf{u}_{\varepsilon }\left( t\right) ,%
\mathbf{u}_{\varepsilon }\left( t\right) ,\mathbf{\Phi }^{\varepsilon
}\left( t\right) \right) dt\right\vert \leq 2^{\frac{1}{2}}\left\Vert 
\mathbf{u}_{\varepsilon }\right\Vert _{L^{\infty }\left( 0,T;L^{2}\left(
\Omega \right) \right) }\int_{0}^{T}\left\Vert \nabla \mathbf{u}%
_{\varepsilon }\left( t\right) \right\Vert _{L^{2}\left( \Omega \right)
}\left\Vert \nabla \mathbf{\Phi }^{\varepsilon }\left( t\right) \right\Vert
_{L^{2}\left( \Omega \right) }dt
\end{equation*}%
\begin{equation*}
\qquad \qquad \qquad \qquad \leq 2^{\frac{1}{2}}\left\Vert \mathbf{u}%
_{\varepsilon }\right\Vert _{L^{\infty }\left( 0,T;L^{2}\left( \Omega
\right) \right) }\left\Vert \nabla \mathbf{u}_{\varepsilon }\right\Vert
_{L^{2}\left( Q\right) }\left\Vert \nabla \mathbf{\Phi }^{\varepsilon
}\right\Vert _{L^{2}\left( Q\right) }\text{.}
\end{equation*}%
Hence, by (\ref{eq4.11}) and (\ref{eq4.18}) we get 
\begin{equation*}
\left\vert \int_{0}^{T}b\left( \mathbf{u}_{\varepsilon }\left( t\right) ,%
\mathbf{u}_{\varepsilon }\left( t\right) ,\mathbf{\Phi }^{\varepsilon
}\left( t\right) \right) dt\right\vert \leq K\varepsilon ^{2}\left\Vert
\nabla \mathbf{\Phi }^{\varepsilon }\right\Vert _{L^{2}\left( Q\right) }
\end{equation*}%
where $K>0$ is a constant independent of $\varepsilon $. But%
\begin{equation*}
\sup_{0<\varepsilon <1}\varepsilon \left\Vert \nabla \mathbf{\Phi }%
^{\varepsilon }\right\Vert _{L^{2}\left( Q\right) }<\infty \text{.}
\end{equation*}%
Therefore, as $0<\varepsilon \rightarrow 0$%
\begin{equation*}
\int_{0}^{T}b\left( \mathbf{u}_{\varepsilon }\left( t\right) ,\mathbf{u}%
_{\varepsilon }\left( t\right) ,\mathbf{\Phi }^{\varepsilon }\left( t\right)
\right) dt\rightarrow 0\text{.}
\end{equation*}%
Finally, by (\ref{eq4.20})-(\ref{eq4.22}), one quickly arrives at 
\begin{equation}
\nu \int_{Q}\int \int_{Y\times Z}\nabla _{y}\mathbf{u}_{0}{\small \cdot }%
\nabla _{y}\mathbf{\Phi }dxdtdyd\tau -\int_{Q}p_{0}div\widetilde{\mathbf{%
\Phi }}dxdt=\int_{Q}\mathbf{f}{\small \cdot }\widetilde{\mathbf{\Phi }}dxdt
\label{eq4.28}
\end{equation}%
and that for any $\mathbf{\Phi }\in \mathcal{D}\left( Q\right) \otimes 
\mathcal{C}_{per}^{\infty }\left( Z\right) \otimes W\left( Y\right) $.
Recalling that $\mathcal{D}\left( Q\right) \otimes \mathcal{C}_{per}^{\infty
}\left( Z\right) \otimes W\left( Y\right) $ is dense in $\mathcal{H}\left(
Q;L_{per}^{2}\left( Z;W\left( Y\right) \right) \right) $ (see \cite{bib20}),
(\ref{eq4.28}) holds for all

\noindent $\mathbf{\Phi }\in \mathcal{H}\left( Q;L_{per}^{2}\left( Z;W\left(
Y\right) \right) \right) $. Therefore, the variational equation in (\ref%
{eq4.20}) follows at once by taking in particular $\mathbf{\Phi }=\mathbf{v}$
with $\mathbf{v}\in \mathcal{H}_{0}\left( Q;L_{per}^{2}\left( Z;W\left(
Y\right) \right) \right) $. Thus, the proof is complete once (\ref{eq4.24})
is established. To achieve this, let $1\leq j\leq N$ be arbitrary fixed.
Take in (\ref{eq4.28}) $\mathbf{\Phi }=\varphi \otimes \mathbf{\chi }_{j}$,
where $\mathbf{\chi }_{j}$ is defined by (\ref{eq4.19}) and $\varphi $ is
freely fixed in $\mathcal{D}\left( Q\right) $. Then it is an easy matter to
arrive at 
\begin{equation*}
\nu \int \int_{Y\times Z}\nabla _{y}\mathbf{u}_{0}\left( x,t\right) {\small %
\cdot }\nabla _{y}\mathbf{\chi }_{j}dy+\sum_{i=1}^{N}K_{ij}\left( \frac{%
\partial p_{0}}{\partial x_{i}}\left( x,t\right) -f^{i}\left( x,t\right)
\right) =0
\end{equation*}%
a.e. in $\left( x,t\right) \in Q$. But 
\begin{equation*}
\int \int_{Y\times Z}\nabla _{y}\mathbf{u}_{0}\left( x,t\right) {\small %
\cdot }\nabla _{y}\mathbf{\chi }_{j}dy=\widetilde{u}_{0}^{j}\left(
x,t\right) 
\end{equation*}%
as is immediate by taking in (\ref{eq4.19}) $\mathbf{w}=\mathbf{u}_{0}\left(
x,t\right) $ for fixed $\left( x,t\right) \in Q$. Hence (\ref{eq4.24})
follows. The theorem is proved.
\end{proof}

\noindent \textbf{Conclusion. }In our study, we have been limited in spatial
dimension $N=2$. It would be interesting to investigate the case $N=3$,
costumary used in physics. Unfurtunately, we come up against the lack of
uniqueness for Non-stationary Navier-Stokes equations in dimension $N\geq 3$%
. Moreover, for flows in porous media, we have been interested uniquely for
the periodic case, the problem beyond the periodic setting being not only to
be formulated mathematically, but to be justified by physics.

However, one convergence theorem has been proved for each problem, and we
have derived the macroscopic homogenized model.


\begin{thebibliography}{10}
\bibitem[1]{bib1} G. Allaire, Homogenization and two-scale convergence, SIAM
J. Math. Anal., \textbf{23 }(1992), 1482-1518.

\bibitem[2]{bib2} G. Allaire, Homogenization of the Stokes flow in a
connected porous medium, Asymptot. Anal., \textbf{2} (1989), 203-222.

\bibitem[3]{bib3} G. Allaire, Homogenization of the Navier-Stokes equations
in open sets perforated with tiny holes I-II, Arch. Rat. Mech. Anal., 
\textbf{113} (1991), 261-298.

\bibitem[4]{bib4} A. Bensoussan, J.L. Lions, G. Papanicolaou, Asymptotic
analysis for periodic structures, North-Holland, Amsterdam,1978.

\bibitem[5]{bib5} H.J. Choe, H. Kim, Homogenization of the non-stationary
Stokes equations with periodic viscosity, J. Korean Math. Soc. 46 (2009), No
5, pp. 1041-1069.

\bibitem[6]{bib6} C. Conca, On the Application of Homogenization theory to a
class of problems arising in Fluids Mechanics, J. Math. pures et appl. 64,
1987, p. 31 \`{a} 75.

\bibitem[7]{bib7} A. Guichardet, Analyse Harmonique commutative, Dunod,
Paris, 1968.

\bibitem[8]{bib8} J.B. Keller, Darcy's law for flow in porous media and the
two-space method, Lect. Notes in Pure and Appl. Math., 54, Dekker, New York,
1980.

\bibitem[9]{bib9} R. Larsen, Banach algebras, Marcel Dekker, New York, 1973.

\bibitem[10]{bib10} J.L. Lions, Quelques m\'{e}thodes de r\'{e}solution des
probl\`{e}mes aux limites non lin\'{e}aires, Dunod, Paris, 1969.

\bibitem[11]{bib11} D. Lukkassen, G. Nguetseng, P. Wall, Two scale
convergence, International journal of Pure and Applied Mathematics. Vol.2,
No.1, 2002, 35-86.

\bibitem[12]{bib12} A. Mikelic and L. Paoli, Homogenization of the inviscid
incompressible fluid flow through a 2D porous medium, Proc. Amer. Math.
Soc., 127, No7 (1999), 2019-2028.

\bibitem[13]{bib13} A.K. Nandakumaran, Steady and evolution Stokes equations
in porous media with nonhomogeneous boundary data: a homogenization process,
Diff. Int. Equ. \textbf{5 (1) }(1992)73-93.

\bibitem[14]{bib14} G. Nguetseng, A general convergence result for a
functional related to the theory of homogenization, SIAM J. Math. Anal. 
\textbf{20 }(1989), 608-623.

\bibitem[15]{bib15} G. Nguetseng, Almost periodic homogenization: Asymptotic
analysis of a second order elliptic equation, Department of Mathematics,
University of Yaounde 1, July 2000.

\bibitem[16]{bib16} G. Nguetseng, Sigma-convergence of parabolic
differential operators, Multiple scales problems in biomathematics,
mechanics, physics and numerics, Gakuto inter. ser., Mathematical Sciences
and Applications 31, 93-132 (2009)\textit{.}

\bibitem[17]{bib17} G. Nguetseng, Homogenization Structures and applications
I, Zeit. \ Anal. \ Anwend. \textbf{22} (2003) 73-107.

\bibitem[18]{bib18} G. Nguetseng, Homogenization Structures and applications
II, Z. Anal. Anwendungen \textbf{23} (2004) 482-508.

\bibitem[19]{bib19} G. Nguetseng, Deterministic homogenization of a
semilinear elliptic partial differential equation of order 2$m$, Maths.
Reports, \textbf{8} \textbf{(58)}, No2 (2006), 167-195.

\bibitem[20]{bib20} G. Nguetseng, Asymptotic analysis for a stiff
variational problem arising in mechanics, SIAM J. Math. Anal. \textbf{21},
No6 (1990), 1394-1414.

\bibitem[21]{bib21} G. Nguetseng, H. Nnang, Homogenization of nonlinear
monotone operators beyond the periodic setting. Electron. J. Diff. Eqns.
Vol. 2003 (2003), pp. 1-24.

\bibitem[22]{bib22} G. Nguetseng, L. Signing, Sigma-convergence of
stationary Navier-Stokes type equations. Electronic Journal of Differential
Equations, Vol. \textbf{2009} (2009), No. 74, pp. 1-18.

\bibitem[23]{bib23} G. Nguetseng, L. Signing, Deterministic homogenization
of stationary Navier-Stokes type equations, Multiscale Problems, Theory,
Numerical Approximation and Applications (2011): pp.70-108.

\bibitem[24]{bib24} G. Nguetseng, J.L. Woukeng, Deterministic homogenization
of parabolic monotone operators with time dependent coefficients. Electronic
Journal of Differential Equations, Vol. \textbf{2004} (2004), No. 82, pp.
1-23.

\bibitem[25]{bib25} G. Nguetseng, J.L. Woukeng, $\Sigma $-convergence of
nonlinear parabolic operators, Nonlinear Analysis \textbf{66} (2007)
968-1004.

\bibitem[26]{bib26} E. Sanchez-Palencia, Nonhomogeneous Media and Vibration
Theory, Lecture Notes in Physics, 127, Springer-Verlag, Berlin, 1980.

\bibitem[27]{bib27} L. Signing, Two-scale convergence of unsteady Stokes
type equations, ArXiv: 1101.2781v1, math.AP, 14 Jan 2011.

\bibitem[28]{bib28} L. Tartar, Incompressible fluid flow in Porous Medium,
Convergence of the homogenization process, Appendix to the Lecture Notes in
Physics 127, Springer-Verlag, 1980.

\bibitem[29]{bib29} L. Tartar, Probl\`{e}mes d'homog\'{e}n\'{e}isation dans
les \'{e}quations aux d\'{e}riv\'{e}es partielles, Cours Peccot, Coll\`{e}ge
de France, 1977.

\bibitem[30]{bib30} L. Tartar, Topics in Nonlinear Analysis, Publications
math\'{e}matiques d'Orsay 78.13, Universit\'{e} de Paris-Sud (1978).

\bibitem[31]{bib31} R. Temam, Navier-Stokes Equations, North-Holland, 1979.

\bibitem[32]{bib32} V.V. Zhikov, On two-scale convergence, J. Math. Sc., 
\textbf{120}, No3 (2004), 1328-1352.
\end{thebibliography}
\end{document}